\documentclass[a4paper,11pt, reqno]{amsart}  
\usepackage{amssymb, amsmath, amsthm, graphicx, color, epsfig, 
amsfonts, tikz, url}

\usepackage[mathscr]{euscript}

\usepackage{thmtools}

\numberwithin{equation}{section}

\newcommand{\mc}[1]{{\mathcal #1}}
\newcommand{\mb}[1]{{\mathbf #1}}
\newcommand{\mf}[1]{{\mathfrak #1}}
\newcommand{\bs}[1]{{\boldsymbol #1}}
\newcommand{\bb}[1]{{\mathbb #1}}
\newcommand{\ms}[1]{{\mathscr #1}}

\theoremstyle{plain}
\newtheorem{theorem}{Theorem}[section]
\newtheorem{proposition}[theorem]{Proposition}
\newtheorem{lemma}[theorem]{Lemma}
\newtheorem{corollary}[theorem]{Corollary}

\theoremstyle{definition}

\theoremstyle{remark}
\newtheorem{remark}[theorem]{Remark}

\newcommand{\<}{\langle}
\renewcommand{\>}{\rangle}

\renewcommand{\t}{\tau}

\newcommand{\De}{\Delta}
\newcommand{\Ga}{\Gamma}

\newcommand{\lan}{\langle}
\newcommand{\ran}{\rangle}

\def\R{{\mathbb{R}}}

\def\Z{{\mathbb{Z}}}
\def\T{{\mathbb{T}}}

\def\ux{\underline{x}}

\allowdisplaybreaks

\newcommand{\vertiii}[1]{{\left\vert\kern-0.25ex\left\vert\kern-0.25ex\left\vert #1 
    \right\vert\kern-0.25ex\right\vert\kern-0.25ex\right\vert}}

\title{Linear fluctuation of interfaces in Glauber-Kawasaki dynamics}
\author{Tadahisa Funaki, Claudio Landim and Sunder Sethuraman}
\date{\today}

\begin{document}
\maketitle

\begin{abstract}
In this article, we find a scaling limit of the space-time mass
fluctuation field of Glauber + Kawasaki particle dynamics around its
hydrodynamic mean curvature interface limit. Here, the Glauber rates
are scaled by $K=K_N$, the Kawasaki rates by $N^2$ and space by $1/N$.
We start the process so that the interface $\Ga_t$ formed is
stationary that is, $\Ga_t$ is `flat'. When the Glauber rates are
balanced on $\T^d$, $\Ga_t=\Ga=\{x: x_1=0\}$ is immobile and the
hydrodynamic limit is given by $\rho(t,v) = \rho_+$ for
$v_1\in (0,1/2)$ and $\rho(t,v)= \rho_-$ for $v_1\in (-1/2,0)$ for all
$t\ge 0$, where $v=(v_1,\ldots,v_d)\in \T^d$ identified with
$[-1/2,1/2)^d$. Since in the formation the boundary region about the
interface has width $O(1/\sqrt{K_N})$, we will scale the $v_1$
coordinate in the fluctuation field by $\sqrt{K_N}$ so that the
scaling limit will capture information `near' the interface.

We identify the fluctuation limit as a Gaussian field when
$K_N\uparrow \infty$ and $K_N= O(\sqrt{\log(N)})$ in $d\leq 2$. In the
one dimensional case, the field limit is given by ${\bf e}(v_1) B_t$
where $B_t$ is a Brownian motion and ${\bf e}$ is the normalized
derivative of a decreasing `standing wave' solution $\phi$ of
$\partial^2_{v_1} \phi - V'(\phi)=0$ on $\R$, where $V'$ is the
homogenization of the Glauber rates. In two dimensions, the limit is
${\bf e}(v_1)Z_t(v_2)$ where $Z_t$ is the solution of a one
dimensional stochastic heat equation. The appearance of the function
${\bf e}(\cdot)$ in the limit field indicates that the interface
fluctuation retains the shape of the transition layer $\phi$.
\end{abstract}

\section{Introduction}
\label{sec1}

Consider reaction $+$ diffusion interacting particle systems on
$\Z^d$, such as `Glauber $+$ Kawasaki' models. Informally, particles
on $\Z^d$, with edge scaling of $1/N$, exchange neighbor places with
rate $N^2$ (the `Kawasaki' part), and are born or expire with rates
depending on local configurations of order $K=K_N$ (the `Glauber'
part). When $K_N\uparrow\infty$, such models have been used to study
the phenomenon of phase segregation \cite{deMasi}.

For instance, when the Glauber rates correspond to the gradient of a
`double well' potential, the continuum space-time hydrodynamic limit
of the mass describes mean-curvature of an interface between two
regions of different constant densities (the argmins of the double
well potential) \cite{Bonaventura}, \cite{KS94}, \cite{FT}. Related
hydrodynamic limits have been understood also in more general reaction
$+$ diffusion systems \cite{kfhps}, \cite{FMSTspeedchange}. In this
context, the aim of this article is to derive the associated continuum
fluctuations, certain SPDEs, when the particle system is started from
a certain `flat' nonequilibrium state.

From a broader perspective, the emergence of phase separating
interface continuum flows is a concern of long standing interest; see
\cite{deMasi}, \cite{Spohn} in the statistical physics literature, and
\cite{Belletini} and references therein for results in the PDE
literature. There are also works which investigate the `mesoscopic'
fluctuations of the interface through SPDE scaling limits; see
\cite{Funakibook} and references therein. However, less is known or
conjectured about the continuum mass fluctuations, with respect to the
hydrodynamics, as a limit from microscopic systems.

This paper may be seen as a companion to \cite{F1}, which considers
the `mesoscopic' scaling limit of an SPDE (equation (2.3) in
\cite{F1}), written as a formal expansion of the stochastic
differential for the mass empirical measure in a Glauber + Kawasaki
particle system. Our purpose is to derive rigorously the linear form
of this SPDE directly from Glauber + Kawasaki microscopic
interactions, supplying the homogenizations required, starting from a
certain `stationary' initial condition on the discrete torus
$\T_N^d= (\Z/ N\Z)^d$ for $d\leq 2$. This SPDE limit
captures microscopic-to-macroscopic fluctuations of the mass near and
away from the associated continuum interface in a certain regime.

To describe the results, let us recall the intuition for the
mean-curvature limit \cite{FT, F24} say; see also the `Motion by mean
curvature' discussion below.  Denote by
$\eta(t) \in \{0,1\}^{\bb T^d_N}$ the state of the process at time
$t$, so that $\eta_x(t)=1$ if site $x\in \T_N^d$ is occupied at time
$t$ and $\eta_x(t)=0$ if it is empty. By looking at stochastic
differentials, one expects that the random configuration
$\eta(t)=\{\eta_x(t)\}$ of particles on $\T^d_N$ is close to
deterministic $u^N(t)=\{u^N_x(t)\}$ values where
$\partial_t u^N_x(t) = \Delta^N u^N_x(t) - K_NV'(u^N_x(t))$ is a
discretized form of the hydrodynamic equation \eqref{hyd equation}
with $K=K_N$. Starting from a distribution $\mu^N(0)$ which identifies
an initial continuum interface $\Gamma_0$ between two phases, the
relative entropy $H(\mu^N(t); \nu^N_t)$ between the distribution
$\mu^N(t)$ of $\eta(t)$ and a product measure $\nu^N_t$ whose
marginals have means $u^N(t)$ is bounded as $O(N^{d-\epsilon})$. As a
consequence, the empirical measure $\pi^N_t$ of the particles is close
to that of the values $u^N(t)$. Then, PDE estimates are used to show
that the deterministic values $u^N(t)$ converge to the mean-curvature
limit \eqref{meancurvaturelimit}. Here $K_N\uparrow\infty$ and its
divergence is restricted to $K_N \leq \delta \log(N)$ to show a form
of homogenization or `Boltzmann-Gibbs principle', needed to deduce the
relative entropy bound.

In the setup to analyze fluctuations, we will start from initial
conditions close to a stationary profile in one direction and constant
in others. To describe them consider $u_0:\T^d \rightarrow \R$ where
$u_0(v_1, v')=u_0(v_1, v'')$ for $v', v'' \in \T^{d-1}$ and
$\partial^2_{v_1} u_0(v) - K_N V'(u_0(v)) = 0$. The Glauber rates will
be chosen so that its homogenization $V'$ has exactly three zeroes,
$\rho_-<\rho_*<\rho_+$. We will take $u_0$ as the (unique) solution
passing through $\rho_*$ exactly twice, $u_0(0)=\rho_*$,
$\partial_{v_1}u_0(0)<0$, and whose extreme values are bounded by and
close to $\rho_\pm$. Then, the level set of the profile $u_0$ on the
continuum torus $\T^d$ with respect to $\rho_*$ corresponds to an
interface with two `flat' components which remain fixed under
mean-curvature flow. In this formulation, there are boundary regions
of width $O(1/\sqrt{K_N})$ about the interface where $u_0$ transitions
from values close to $\rho_\pm$ to those close to $\rho_\mp$.

To capture the behavior near the interface, we form
$\rho^K\colon \sqrt{K}\T\rightarrow \R$ where
$\rho^K(\vartheta) = u_0(\vartheta/\sqrt{K}, \cdot, \ldots, \cdot)$.
Note that $\rho^K$ has period $\sqrt{K}$ and
$\partial^2_{\vartheta} \rho^K - V'(\rho^K)=0$. Also, $\rho^K$
converges to the decreasing `standing wave' $\phi$ on finite intervals
of $\R$, which solves $\partial^2_{\vartheta} \phi - V'(\phi) =0$ with
$\phi(\pm \infty) = \rho_\mp$ and $\phi(0)=\rho_*$. Now, we will set
initially $E[\eta_x(0)] = {\bf u}^N(x)$ where
${\bf u}^N(x)=\rho^K(x_1 \sqrt{K}/N)$ for $x\in \T^d_N$. In this way,
at a macroscopic point $a=x/N$ for $x\in \T^d_N$, we have
${\bf u}^N(x)=\rho^K(x_1\sqrt{K}/N)= \rho^K(a_1\sqrt{K}) \rightarrow
\rho_+{\bf 1}(a_1<0) + \rho_-{\bf 1}(a_1>0)$, as
$K=K_N\uparrow\infty$, consistent with the hydrodynamics picture given
above.

We will focus on the fluctuations of the mass in one of the `layers',
that is on the part of the torus $\T^d_N$ corresponding to the region
where ${\bf u}^N$ transitions from values near $\rho_-$ to $\rho_+$;
the analysis of the other `layer' is similar. Let
$Ax = (x_1\sqrt{K}/N, x_2/N,\ldots, x_d/N)$ for $x\in \T^d_N$.
Typically, the fluctuation field
$\sum_{x\in \T^d} (\eta_x(t) - {\bf u}^N(x)) \delta_{Ax}$ should be
scaled by $N^{d/2}$ to obtain nontrivial limits. However, the
martingale, introduced by the Glauber rates, in the stochastic
differential of such a scaled field is of order $O(\sqrt{K})$, whereas
the one from the Kawasaki rates is of order $O(1)$. Hence, we will
consider the field divided by $N^{d/2}K^{1/4}$ denoted as $X^N_t$; see
\eqref{2-16}.

To be able to analyze the limit of $X^N_t$, we (1) obtain a suitable
estimate of the relative entropy between the process distribution
$\mu^N(t)$ at time $t$ with $\nu^N_0$, now a product measure with
means $\{{\bf u}^N(x): x\in \T^d\}$, then (2) make use of this bound
to perform a homogenization or Boltzmann-Gibbs principle of nonlinear
microscopic rates (see the next paragraph), and finally (3) identify
limits of the scaled fluctuation field via its stochastic
differential. We observe that $\nu^N_0$ is not the stationary
distribution of the Glauber + Kawasaki dynamics $\eta(t)$ on $\T_N^d$,
and so in this way, we will identify here a nonequilibrium fluctuation
limit. The unique stationary distribution is called the steady state
and its concrete form is unknown.

To this end, in Theorem \ref{mt1} and Corollary \ref{m-cor}, we show
$H(\mu^N_t; \nu^N_0) = O(N^{\alpha_d} e^{CtK_N^2})$ where
$\alpha_1 = 1/5$, $\alpha_2 = 4/3$ and $\alpha_d = d-(4/3)$ in
$d\geq 3$, when $K_N\uparrow\infty$ and $K_N = O(\ell_N^{d/2})$ with
$\ell_N=N^{4/5}$ in $d=1$, $\ell_N \log \ell_N = N^2$ in $d=2$, and
$\ell_N=N^{4/(3d)}$ in $d\geq 3$. These estimates allow homogenization
of the microscopic rates in the stochastic differential of $X^N_t$ in
dimensions $d=1,2$, but not in $d\geq 3$. Moreover, in the
homogenization the divergence of $K_N$ is further limited to
$K_N \leq \delta \sqrt{\log(N)}$, see \eqref{2-89},  for a small constant $\delta>0$,
instead of $K_N \leq \delta \log(N)$. Finally, we identify the limit
of $X^N_t$ in Theorem \ref{mt2} as a Gaussian field $X_t$ acting on
compactly supported smooth functions
$F:\R\times \T^{d-1}\rightarrow \R$ where
$$X_t(F) = \left\{\begin{array}{rl}
\sqrt{\varpi}\langle F, \bs e\rangle B_t & \ {\rm in \ }d=1\\
\langle F, \bs e(\vartheta)Z_t(\theta)\rangle & \ {\rm in \ }d=2.
\end{array}\right.
$$
Here,
$\partial_t Z_t = \partial^2_\theta Z_t + \sqrt{\varpi}\xi(t,\theta)$
with $Z_0=0$ and $\xi$ is a space-time white noise on $\R_+\times \T$.
Also, $\sqrt{\varpi}$ is a certain constant (below \eqref{2-78}) and
$\bs e = \phi'/\|\phi\|_{L^2(\R)}$. In these evaluations, $X_t$ is a
random function in $d\leq 2$, whereas any derivations of $X_t$ in
$d\geq 3$ would be as a generalized distribution (see equations
(4.5)--(4.7) in \cite{F1}).

Since $E[\eta_x(t)] \sim \phi(x_1\sqrt{K}/N)\sim \rho_\pm$ depending
on whether $x_1<0$ or $x_1>0$, the field $X^N_t$ may be seen as the
second order correction, in a scaled window of the interface from the
hydrodynamics, limiting to ${\bs e}(\vartheta) \sqrt{\varpi}B_t$ in
$d=1$ and ${\bs e}(\vartheta)Z_t(\theta)$ in $d=2$.  In both limits,
as $|{\bs e}(\vartheta)|$ diminishes for $|\vartheta|$ large, the
fluctuation has largest amplitude near the interface
$\Gamma_t = \{(\vartheta, \theta): \vartheta =0\}$ as, far from the
interface, the values $\rho_\pm$ are more stable.  It is interesting
that the fluctuations involve the `shape' of ${\bs e}$, indicating an
interesting `rigidity'. The appearance of $e(\cdot)$ in the limit
field can be interpreted as an indication of the fluctuation of the
interfaces retaining the shape of the transition layer $\phi$; see
Section 4.2 of \cite{F1}. Such rigidity might be understood in terms
of `metastability': In \cite{cp}, it is argued that macroscopic times
of order $e^{O(\sqrt{K})}$ are needed to move the interface in
dimension $d=1$.  Since $t\in [0,T]$ for $T>0$, such times are not
reached in the fluctuation limit in Theorem \ref{mt2}, which fixes the
initial interface for all $t\geq 0$.

As in other works, the proof of the relative entropy estimate follows
by analyzing its derivative. Here, we follow the
general scheme of \cite{jm2, jl}.  In \cite{jl}, a nonequilibrium
fluctuation limit for types of voter models is shown. However, unlike
in \cite{jl}, intrinsically due to variation across the interface,
$\nu_0^N$ is an inhomogeneous product measure which introduces several
differences. The inhomogeneity also affects the proof of tightness of
$X^N_t$. Indeed, we show only tightness for $\int_0^t X^N_sds$ instead
of $X^N_t$ in Theorem \ref{mt3}.

To identify the limit points and prove tightness we employ an analysis
of a Sturm-Liouville problem from which the `shape' ${\bs e}$ emerges
in the results. For $(t,v)\in [0,T]\times \sqrt{K}\T\times \T^{d-1}$,
and $v=(\vartheta, \theta)$, we will be able to approximate the
stochastic differential of $X^N_t$, via the Boltzmann-Gibbs principle,
as
$$\partial_t X^N_t(t, v) \approx K\, \big \{\,
\partial_{\vartheta}^2 X^N_t(t,v) - V''(\rho^K(\vartheta))\,
X^N_t(t,v) \,\big\} + \Delta_{d-1}X^N_t(t,v) + dM^N_t$$ where
$\Delta_{d-1}=\sum_{j=2}^d \partial^2_{\theta_j}$. One may consider
the mild form: By applying the field $X^N_t$ to the time-dependent
function $T^N_{t-s}G = e^{(t-s)K\mathcal{A}_K}e^{(t-s)\Delta_{d-1}}G$,
we will show that $X^N_t(G)$ is well approximated by
the value at time $t$ of a martingale $M^N_s(T^N_{t-s}G)$
(cf. Corollary \ref{2-cor1}). Here,
$\mathcal{A}_K = \partial^2_{\vartheta} - V''(\rho^K)$, for which it
is known that $\partial_{\vartheta}\rho^K$ is a ground state.

By the Sturm-Liouville analysis in \cite{cp} and \cite{F1}, the
semigroup $e^{tK\mathcal{A}_K}F$ converges to
${\bs e} \, \langle F, {\bs e}\rangle$. Then, we may approximate the
value at time $t$ of the martingale $M^N_s(T^N_{t-s}G)$ by the value
at time $t$ of the martingale
$M^N_s({\bs e}_K(\vartheta)e^{(t-s)\Delta_{d-1}}\langle
G\rangle(\theta_2,\ldots, \theta_d))$. By a martingale relation, this
last expression equals
$$M^N_t({\bs e}_K(\vartheta)\langle G\rangle(\theta_2,\ldots, \theta_d)) + \int_0^t M^N_s\big({\bs e}_K(\vartheta)\Delta_{d-1}e^{(t-s)\Delta_{d-1}}\langle G\rangle(\theta_2,\ldots, \theta_d)\big).
$$
In these terms,
$\langle G\rangle = \int {\bs e}(\vartheta)
G(\vartheta,\theta)d\vartheta$, and ${\bs e}_K$ is a $\sqrt{K}\T$
torus approximation of ${\bs e}$. Finally, in the $N\uparrow\infty$
limit, the covariances of the last display are computed as those of
the limits described earlier.

Last, we comment on the role of $K_N$ in the results. Since we do not
know the stationary measure of the Glauber + Kawaski dynamics on
$\T^d_N$, we are left to use $\nu_0^N$ as a proxy. The allowed growth
of $K_N$ in the estimates to prove the relative entropy bound and the
Boltzmann-Gibbs principle is therefore limited. As mentioned in
\cite{F1}, if the $K_N$ growth rate could be improved to $K_N = N^b$
for large enough powers $0<b<2$, then higher order terms in the
stochastic differential of $X^N_t$ can be homogenized so that one may
be able to obtain nonlinear SPDEs for the fluctuation limits of
$X^N_t$.

In Section \ref{sec2}, after introducing notation, we state the
results of the article. In Sections \ref{sec3}, \ref{sec8}, and
\ref{sec4}, we develop notions to prove the relative entropy bound
Theorem \ref{mt1}. In Section \ref{sec-BG}, we show the homogenization
or Boltzmann-Gibbs principle needed. In Section \ref{sec5}, we
identify the limits in Theorem \ref{mt2}, and in Section \ref{sec6},
we show the tightness results Theorem \ref{mt3}. Finally, in the
Appendix, several technical estimates are discussed.

\section{Notation and results}
\label{sec2}

Denote by $\color{blue} \bb T_N = \bb Z / N \bb Z$,
$N \in \bb N =\{1, 2, \dots \}$, the one-dimensional discrete torus
with $N$ points. Let $\color{blue} \Omega_N = \{0,1\}^{\bb T_N^d}$
be the state space. Elements of $\Omega_N$ are represented by the
Greek letters $\eta = (\eta_x : x\in \bb T^d_N)$ and $\xi$. Hence,
$\eta_x =1$ if the configuration $\eta$ is occupied at site $x$, and
$\eta_x =0$ otherwise.

Denote by $L_N^E$ the generator of the symmetric simple exclusion
process, or alternatively Kawasaki process, on $\bb T^d_N$ given by
\begin{equation}
\label{10}
(L_N^Ef)\, (\eta)\;=\; \sum_{x\in\bb T^d_N} \sum_{j=1}^d 
\, \big\{\, f(T^{x,x+e_j}\eta)-f(\eta)\, \big\}
\end{equation}
for all $f: \Omega_N \to \bb R$.  In this formula, $T^{ x,y}\eta$
represents the configuration of particles obtained from $\eta$ by
exchanging the occupation variables $\eta_x$ and $\eta_y$:
\begin{equation*}
{\color{blue} (T^{x,y}\eta)_{z} } \;:=\;
\begin{cases}
\eta_{z}  \,, & z\not =x \,,\, y\;,\\
\eta_{x} \,, & z=y\;, \\
\eta_{y} \,, & z=x\;,
\end{cases}
\end{equation*}
and $\color{blue} \{e_1,\dots,e_d\}$ the canonical basis of $\bb
R^d$. Here and below, summations are to be understood modulo $N$.

Fix a positive cylinder function
$c_0\colon \{0,1\}^{\bb Z^d} \to \bb R$, and let $L_N^G$ be the
generator of the Glauber dynamics on $\Omega_N$:
\begin{equation*}
(L_N^G f)(\eta)  \;=\; \sum_{x\in \bb T^d_N}  c_0(\tau_x \eta)
\,\big[\,  f(T^x \eta) \,-\, f(\eta)\, \big]
\end{equation*}
for all $f: \Omega_N \to \bb R$.  In this formula, $T^x\eta$
represents the configuration obtained from $\eta$ by flipping the
value of $\eta_x$:
\begin{equation*}
{\color{blue}(T^{x}\eta)_{z}} \;:=\;
\begin{cases}
\eta_{z}\;, & z\not =x\;,\\
1\,-\, \eta_{x}\;, & z=x\;,
\end{cases}
\end{equation*}
and $\{\tau_x : x\in \bb T^d_N\}$ are the translations acting on
$\Omega_N$: 
\begin{equation}
\label{12}
{\color{blue} (\tau_x \eta)_z}
\;:=\;\eta_{x+z}\;, \quad x\,,\, z\,\in\, \bb T^d_N\;,
\;\; \eta\,\in\, \Omega_N  \;.
\end{equation}

Fix an increasing sequence $(K_N:N\ge 1)$ such that $K_N\to+\infty$,
$K_N/N\to 0$, and satisfying certain growth conditions to be specified
later.  Let $\color{blue} (\eta^N(t); t\geq 0)$ be the
$\Omega_N$-valued, continuous-time Markov chain whose generator,
denoted by $L_N$, is given by
\begin{equation*}
L_N \;=\; N^2\, L_N^E \;+\; K_N \, L_N^G\;.
\end{equation*}
Denote by $\color{blue} (S^N(t) : t\ge 0)$ the associated  semigroup.

Fix a topological space $X$. Let $D(I,X)$, $I=[0,T]$, $T>0$, or
$I = \R_+$, be the space of right-continuous trajectories from $I$ to
$X$ with left-limits, endowed with the Skorohod topology.  For a
probability measure $\nu$ on $\Omega_{N}$, denote by
$\color{blue} \bb P_{\nu}$ the probability measure on
$D(\R_{+},\Omega_{N})$ induced by the process $\eta^{N} (t)$ starting
from $\nu$. The expectation with respect to $\bb P_{\nu}$ is
represented by $\color{blue} \bb E_{\nu}$. Denote by $\bb P_{\eta}$
the measure $\bb P_{\nu}$ when the probability measure $\nu$ is the
Dirac measure concentrated on the configuration $\eta$. Analogously,
$\bb E_{\eta}$ stands for the expectation with respect to
$\bb P_{\eta}$.

Denote by $\color{blue} \nu^N_\alpha$, $0\le \alpha \le 1$, the
Bernoulli product measure on $\Omega_N$ with density $\alpha$. This is
the product measure whose marginals are Bernoulli distributions with
parameter $\alpha$. A straightforward computation shows that these
measures satisfy the detailed balance conditions for the symmetric
exclusion process. In particular, they are stationary for this
dynamics.

\subsection*{Hydrodynamical limit}

Assume in this subsection that the sequence $K_N$ is constant equal to
$K$: $K_N=K$.  Let $\color{blue} \bb T = \bb R \backslash \bb Z$ be
the continuous one-dimensional torus, Denote by $\mc M_{+}$ the space
of nonnegative measures on $\bb T^d$ with total mass bounded by $1$,
endowed with the weak topology.  For a measure $\mu$ in
$\mathcal{M}_+$ and a continuous function $G:\bb T^d\to\R$, denote by
$\lan \mu, G \ran$ the integral of $G$ with respect to $\mu$:
\begin{equation*}
\lan \mu , G \ran \;=\; \int_{\bb T^d}  G(\theta) \; \mu(d\theta) \;.
\end{equation*}

Let $\pi^N\colon \Omega_N\to\mc M_+$ be the function which associates to a
configuration $\eta$ the positive measure obtained by assigning mass
$N^{-d}$ to each particle of $\eta$,
\begin{align*}
\pi^N(\eta)\;=\;\frac{1}{N^d}\sum_{x\in \bb T^d_N}\eta_{x}\, \delta_{x/N}\;,
\end{align*}
where $\delta_\theta$ stands for the Dirac measure which has a point
mass at $\theta\in\bb T^d$. Let $\pi^{N}_{t} = \pi^N (\eta^N(t))$,
$t\ge 0$.  The next result is due to De Masi, Ferrari and Lebowitz in
\cite{dfl}. We refer to \cite{dfl, jlv, kl} for a proof. Let
$V\colon [0,1] \to \bb R$ be the potential characterized by the
identity
\begin{equation}
\label{2-54}
{\color{blue} V'(\alpha)}
\;=\; -\, E_{\nu^N_\alpha} \big[\, c_0(\eta) \, [1- 2\eta_0]\,\big]\;.
\end{equation}

\begin{theorem}
\label{hydrodynamics}
Fix a measurable function $\rho \colon \bb T^d\to[0,1]$.  Let $\nu_N$
be a sequence of product probability measures on $\Omega_N$ associated
to $\rho$, in the sense that
\begin{equation*}
\lim_{N\to\infty}
\nu_N \Big(\, \big |\lan\pi^{N},G\ran -\int_{\bb T^d}
G(\theta)\, \rho(\theta)\, d\theta\,\big|>\delta\, \Big)\;=\;0\;,
\end{equation*}
for every $\delta>0$ and every continuous function $G\colon \bb T^d\to\R$.  Then,
for every $t\ge 0$, every $\delta>0$ and every continuous function
$G:\bb T^d\to\R$, 
\begin{equation*}
\lim_{N\to\infty}
\mathbb{P}_{\nu_N}
\Big(\, \big |\lan\pi^{N}_t ,G\ran -\int_{\bb T^d}
G(\theta)\, u(t,\theta)\, d\theta\,\big|>\delta\, \Big) \;=\;0\;,
\end{equation*}
where $u\colon [0,\infty)\times\bb T^d\to[0,1]$ is the unique weak solution of
the Cauchy problem
\begin{equation}
\begin{cases}
     \partial_tu \;=\; \Delta u \,-\, K\, V'(u) \ \text{ on }\ \bb T^d\;,\\
     u(0,\cdot)\;=\;\rho (\cdot)\;,
\end{cases}
\label{hyd equation}
\end{equation}
and $V'(\cdot)$ is given by \eqref{2-54}.
\end{theorem}

The non-equilibrium fluctuations around the hydrodynamic limit have
been examined in \cite{jm2, gjmm}, and the dynamical and static
large deviations in \cite{jlv, bl, lt, flt}

\subsection*{Motion by mean curvature}

Consider a smooth evolution of closed hypersurfaces
$\{\Gamma_t: t\in [0,T]\} \subset \T^d$.  Let $n(t)$ and $\kappa(t)$
be the `inward' normal velocity and mean curvature of $\Gamma_t$,
fixing one side of $\Gamma_t$ as `inward'.  We say the
evolution is motion by mean curvature starting from $\Gamma_0$ if
\begin{align*}
& n (t) = (d-1)\kappa(t), \ \ \ t\in [0,T]\\
& \Gamma_t|_{t=0}  = \Gamma_0.
\end{align*}
When $\Gamma_0$ is a smooth hypersurface without boundary, it is known
that there exists a smooth solution $\{\Gamma_t: t\in [0,T]\}$
evolving by mean-curvature starting from $\Gamma_0$ for a $T>0$.  One
may also view the mean curvature evolution as a local graph flow in
appropriate coordinates.  See Section 4 in \cite{Belletini} and
\cite{Funakibook}, among other references, for discussion of these
issues.

In the PDE literature, the sharp interface limit of solutions $u=u^K$
as $K$ diverges of the reaction-diffusion equation \eqref{hyd
equation}, under conditions on $V$, to mean curvature flow is well
known; see \cite{Belletini}.  Suppose the potential $V$ is a `double
well', that is $V$ has exactly three critical points
$\rho_- < \rho_* < \rho_+$ such that $\rho_\pm$ are local minima with
equal values $V(\rho_-)=V(\rho_+)$ and $\rho_*$ is a local maximum.
Under such a $V$, \eqref{hyd equation} is seen as an `Allen-Cahn'
equation, and $u^K$ converges pointwise to a step function defined by
mean curvature flow: For $\theta \in \T^d$ and $t\in [0,T]$,
\begin{align}
\label{meancurvaturelimit}
\lim_{K\rightarrow \infty} u^K(t, \theta) = \left\{\begin{array}{rl}
\rho_- & {\rm for \ } \theta  \ {\rm on \ one \ side \ of \ } \Gamma_t\\
\rho_+ & {\rm for \ } \theta \ {\rm on \ the \ other \ side \ of \ }\Gamma_t,
\end{array}
\right.
\end{align}
where $\Gamma_t$ evolves by mean curvature.

As remarked in Section \ref{sec1}, derivation of mean-curvature flow
from reaction-diffusion particle systems is also known; see
\cite{Bonaventura}, \cite{KS94}, \cite{FT}, \cite{kfhps},
\cite{FMSTspeedchange}.  We state a result in Theorem 1.1
of \cite{FT} which applies in our setting of Glauber + Kawasaki
particle dynamics.

Consider the particle system generated by $L_N$ where $K=K_N$ diverges
but is bounded $K_N \leq \delta \log(N)$ for a small constant
$\delta>0$.  Let $\rho_0:\T^d \rightarrow (\rho_-,\rho_+)$ be a
$C^\infty$ function and
$\Gamma_0 = \{\theta : \rho_0(\theta) = \rho_*\}$ its associated
initial interface. We will also suppose that the gradient of $\rho_0$,
denoted by $\color{blue} \nabla \rho_0$, is not parallel to the normal
direction to $\Gamma_0$, and $\rho_0(\theta) <\rho_*$ on $D_-(0)$ and
$\rho_0(\theta)>\rho_*$ on $D_+(0)$ where $D_\pm(0)$ are the regions
separated by $\Gamma_0$. Let $\{\Gamma_t:t\in [0,T]\}$ be the mean
curvature flow starting from $\Gamma_0$, and $D_\pm(t)$ be the regions
separated by $\Gamma_t$.

Let $\nu_N$ be a sequence of product probability measures on
$\Omega_N$ with Bernoulli marginals with success probabilities
$\rho_0(\theta/N)$.  Then, for $t\in [0,T]$ and $G\in C^\infty(\T^d)$,
we have
$$\lim_{N\rightarrow\infty} \mathbb{P}_{\nu_N}\left(\left| \langle
\pi^N_t, G\rangle - \langle \chi_{\Gamma_t}, G\rangle\right| >
\epsilon\right) = 0,$$ 
where 
$$\chi_{\Gamma_t}(\theta) = \left\{\begin{array}{rl}
\rho_- & \ {\rm for \ } \theta\in D_-(t)\\
\rho_+ & \ {\rm for \ } \theta \in D_+(t).
\end{array}\right.
$$

This mean curvature limit can be seen as a `law of large numbers'.  In
our later results, we will consider associated fluctuations when
$\Gamma_0$ is a `flat' interface, not moving by mean-curvature.

\subsection*{A model}
To fix ideas, from now on we assume that the Glauber jump rates are
given by
\begin{equation}
\label{01}
{\color{blue} c_0(\eta)}
\,:=\, 1\,-\, \gamma \, \sigma_0\, [\, \sigma_{-e_1} + \sigma_{e_1}
\,] \,+\, \gamma^2 \, \sigma_{-e_1}\, \sigma_{e_1} \;,
\end{equation}
where $\color{blue} \sigma_x = 2\eta_x -1 \in \{-1,1\}$ and
$|\gamma| \le 1$ for the rate to be positive. Mind that for $\gamma=1$
the rate $c_0(\cdot)$ vanishes if $\sigma(0)$, $\sigma(-e_1)$,
$\sigma(+e_1)$ have the same sign.  For these jump rates, the
potential $V\colon [0,1] \to \bb R$ introduced in \eqref{2-54} is
given by
\begin{equation*}
{\color{blue} V'(\rho)}
\;=\; -\, E_{\nu^N_\rho} \big[\, c_0(\eta) \, [1- 2\eta_0]\,\big]
\;=\; -\, (2\gamma-1)(2\rho-1) + \gamma^2 (2\rho-1)^3\;.
\end{equation*}
For $\color{blue} 1/2<\gamma\le 1$, $V(\cdot)$ is a double-well
potential with a local maximum at $\rho=1/2$, and two local minima at
$0<\rho_-<1/2<\rho_+ = 1-\rho_- = \frac{1}{2}+\frac{\sqrt{2\gamma
-1}}{2\gamma}$. Note that these Glauber rates
$c_0(\eta)$ depend only on $\eta_x$ for $x$ on the first coordinate
axis.

The arguments presented in this article do not require the jump rates
to have the specific form introduced in \eqref{01}, and apply to
Glauber dynamics for which the potential $V$ defined in \eqref{2-54}
has two local minima $\rho_-$, $\rho_+$ at the same depth separated by
a local maximum $\rho^*$: $0<\rho_- < \rho^*< \rho_+<1$,
$V(\rho_-) = V(\rho_+)$.

\subsection*{Entropy production}

Denote by $u^{(\epsilon)}\colon \bb T \to \bb R$ the solution of the
differential equation
\begin{equation}
\label{2-56}
\epsilon \, \Delta u^{(\epsilon)}
\,-\,  \, V'(u^{(\epsilon)}) \,=\, 0
\end{equation}
with two layers: $\sharp\{\theta_1 \in \T; u(\theta_1)= \rho^* \}=2$;
we will take $\epsilon = 1/\sqrt{K}$ in the next paragraph.  Such
$u^{(\epsilon)}$ exists uniquely except for translation; see \cite{cp,
F1} and the Appendix \ref{sec-c} for these claims and properties of
the solution that will be used later.  To fix the idea, we normalize
the solution as $u^{(\epsilon)}(0)= \rho^*$,
$\partial _{\theta_1} u^{(\epsilon)} (0)<0$, where
$\partial _{\theta_1} u^{(\epsilon)}$ represents the derivative of
$u^{(\epsilon)}$.

Fix $K\in \bb N$, and let $\sqrt{K} \bb T$ be the torus of length
$\sqrt{K}$.  Set $\epsilon = 1/\sqrt{K}$, and let
$\rho^K\colon \sqrt{K}\bb T \to (0,1)$ be given by
$\rho^K(\vartheta) := u^{(\epsilon)} (\vartheta/\sqrt{K}\,
)$. Clearly, $\rho^K$ has period $\sqrt{K}$ and solves the equation
\begin{equation}
\label{2-23b}
\Delta \rho^K \,-\,  \, V'(\rho^K) \,=\, 0\;, \quad
\rho^K(0)= \rho^*\,, \quad \partial_\vartheta \rho^K \,  (0) < 0\;.
\end{equation}

Denote by $\color{blue} \phi\colon \bb R \to (0,1)$ the decreasing
standing wave solution:
\begin{equation}
\label{2-79}
\Delta \phi \,-\, \, V'(\phi) \,=\, 0\;, \quad \phi(
\infty )= \rho_{-}\,, \quad \phi(-
\infty )= \rho_{+}\,, \quad \phi (0) = \rho^* \;.
\end{equation}
By Lemma \ref{CP_lemma}, as $K\to\infty$, on all intervals
$[-A,A] \subset \mathbb{R}$, $\rho^K$ converges uniformly to $\phi$.
This means that the solution $u^{(\epsilon)}$ of \eqref{2-56}, with
$\epsilon = 1/\sqrt{K}$, evolves smoothly on the spatial scale
$\theta/\sqrt{K}$.

We construct from the solution $u^{(\epsilon)}$ a discrete density
profile $\bs u^N\colon \bb T^d_N \to (0,1)$ on $\bb T^d_N$ as
follows. Let $u^N\colon \bb T_N \to (0,1)$ be given by
\begin{equation}
\label{2-52}
{\color{blue} u^N(x) \,:=\,} u^{(\epsilon)} (x/N) \;=\;
\rho^K (x \sqrt{K}/N) \quad x\in \bb T_N \;,
\end{equation}
and set
\begin{equation}
\label{2-53a}
{\color{blue} \bs u^N(x_1, \dots, x_d)} \,:=\, u^N(x_1)\;.
\end{equation}

Given a density profile $\bs v \colon \bb T^d_N \to [0,1]$, let
$\nu^N_{\bs v(\cdot)}$ be the product measure on $\{0,1\}^{\bb T^d_N}$
with marginals given by
\begin{equation}
\label{2-12}
{\color{blue} \nu^N_{\bs v(\cdot)} \{\eta_x = 1\}} \,:=\,
\bs v(x)\;, \;\;
\text{and let}\;\;  {\color{blue} \nu^N := \nu^N_{\bs u^N(\cdot)}}\;.
\end{equation}

The first main result of this article provides an estimate of the
entropy production.  It reads as follows. Let $f^N_t$, $t\ge 0$, be
the density of the measure $\nu^N S^N(t)$ with respect to $\nu^N$:
\begin{equation}
\label{2-34c}
f_t^N \;:=\; \frac{d\, \nu^N\,  S^N(t)}{d\, \nu^N}\;,
\end{equation}
and denote by $H_N(f)$ the entropy of a measure $f\, d\nu^N$ with
respect to $\nu^N$:
\begin{equation}
\label{2-34d}
{\color{blue} H_N(f)}\,:=\, \int f \, \log f \, d\nu^N\;  .
\end{equation}

\begin{theorem}
\label{mt1}
Let $(\ell_N:N\ge 1)$ be the sequence such that $\ell_N = N^{4/5}$ in
dimension $1$, $\ell^3_N \log \ell_N = N^2$ in
dimension $2$, and $\ell_N = N^{4/3d}$ in dimension $d\ge 3$.  Assume
that $\lim_{N\uparrow\infty}K_N=\infty$ and
$\sup_{N\ge 1} K^2_N / \ell^d_N < \infty$.  Then, there exists a
constant $\mf c(\bs u^N)$, depending only on $\bs u^N$ and uniformly
bounded in $N$, such that
\begin{equation*}
\frac{d}{dt}\, H_N(f^N_t) \;\leq\; 
\mf c(\bs u^N)\,\, K^2_N
\, \big\{ H_N(f^N_t) \,+\, (N/\ell_N)^d \, \big\}
\end{equation*}
for all $t\ge 0$, $N\ge 1$. 
\end{theorem}

The next result follows by integrating the bound for
$(d/dt) \, \{ H_N(f^N_t) \, \exp [\, - \mf c(\bs u^N) \, K^2_N \, t]\, \}$ and
noting that $H_N(f^N_0)=0$.

\begin{corollary}
\label{m-cor}
Under the hypotheses of Theorem \ref{mt1}, there exists a finite
constant $\mf c_0 = \mf c_0(\bs u^N)$, depending only
on $\bs u^N$ and uniformly bounded in $N$, such that
\begin{equation*}
H_N(f^N_t) \;\leq\; 
(N/\ell_N)^d\, e^{\mf c_0 \, K^2_N \, t}
\end{equation*}
for all $t\ge 0$, $N\ge 1$. 
\end{corollary}

In Theorem \ref{mt1}, $\ell_N$ represents the linear size of the cube
on which a two-block replacement is carried out. Denote by
$\color{blue} R_d(N) = (N/\ell_N)^d$ the order of the entropy in
dimension $d$. By definition of the sequence $\ell_N$ introduced in
\eqref{2-18}, \eqref{2-18b},
\begin{align}
\label{2-28}
\ell_N \, & =\, N^{4/5} \;, &
R_1(N) \; &=\; N^{1/5}\;, &
&\text{in}\;\; d=1\;, \nonumber
\\
\ell_N \, &\approx \, N^{2/3}\;, &
R_2(N) \;& \approx \; N^{2/3}\;,&
& \text{in}\;\; d=2\;,
\\
\ell_N \,& = \, N^{4/3d} \;, &
R_d(N) \;&=\; N^{d-(4/3)}
\;, &
& \text{in}\;\; d\ge 3\;. \nonumber
\end{align}

\subsection*{Growth conditions on $K_N$}

In the interface fluctuations analysis certain conditions on the
growth of $K_N$ are needed. We present them in this subsection.

Fix $T\ge 1$, and let
$\color{blue} \kappa_N(T) := \exp\{\mf c_0 K^2_N T\}$, where $\mf c_0$
is the constant appearing in Corollary \ref{m-cor}. With this choice,
\begin{equation}
\label{2-50b}
H_N(f^N_t) \,\le\,  (N/\ell_N)^d\, \kappa_N(T) \quad
\text{for all} \;\; 0\le t\le T\;.
\end{equation}
Assume that the sequence $\ell_N$, introduced in the statement of
Theorem \ref{mt1}, satisfies
\begin{equation}
\label{2-80}
\ell_N \,\gg\, \max \{ \sqrt{N} , N^{(d-2)/d}\}\;.
\end{equation}
This requirement restricts the proof to dimensions $1$ and
$2$. 
Let $\mf c_5$ be the constant introduced in the maximum principle
stated in Lemma \ref{maximum lemma}. Suppose, furthermore, that
\begin{equation}
\label{2-90}
K^{17/4} \,  \sqrt{\kappa_N(T)}\,
e^{4 \mf c_5 K T} \,\ll\,
\Big(\frac{\ell_N} {N} \Big)^{d/2} \, N  \;,
\quad
K^{11/4} \, \kappa_N(T)
\, e^{2 \mf c_5 K T} \, \ll \, \Big( \frac{\ell_N} {N} \Big)^d
\, N^{d/2}  \;.
\end{equation}
In the proof of the tightness, the following further bound is needed.
Assume that
\begin{equation}
\label{2-48}
\sup_{N\ge 1} \frac{1}{N^{(d/2)(1-\delta)}} \,
e^{2 (1+2\delta^2)  (\mf c_5 K+1)T}\,
K^{7/4}\, \big[\, \kappa_N(T) \,
(N/\ell_N)^d \,\big]^{1-(\delta/2)}
\,<\, \infty
\end{equation}
for some $0<\delta<2/3$. Note that in dimension $1$ and $2$,
$(N/\ell_N) / N^{1/2} \le N^{-1/7}$. Hence, there exists
$0<\delta_0<2/3$ such that for all $0<\delta<\delta_0$, we have
$(N/\ell_N)^{1-(\delta/2)} / N^{(1-\delta)/2} \le N^{-a}$ for some
$a=a(\delta)>0$. Condition \eqref{2-48} requires
$e^{2 (1+2\delta^2) (\mf c_5 K+1)T}\, K^{7/4} \,
\kappa_N(T)^{1-(\delta/2)}$ to be bounded by $C_0 N^{ad}$ for some
finite constant $C_0$.

It follows from the second condition in \eqref{2-90} in $d=1$ or $2$
that
\begin{equation}
\label{2-72}
K^{11/4}_N\, \kappa_N(T)\, N^{d/2} \,\ll\, \ell_N^d\;, \quad
K^{5/2}_N\, \kappa_N(T)\, N^{d-2} \,\ll\, \ell_N^d\;.
\end{equation}
These estimates yield that
\begin{equation}
\label{2-67}
K_N\, \kappa_N(T) \ll \ell_N^d\;, \quad
K_N \ll \ell^d_N \,, \quad
K^{3/2}_N \,\ll\,  N^{2 - (d/2)} \;.
\end{equation}
Finally, by \eqref{2-90}, in $d=1$ or $2$,
\begin{equation}
\label{2-89}
K^{5/4}_N  \,  \sqrt{\kappa_N(T)}\,
\, e^{\mf c_5 K T} \,\ll \,
\Big(\frac{\ell_N} {N} \Big)^{d/2} \, N^{d/4}
\,\ll \,
\Big(\frac{\ell_N} {N} \Big)^{d/2} \, N  \;.
\end{equation}

Certainly, to give an example, $K_N \leq \delta_0\sqrt{\log(N)}$ for
$\delta_0=\delta_0(c_0, \frak{c_5}, T)>0$ small enough satisfies the
above requirements.

\subsection*{Density fluctuations}

Denote by $\color{blue} C^\infty_c(\bb R)$,
$\color{blue} \ms D(\bb R \times \bb T^{d-1}) = C^\infty_c(\bb R
\times \bb T^{d-1})$ the spaces of infinitely differentiable functions
with compact support on $\bb R$,  $\bb R \times \bb T^{d-1}$,
respectively. Let $X^N_t$, 
$t\ge0$, be the random element of $\ms D'$ defined by
\begin{equation}
\label{2-16}
X_t^N(F) \;=\; \frac{1}{\sqrt{N^d K^{1/2}} }
\sum_{x\in \bb T^d_N} \widehat {F} (Ax)\, 
(\, \eta_x^N(t) \,-\, \bs u^N(x)\, )\;, 
\end{equation}
for $F\in \ms D(\bb R \times \bb T^{d-1})$.  In this formula and
below,
$A\colon N \bb T^d \to \sqrt{K} \bb T \times \bb
T^{d-1}$ represents the map
\begin{equation}
\label{2-44}
A x \, =\, (x_1 \sqrt{K} /N, x_2/N, \dots, x_d/N)\;,
\end{equation}
regarding $N \bb T^d = [- (N/2), (N/2))^d$, $\sqrt{K} \bb T
= [- (\sqrt{K}/2), (\sqrt{K}/2))$, $\bb T^{d-1} = [- (1/2), (1/2))^{d-1}$.
For a function $G\colon \bb R \times \bb T^{d-1}$,
$x\in \bb T^d_N$,
\begin{equation}
\label{2-69}
{\color{blue} \widehat G (Ax)} \,: =\
\frac{N^d}{\sqrt{K}}\, \int_{\square_N(x)} G(\vartheta,
\bar \theta) \, d\vartheta \, d\bar\theta\;, 
\end{equation}
where we write $\vartheta$ for $\vartheta \in \sqrt{K} \bb T$ or
$\vartheta \in \bb R$, $\theta$ for $\theta \in \bb T^d$, and
$\bar\theta$ for $\bar\theta \in \bb T^{d-1}$. Moreover,
$\square_N(x)$ is the $d$-dimensional hyperrectangle defined by
$\square_N(x) := [A(x- {\bf 1/2}), A(x + {\bf 1/2}] \subset \sqrt{K}
\bb T \times \bb T^{d-1}$,
${\bf 1/2} = (1/2, \dots, 1/2) \in \bb R^d$.

The statement of the second main result of the article requires some
notation. Denote by $\color{blue} (S_t:t\ge 0)$ the semigroup of the
heat equation in $\bb T^{d-1}$. Recall that $\phi(\cdot)$ introduced
in \eqref{2-79} represents the standing wave solution of the
reaction-diffusion equation.  Let $\bs e\colon \bb R \to \bb R$ be
given by
\begin{equation}
\label{2-83}
{\color{blue} \bs e(\vartheta) } \,:=\,
(\partial_\vartheta \phi) (\vartheta)/
\Vert \partial_\vartheta \phi  \Vert_{L^2(\bb R)} \;,
\end{equation}
where $\color{blue} \Vert F \Vert_{L^2(\bb R)}$ represents the
$L^2(\bb R)$ norm of a function $F\colon\bb R\to \bb R$.  Denote by
$\color{blue} \< F, G\>$ the scalar product in $L^2(\bb R)$.  For a
function $F$ in $\ms D(\bb R \times \bb T^{d-1})$, let
$\<F\>\colon \bb T^{d-1}\to \bb R$,
$\Pi F \colon \bb R \times \bb T^{d-1} \to \bb R$ be given by
\begin{equation}
\label{2-75}
{\color{blue} \< F\>  (\bar\theta) } \,=\, 
\< F(\cdot, \bar\theta)\,,\, \bs e(\cdot)\>\;,
\quad 
{\color{blue} ( \Pi F)  (\vartheta, \bar\theta)}
\,:=\, \bs e (\vartheta)\, \< F\> (\bar \theta) \;.
\end{equation}
Note that
\begin{equation}
\label{2-77}
S_{t-r} \Pi F = \bs e \; S_{t-r} \< F\>,\quad \;
S_{t-r} \partial_{\vartheta} \Pi F = \partial_{\vartheta} \bs e
\; S_{t-r} \< F\>\;.
\end{equation}

\begin{theorem}
\label{mt2}
Fix $T\ge 1$, and assume that conditions \eqref{2-80}, \eqref{2-90}
are in force. In particular, $d=1$ or $2$.  Fix $p\ge 1$,
$0\le t_1<t_2 < \cdots <t_p\le T$ and functions
$F_j\in \ms D(\bb R \times \bb T^{d-1})$.  Under the measure
$\bb P_{\nu^N}$, the random vector
$(X^{N}_{t_1}(F_1), \dots, X^{N}_{t_p}(F_p))$ converges in
distribution to the centered Gaussian random vector whose covariances
are given by
\begin{equation*}
\begin{aligned}
\text{\rm Cov\,} \big(\, X_{t}(F) \,,\, X_{s}(G)\,\big)\; & =\;
2 \, \int_0^{s} dr\, \int_{\bb R\times \bb T^{d-1}} 
\chi (\phi) \, S_{t-r} \partial_{\vartheta} \Pi F\,
\, S_{s-r} \partial_{\vartheta} \Pi G\, d\vartheta\, d \bar\theta
\\
& +\;
\int_0^{s} dr\, \int_{\bb R\times \bb T^{d-1}} 
\widehat c_0(\phi) \, S_{t-r} \Pi F\,
\, S_{s-r} \Pi G\, d\vartheta\, d \bar\theta\;.
\end{aligned}
\end{equation*}
for every $0\le s \le t\le T$, $F$,
$G\in \ms D(\bb R \times \bb T^{d-1})$.
\end{theorem}

\noindent In the previous result,
$\color{blue} \chi(\alpha) = \alpha (1-\alpha)$ represents the static
compressibility of the exclusion dynamics, and
$\color{blue} \widehat c_0(\alpha) = E_{\nu^N_\alpha}[c_0]$.

By \eqref{2-77}, the previous covariance can be written as
\begin{equation}
\label{2-78}
\text{\rm Cov\,} \big(\, X_{t}(F) \,,\, X_{s}(G)\,\big)\; =\;
\varpi \, \int_0^{s} dr\,
\int_{\bb T^{d-1}}  \, S_{t-r}  \<F\>\,
\, S_{s-r} \<G\> \, d \bar\theta \;. 
\end{equation}
where
\begin{equation*}
\varpi \,=\, \int_{\bb R}
\big\{ \, 2\, \chi (\phi) \,  [\partial_{\vartheta} \bs e ]^2
\,+\, \widehat c_0(\phi) \, \bs e^2  \, \big\}\;
d\vartheta 
\;=\; \big\Vert \sqrt{ 2\, \chi (\phi)} \,  \partial_{\vartheta} \bs
e\, \big\Vert^2_{L^2(\bb R)} \,+\,
\big\Vert \sqrt{  \widehat c_0(\phi) } \,  \bs
e\, \big\Vert^2_{L^2(\bb R)} \; .
\end{equation*}
Note that $\varpi$ coincides with $c^2_*$ in (4.7) of \cite{F1}.

By \eqref{2-78}, in dimension 1,
\begin{equation}
\label{2-81}
X_t (\vartheta) \;=\; \bs e(\vartheta)\,  \sqrt{\varpi} \,  B_t
\quad \text{so that}\quad X_t(F) \,=\, \sqrt{\varpi} \,
\<F\,,\, \bs e \> \, B_t\;, 
\end{equation}
where $B_t$ is a Brownian motion, and in dimension $2$,
\begin{equation}
\label{2-82}
X_t (\vartheta, \theta) \;=\;  \bs e(\vartheta)\,  Z_t
(\theta) \;,
\end{equation}
where $Z_t$ is the solution of the stochastic heat equation on the
one-dimensional torus with initial condition equal to $0$:
\begin{equation*}
\left\{
\begin{aligned}
& \partial_t Z_t\,=\, \Delta Z_t  \,+\, \sqrt{\varpi}\, 
\xi(t,\theta) \;, \\
& Z_0 = 0\;,
\end{aligned}
\right.
\end{equation*}
and $\xi$ is the space-time Gaussian white noise on
$\bb R_+\times \bb T$. As noted in Section \ref{sec1}, our result, in
particular, the projection to $\bs e(\cdot)$ in the first variable,
indicates the fluctuation of the interface retaining the shape of the
transition layer $\phi$ to the normal direction to the interface.

We turn to the tightness of the sequence $X^N_t$. By \cite[Chapter
10]{v20}, $\ms D(\bb R \times \bb T^{d-1})$ is a LF-space (a locally
convex inductive limit of a sequence of Fr\'echet spaces). Denote by
$\color{blue} \ms D'$ the dual of $\ms D(\bb R \times \bb T^{d-1})$,
and by $\bb Q_N$ the measure on $C([0,T], \ms D')$ induced by the
measure $\bb P_{\nu^N}$ and the process
$\bb X^N_t = \int_0^t X^N_s\, ds$, $t\ge 0$:
$\bb Q_N := \bb P_{\nu^N} \circ (\bb X^N)^{-1}$.

\begin{theorem}
\label{mt3}
Fix $T\ge 1$ and assume that the
hypotheses of Theorem \ref{mt2} are in force. Suppose, furthermore,
that \eqref{2-48} holds. Then, the sequence of probability measures
$\bb Q_N$ is tight.
\end{theorem}

Theorem \ref{mt2} characterizes the limit points of the sequence $\bb
Q_N$. Therefore, by Theorem \ref{mt2} and \ref{mt3}, the sequence
$\bb Q_N$ converges weakly to the measure on $C([0,T], \ms D')$ induced by
the time integral of the processes $X_t$ introduced in \eqref{2-81}
and \eqref{2-82}.

\section{Entropy production}
\label{sec3}

We present in this section a bound on the entropy production. This is
one of the main estimates needed in the proof of Theorem \ref{mt1}.
For a probability measure $\mu^N$ and a density $h$ with respect to
$\mu^N$, denote by $I_N(h; \mu^N)$ the carr\'e du champ of $h$
integrated with respect to $\mu^N$ given by
\begin{equation}
\label{13}
{\color{blue}  I_N(h; \mu^N)}
\;:=\;  \frac{1}{2}\, \int \big\{ L_N h
\,-\, 2\, (\, L_N \sqrt{h} \,) \,
\sqrt{h} \,\big\} \; d \mu^N\;.
\end{equation}
An elementary computation yields that
\begin{equation*}
I_N(h;\mu^N) \;=\; N^2 \, I^E_N(h;\mu^N) \;+\; K \, I^G_N(h;\mu^N)\;,
\end{equation*}
where
\begin{equation*}
\begin{gathered}
I^E_N(h;\mu^N)  \;=\; \frac{1}{2}\, \sum_{j=1}^d \sum_{x\in\bb T^d_N} 
\int \big[\,  \sqrt{h (T^{x,x+e_j} \eta)} \,-\,
\sqrt{h(\eta)}\, \big]^2 \; d \mu^N\;, \\
I^G_N(h;\mu^N)  \;=\;  \frac{1}{2}\, \sum_{x\in\bb T^d_N} 
\int c_0(\tau_x \eta)\, \big[\,  \sqrt{h (T^{x} \eta)} \,-\,
\sqrt{h(\eta)}\, \big]^2 \; d \mu^N \;.
\end{gathered}
\end{equation*}

Fix a density profile
$\color{blue} \varrho\colon \bb T^d_N \to (0,1)$. Recall from
\eqref{2-12} that we denote by $\nu^N_{\varrho (\cdot)}$ the product
measure on $\Omega_N$ with marginals given by $\varrho(\cdot)$.
Recall that $S^N(t)$ represents the semigroup associated to the
generator $L_N$.  Let $h^N_t$, $t\ge 0$, be the density of the measure
$\nu^N_{\varrho (\cdot)} S^N(t)$ with respect to
$\nu^N_{\varrho (\cdot)}$:
\begin{equation}
\label{2-34}
h_t^N \;:=\; \frac{d\, \nu^N_{\varrho (\cdot)}\,
S^N(t)}{d\, \nu^N_{\varrho (\cdot)}}\;,
\end{equation}
and denote by $H_N(h^N_t\,|\, \nu^N_{\varrho (\cdot)})$
the entropy of $\nu^N_{\varrho (\cdot)} S^N(t)$ with respect to
$\nu^N_{\varrho (\cdot)}$:
\begin{equation}
\label{2-34b}
{\color{blue} H_N(h^N_t\,|\, \nu^N_{\varrho (\cdot)})}\,:
=\, \int h^N_t\, \log h^N_t \, d\nu^N_{\varrho (\cdot)}\;  .
\end{equation}
The proof of the next result is similar to the one of
\cite[Proposition 5.1]{jl}. Note that in this computation the carr\'e
du champ engendered by the Glauber dynamics appears, but will not be
used in the rest of the argument.

\begin{proposition}
\label{l06-1}
For all $t\ge 0$,
\begin{equation}
\label{2-22}
\frac{d}{dt} H_N(h^N_t\,|\, \nu^N_{\varrho (\cdot)})
\;\leq\; -\,2\,  I_N(h_t^N; \nu^N_{\varrho (\cdot)})
\;+\;    \int  h^N_t \, L^*_{N} \, \mb 1   \; d \nu^N_{\varrho (\cdot)} \;,
\end{equation}
where $\color{blue} \mb 1 \colon \Omega_N \to \bb R$ is the constant
function equal to $1$, and $\color{blue} L^*_{N}$ represents the
adjoint of $L_N$ in $L^2(\nu^N_{\varrho (\cdot)})$.
\end{proposition}

In the rest of this section, we compute $L^*_{N} \, \mb 1 $ for a
generic product measure $\nu^N_{\varrho (\cdot)}$.  Denote by $L^{E,*}_N$, $L^{G,*}_N$
the adjoint of the generators $L^E_N$, $L^G_N$ in $L^2(\nu^N_{\varrho (\cdot)})$,
respectively. A straightforward computation yields that
\begin{equation}
\label{02}
\begin{gathered}
(L^{E,*}_N g)(\eta) \;=\;
\sum_{j=1}^d \sum_{x\in \bb T^d_N} \Big\{\,
\frac{\nu^N_{\varrho (\cdot)}(T^{x,x+e_j} \eta)}{\nu^N_{\varrho (\cdot)}(\eta) } \,
g(T^{x,x+e_j} \eta) \,-\, g(\eta)  \, \Big\}\;,
\\
(L^{G,*}_N g)(\eta) \;=\;
\sum_{x\in \bb T^d_N} \Big\{\,
\frac{\nu^N_{\varrho (\cdot)}(T^x \eta)}{\nu^N_{\varrho (\cdot)}(\eta) } \, c_0(\tau_x  T^x \eta) \,
g(T^x \eta) \,-\, c_0(\tau_x \eta) \,g(\eta)  \, \Big\} 
\end{gathered}
\end{equation}
for all functions $g\colon \Omega_N\to \bb R$.

The next result requires some notation.  Let
\begin{equation}
\label{2-20}
\begin{gathered}
{\color{blue} \omega_x} \,:=\, \eta_x - \varrho (x)\;,
\quad
{\color{blue} \zeta_x} \;:=\;
\frac{\omega_{x}}{\chi(x)} \;=\;
\frac{\eta_{x} - \varrho (x)}{\chi(x)}\; , \\
{\color{blue} \chi(x)}  \,:=\,  \varrho(x)\, [1-\varrho(x)]\;, \quad 
\quad
{\color{blue} \beta(x)}  \,:=\,
2\varrho(x) - 1 \;, \quad x\in \bb T^d_N\;,
\end{gathered}
\end{equation}
where we omitted the dependence on $N$ of all these variables.  Denote
by $G_a$, $G_{1,j} \colon \bb T^d_N \to \bb R$, $0\le a\le 3$,
$1\le j\le d$, the functions given by
\begin{equation*}
\begin{gathered}
{\color{blue} G_0 (x)} \;:=\; \gamma\,[ \beta(x-e_1)
\,+\,  \beta(x+e_1) \,]
\,-\, \beta(x) \,-\, \gamma^2 \beta(x-e_1)\, \beta(x)\,
\beta(x+e_1) \;,
\\
{\color{blue} G_1(x) \,=\, G_{1,1} (x)} \,:=\, 
G_3(x) \, -\, (N^2/K)  \,
\frac{[\, \varrho(x+e_1) - \varrho(x)\,]^2}{\chi(x) \, \chi(x+e_1)}
\;,
\\
{\color{blue} G_{1,j} (x)} \,:=\, 
-\, (N^2/K)  \,
\frac{[\, \varrho(x+e_j) - \varrho(x)\,]^2}{\chi(x) \, \chi(x+e_j)}
\;, \quad 
{\color{blue} G_2(x)} \,:=\,
-\, 4\, \gamma^2\, \beta(x) \, \chi (x)^{-1} \;,
\end{gathered}
\end{equation*}
$2\le j\le d$, and
\begin{equation*}
{\color{blue} G_3(x)} \;:=\; 2\, \gamma\,
\Big(\, \frac{1}{\chi(x)} \,+\,
\frac{1}{\chi (x+e_1)} \, \Big)
\,-\, 2\, \gamma^2 \,
\Big(\, \frac{\beta(x-e_1)\, \beta(x)}{\chi(x)} \,+\,
\frac{\beta(x+e_1)\, \beta(x+2e_1)}{\chi(x+e_1)} \, \Big)\;.
\end{equation*} 

\begin{lemma}
\label{2-l07}
For all $N\ge 1$, 
\begin{equation}
\label{07}
L^*_N \mb 1\; =\;  U \;+\;  K\, \Upsilon \;,
\end{equation}
where
\begin{equation}
\label{2-08}
\begin{gathered}
{\color{blue} U (\eta)} \,:=\,
\sum_{x\in \bb T^d_N} \big\{  (\Delta_N \varrho) (x)\,
+\, K\, G_0(x)\, \big\}  \, \zeta_x 
\\
{\color{blue} \Upsilon (\eta)} \,:=\,
\sum_{j=1}^d \sum_{x\in \bb T^d_N}  G_{1,j}(x) \, \omega_{x} \, \omega_{x+e_j}
\,+\,
\sum_{x\in \bb T^d_N}  G_2(x) \, \omega_{x-e_1} \, \omega_{x} \,
\omega_{x+e_1} \;,
\end{gathered}
\end{equation}
and $\Delta_N$ stands for the discrete Laplacian, $(\Delta_NG)(x) =
N^2 \sum_{1\le j \le d} [ G(x+e_j) + G(x-e_j)  - 2 G(x) ]$. 
\end{lemma}

\begin{proof}
As $\nu^N_{\varrho (\cdot)}$ is the product probability measure on $\Omega_N$ associated to a
density profile $\varrho (\cdot)$,
\begin{equation*}
\frac{\nu^N_{\varrho (\cdot)}(T^{x,x+e_j} \eta)}{\nu^N_{\varrho (\cdot)}(\eta) } 
\,-\, 1  \;=\;
-\, [\, \varrho (x+e_j) - \varrho(x)\,] \,
\{\, \zeta_{x+e_j} \,-\, \zeta_{x}  \,\} \;-\;
[\, \varrho(x+e_j) - \varrho(x)\,]^2 \, \zeta_{x} \, \zeta_{x+e_j} 
\;,
\end{equation*}
where $\zeta_y$ has been introduced in \eqref{2-20}.  Hence, in the
case of the symmetric simple exclusion process, an integration by
parts yields that
\begin{equation}
\label{2-21}
N^2 (L^{E,*}_N \mb 1)(\eta) \;=\;
\sum_{x\in \bb T^d_N}  (\Delta_N \varrho) (x)\, \zeta_{x}
\,-\, N^2\, \sum_{j=1}^d  \sum_{x\in \bb T^d_N}
[\, \varrho(x+e_j) - \varrho(x)\,]^2 \, \zeta_{x} \, \zeta_{x+e_j} \;.
\end{equation}

We turn to the Glauber dynamics for which
\begin{equation*}
\frac{\nu^N_{\varrho (\cdot)}(T^{x} \eta)}{\nu^N_{\varrho (\cdot)}(\eta) } 
\;=\;  1  \;+\;  [\, 1 - 2 \, \varrho(x)\,]\,  \zeta_{x} \;.
\end{equation*}
Replacing this ratio in \eqref{02} yields that
\begin{equation}
\label{03}
(L^{G,*}_N \mb 1)(\eta) \;=\;
\sum_{x\in \bb T^d_N} \big\{\, c_0(\tau_x T^x \eta) \,
\,-\, c_0(\tau_x \eta)  \, \big\} 
\;+\; \sum_{x\in \bb T^d_N}
[\, 1 - 2 \, \varrho(x)\,]\,  \zeta_{x}
\, c_0(\tau_x T^x \eta) \;.
\end{equation}
Recall from \eqref{01} the definition of the jump rate $c_0(\cdot)$,
and that $\sigma_x = 2\eta_x -1 \in \{-1,1\}$, $x\in \bb T^d_N$. In
these coordinates, $\zeta_x = \hat \sigma_x/2 \chi(x)$, where
$\color{blue} \hat \sigma_x = \sigma_x - \beta (x)$, and $\beta(x)$
has been introduced in \eqref{2-20}.

The first term on the right-hand side of \eqref{03} is equal to
\begin{equation}
\label{04}
4\gamma \sum_{x\in \bb T^d_N} \big\{\, \beta(x)\, \beta(x+e_1) \,+\,
\hat\sigma_x \,[ \, \beta(x-e_1)\, +\, \beta(x+e_1)\,]\, +\,
\hat\sigma_x \hat\sigma_{x+e_1}\,\big\}\;,
\end{equation}
and the second one is equal to
\begin{equation}
\label{05}
-\, 2\, \sum_{x\in \bb T^d_N} \frac{\beta(x)}{1-\beta(x)^2} \,
\hat\sigma_x\,
\Big\{\, 1 \,+\, \gamma\,\sigma_x\, (\sigma_{x-e_1} +\sigma_{x+e_1})
\,+\,
\gamma^2 \, \sigma_{x-e_1} \, \sigma_{x+e_1}\,\Big\}\;.
\end{equation}
In this last term, write $\sigma_x$ as $\hat\sigma_x + \beta(x)$ and
observe that
$\hat\sigma_x^2 = 4 \chi(x) \,-\, 2 \beta(x) \, \hat\sigma_x$, to
conclude that
\begin{equation}
\label{06}
\hat\sigma_x\,
\Big\{\, 1 \,+\, \gamma\,\sigma_x\, (\sigma_{x-e_1} +\sigma_{x+e_1})
\,+\,
\gamma^2 \, \sigma_{x-e_1} \, \sigma_{x+e_1}\,\Big\} \;=\;
\sum_{p=0}^3 I_p(x)\;, 
\end{equation}
where $I_p(x)$ accounts for the monomials
of degree $p$. The term of degree $0$ is easy to compute and is given
by
\begin{equation*}
I_0(x) \;=\; 4\,
\gamma \, \{\, \beta(x-e_1) + \beta(x+e_1) \,\} \, \chi(x) \;.
\end{equation*}
A long but straightforward calculation gives that
\begin{equation*}
\begin{aligned}
I_1(x) \; & =\; \Big\{ \, 1 \,-\,  \gamma \, \beta(x) \, [ \beta(x-e_1)
\,+\,  \beta(x+e_1)\,] \,
\,+\, \gamma^2 \, \beta(x-e_1\, )\, \beta(x+e_1) \, \Big\} \,
\hat\sigma_x \\
\; & +\; 4\, \gamma\, \chi(x) \, [ \hat\sigma_{x-e_1}
\,+\, \hat\sigma_{x+e_1} \, ]\;.
\end{aligned}
\end{equation*}
The terms of degree $2$ and $3$ are given by
\begin{equation*}
\begin{gathered}
I_2(x) \; =\; \gamma^2 \, \hat\sigma_{x}
\, [ \beta(x+e_1)\, \hat\sigma_{x-e_1}
\,+\,  \beta(x-e_1)\, \hat\sigma_{x+e_1} \,]
\;-\; \gamma \, \beta(x) \, \hat\sigma_{x}
\, [ \, \hat\sigma_{x-e_1} \,+\,  \hat\sigma_{x+e_1} \,]\;, \\
I_3(x) \; =\; \gamma^2 \, \hat\sigma_{x-e_1} \, \hat\sigma_{x}
\, \hat\sigma_{x+e_1} \;.
\end{gathered}
\end{equation*}

By \eqref{06} and the previous formulas for $I_p(x)$, adding
\eqref{04} and \eqref{05} yields that
\begin{equation*}
(L^{G,*}_N \mb 1)(\eta) \;=\;
\sum_{x\in \bb T^d_N}  \big\{\, a_0(x)\, \hat\sigma_{x} \,+\,
a_1(x)\, \hat\sigma_{x} \, \hat\sigma_{x+e_1}
\,+\,
a_2(x)\, \hat\sigma_{x-e_1} \, \hat\sigma_{x} \, \hat\sigma_{x+e_1}\,\big\}\;,
\end{equation*}
where
\begin{equation*}
\begin{gathered}
a_0(x) \;=\; \frac{1}{2\, \chi(x)} \,
\Big\{\, \gamma\,[ \beta(x-e_1)\,+\,  \beta(x+e_1) \,] \,-\,
\beta(x) \,-\, \gamma^2 \beta(x-e_1)\, \beta(x)\,
\beta(x+e_1)\,\Big\}\;, \\
a_1(x) \;=\; \frac{\gamma}{2}\, \Big(\, \frac{1}{\chi(x)} \,+\,
\frac{1}{\chi (x+e_1)} \, \Big)
\,-\, \frac{\gamma^2}{2} \,
\Big(\, \frac{\beta(x-e_1)\, \beta(x)}{\chi(x)} \,+\,
\frac{\beta(x+e_1)\, \beta(x+2e_1)}{\chi(x+e_1)} \, \Big)\;, \\
a_2(x) \;=\; -\, \gamma^2\, \frac{\beta(x)}{2\, \chi(x)}\;.
\end{gathered}
\end{equation*}
To complete the proof of the lemma, 
it remains to replace $\hat\sigma_{y}$ by $2\, [\,\eta_y -
\varrho(y)\,]$, and recall equation \eqref{2-21}.
\end{proof}

\begin{remark}
Given a density profile $\varrho\colon \bb T^d_N \to (0,1)$, a
function $f\colon \Omega_N \to \bb R$ is said to be of
$\varrho$-degree $k$, where $k\ge 0$, (or simply of degree $k$) if
there exist $p\ge 1$, subsets $A_1, \dots, A_p$ of $\bb T^d_N$, each
one with $k$ elements, and constants $c_1, \dots c_p$ such that
\begin{equation*}
f (\eta) \,=\, \sum_{i=1}^p c_i\, \prod_{x\in A_i} \omega_x\;.
\end{equation*}

In the previous lemma, the term of degree $0$ vanishes. On the other
hand, if $\varrho(\cdot)$ (and therefore $\beta$) is constant, then
\begin{equation*}
a_0(x) \,=\, \frac{-\, 2}{1-\beta^2} \, \big\{\, (1-2\gamma) \beta + \gamma^2
\beta^3 \,\big\}
\,=\, \frac{-\, 1}{ 2\rho(1-\rho)} \, V'(\rho) \;.
\end{equation*}
In particular, this term vanishes
if the computation is performed with respect to a Bernoulli product
measure at a critical density. Similarly, if $\beta$ is constant,
\begin{equation*}
a_1(x) \,=\, \frac{4\gamma}{1-\beta^2} \, \big\{\, 1\,-\, \gamma\,
\beta^2\,\big\} \;.
\end{equation*}
If $\gamma=1$, the stable densities are $\pm 1$, while $0$ is an
unstable one. $\triangle$
\end{remark}

The term of order one in the computation of $L^*_N \mb 1$ reveals the
reference density profile. To cancel the order one term, we have to
choose $\varrho(\cdot)$ which satisfies
\begin{equation}
\label{2-93}
(\Delta_N \varrho) (x)\, +\, K\, G_0(x) \,=\, 0\;, \quad x\in \bb
T^d_N\;. 
\end{equation}
By Lemma \ref{2-l18} the density profile $\bs u^N$ introduced in
\eqref{2-53a}, \eqref{2-52} almost fulfills this identity in the sense
that
\begin{equation}
\label{2-39}
\sum_{x\in \bb T^d_N} \big|\, (\Delta_N \bs u_N)(x) + K
G_0(x)\,\big|\;\le\; C_0\,  K^2\, N^{d-2}
\end{equation}
for some finite constant $C_0$, and all
$N\ge 1$. Actually, by the proof of Lemma \ref{2-l18}, whose main part
consists in showing that $G_0$ is close to $-V'(\bs u^N)$,
$C_0 = C_1 \, \sum_{1\le j\le 4} (1 + \Vert
\partial^j_\vartheta \rho^K \Vert_\infty) \, \Vert
\partial^j_\vartheta \rho^K \Vert_\infty$ for some finite constant
$C_1$ independent of $N$. By \eqref{2-84}, $C_0<\infty$.

Assume from now on that $\varrho (\cdot)= \bs u^N(\cdot)$.  Recall
from \eqref{2-12} that $\nu^N = \nu^N_{\bs u^N(\cdot)}$ represents the
product measure associated to the density profile $\bs u^N$, $f^N_t$
the density of the measure $d\nu^N S^N(t)$ with respect to $d\nu^N$,
and $H_N(f^N_t)$ the relative entropy introduced in \eqref{2-34d} or
in \eqref{2-34b} for $h^N_t = f^N_t$. With this choice, by
\eqref{2-39} and \eqref{2-38b}, there exists a constant
$\mf c(\bs u^N)$, depending only on $\bs u^N$ and uniformly bounded in
$N$, such that
\begin{equation}
\label{2-40}
\sup_{\eta\in\Omega_N} \,|U(\eta)| \,\le\,  \mf c(\bs u^N)\,  K^2\, N^{d-2}
\end{equation}
for all $N\ge 1$.

Since $\bs u^N(\cdot)$ is translation invariant in the directions
$e_j$, $2\le j\le d$, $G_{1,j}(\cdot)=0$ for $2\le j\le d$. Recall
from Lemma \ref{2-l07} that $G_1(\cdot) = G_{1,1}(\cdot)$.  By
Proposition \ref{l06-1}, Lemma \ref{2-l07}, and \eqref{2-39}, to
estimate the entropy production $(d/dt) \, H_N(f^N_t)$ we have to
evaluate
\begin{equation*}
K \int \Upsilon \, f^N_t\, d\nu^N \;=\;
K \int \sum_{x\in \bb T^d_N}  \{\,
G_1(x) \, \omega_{x} \, \omega_{x+e_1}
\,+\, G_2(x) \, \omega_{x-e_1} \, \omega_{x} \,
\omega_{x+e_1}\} \,  f^N_t\, d\nu^N \;.
\end{equation*}
Since similar quantities appear in the proof of the Boltzmann-Gibbs
principle, presented in Section \ref{sec-BG}, the analysis is carried
out for a general function $G\colon \bb T^d_N \to \bb R$ and a general
cylinder function $\prod_{x\in A} \omega_x$. This is the content of
the next section.

\section{Proof of Theorem \ref{mt1}}
\label{sec8}

The proof of Theorem \ref{mt1} is based Theorem \ref{t02} below whose
statement requires some notation.  For a finite subset $A$ of
$\bb Z^d$, let 
\begin{equation}
\label{x2}
{\color{blue} x+A } \,:=\,  \{x+y : y\in A\}\;, \quad 
{\color{blue}  \omega_{A} }  \;:=\; \prod_{x\in A} \, \omega_{x} \;.
\end{equation}
Fix a finite subset $A$ of $\bb Z^d$ with at least two elements and a
function $G\colon \bb T^d_N \to \bb R$.  Let $\mc W$,
be given by
\begin{equation}
\label{06a}
{\color{blue} \mc W (\eta)} \,=\, \mc W^{G,A} (\eta)
\;:=\; \sum_{x \in \bb T^d_N} G(x) \, \omega_{x+A}\;.
\end{equation}

Let $(g_d(\ell) : \ell\ge 1)$,  $(s_d(\ell) : \ell\ge 1)$, be the
sequences given by
\begin{equation}
\label{2-18}
\begin{gathered}
{\color{blue} g_d(\ell)}\,:=\, 
\left\{
\begin{array}{c@{\;,\quad}l}
\ell & d=1\\
\log \ell & d=2\\
1 & d \geq 3\;,
\end{array}
\right.
\qquad {\color{blue}  s_d(\ell) } \;=\; \ell^{d} \, g_d(\ell) \;,
\end{gathered}
\end{equation}
The sequence $(\ell_N:N\ge 1)$ introduced in the statement of Theorem
\ref{mt1} can be defined as the sequence which fulfills 
\begin{equation}
\label{2-18b}
{\color{blue} \ell_N^{d/2} \, s_d(\ell_N) } \;=\; \ell_N^{3d/2} g_d(\ell_N) \;=\;
N^{2} \;.
\end{equation}

For a real function $G$ defined on a set $\Lambda$ (which can be
$\bb T^d_N$, $\sqrt{K}\, \bb T \times \bb T^{d-1}$ or
$\bb R \times \bb T^{d-1}$), denote by $\Vert G\Vert_\infty$ its sup
norm:
$\color{blue} \Vert G\Vert_\infty = \sup_{z\in \Lambda} | G(z)|$.
Recall from \eqref{2-12} that $\nu^N = \nu^N_{\bs u(\cdot)}$.

\begin{theorem} 
\label{t02}
Fix a function $G\colon \bb T^d_N \to \bb R$, and a finite subset $A$
of $\bb Z^d$ with at least two elements. Assume that
$\sup_{N\ge 1} K^2_N / \ell^d_N < \infty$. Then, there exists a finite
constant $\mf c(\bs u^N)$ which depends only on $\bs u^N$, the set
$A$, the dimension $d$, and is uniformly bounded in $N$, such that
\begin{equation*}
K_N\, \int  \mc W^{G,A}
\, f\; d\nu^N \; \le\;
\delta \, N^2\, I^E_{N}(f;\nu^N)
\;+\; \mf c(\bs u^N)\, \frac{1}{\delta}\, \{\, 1 + \Vert
G\Vert^2_\infty\,\} \, K^2_N
\, \big\{ H_N(f) \,+\, (N/\ell_N)^d \, \big\} 
\end{equation*}
for all $0<\delta<1$, density $f$ with respect to $\nu^N$ and
$N\ge 1$.
\end{theorem}

We first complete the proof of Theorem \ref{mt1}.

\begin{proof}[Proof of Theorem \ref{mt1}]
Apply Proposition \ref{l06-1} and Lemma \ref{2-l07} for
$\varrho(\cdot) = \bs u^N(\cdot)$ and $h^N_t = f^N_t$. By
\eqref{2-53a}, $G_{1,j} =0$ for $j\neq 1$. By \eqref{2-40} and the
inequality $N^{d-2} \le (N/\ell_N)^d$ cf. \eqref{2-28}, which follows
from the definition of $\ell_N$,
$\max_{\eta\in\Omega_N} |U(\eta)| \le \mf c(\bs u^N) K^2
(N/\ell_N)^d$.  Apply Theorem \ref{t02} with $\delta=1$ and
$(G,A) = (G_{1,1}, \{0,e_1\})$, $(G,A) = (G_{2}, \{-e_1,0,e_1\})$, to
get that
\begin{equation*}
\frac{d}{dt}\, H_N(f^N_t) \;\leq\; 
\mf c(\bs u^N)\, \Big\{\, 2\, +\, \Vert G_{1,1}\Vert^2_\infty
\, +\, \Vert G_{2}\Vert^2_\infty \,\Big\} \, K^2_N
\, \big\{ H_N(f^N_t) \,+\, (N/\ell_N)^d \, \big\} \;.
\end{equation*}
By \eqref{2-53a}, \eqref{2-52}, \eqref{2-38b},  \eqref{2-84} the
sequence $\Vert G_{1,1}\Vert_\infty + \Vert G_{2}\Vert_\infty $ is
uniformly (in $N$) bounded.
\end{proof}

We turn to the proof of Theorem \ref{t02}. Fix till the end of this
section a function $G\colon \bb T^d_N \to \bb R$, and a finite subset
$A$ of $\bb Z^d$ with at least two elements.  The proof is divided in
two steps. We first show that we may average $\omega_x$ in a cube of
size $\ell = \ell_N$ at a cost which is bounded by the sum of the
integrated carr\'e du champ $I^E_{N}(f;\nu^N)$ with sums of squares of
the local empirical density over these mesoscopic boxes. This estimate
is proved in the next section. Then, we estimate these sums of squares
by the entropy and small terms using the entropy inequality and
concentration inequalities.

Let
$m_\ell$, $\ell\ge 1$, be the uniform measure on the cube
$\color{blue} \Lambda_\ell \,:= \, \{0 \,,\, 1\,,\, \dots
\,,\, (\ell-1)\}^d$:
\begin{equation*}
{\color{blue} m_\ell(z)} \;:=\; \frac{1}{ \ell^{d}}
\, \chi_{_{\Lambda_\ell}} (z)\;,
\end{equation*}
where $\color{blue} \chi_{_A}$ stands for the indicator of the set
$A$. Denote by $m^{(2)}_\ell$ be the convolution of $m_\ell$ with itself: 
\begin{equation*}
m^{(2)}_\ell(z) = \sum_{y \in \bb Z^d}
m_\ell(y) \, m_\ell(z-y) \;.
\end{equation*}
Notice that $m^{(2)}_\ell$ is supported on the cube
$\Lambda_{2(\ell-1)}$.  Denote by $\omega_x^{\ell}$ the average of
$\omega_{x+z}$ with respect to the measure $m^{(2)}_\ell$:
\begin{equation}
\label{20}
{\color{blue} \omega_x^{\ell}}
\;:=\; \sum_{y \in \bb T^d_N} m^{(2)}_\ell(y) \, \omega_{x+y}
\;=\; \sum_{y \in \Lambda_{2(\ell-1)}} m^{(2)}_\ell(y) \, \omega_{x+y}\;.
\end{equation}

Consider the partial order $\prec$ on $\bb Z^d$ defined by
$(x_{1}, \dots, x_{d}) \prec (y_{1}, \dots, y_{d})$ if $x_j\le y_j$
for all $1\le j\le d$.  Let $B = \{\mb x^{(1)}, \dots , \mb x^{(p)}\}$
be a finite subset of $\bb Z^d$ with at least two elements.  A point
$\mb x^{(k)}$ in $B$ is said to be \emph{maximal} if
$\mb x^{(k)} \prec \mb x^{(j)}$ entails that
$\mb x^{(k)} = \mb x^{(j)}$. The set $B$ has at least one maximal
element, which may not be unique. Take a maximal
element $\color{blue} \mb x_B$ of $B$, and let
$\color{blue} B_\star = B \setminus \{\mb x_B\}$.

Let $\mc W^{(1)}_\ell$, $\mc W^{(2)}_\ell\colon \Omega_N \to \bb R$ be
given by
\begin{equation}
\label{2-04} 
\begin{gathered}
{\color{blue} \mc W^{(1)}_\ell(\eta)} \;:=\; \sum_{x \in \bb T^d_N}
G(x) \, \omega_{x+A_\star} \, \omega^\ell_{x+ \mb x_A} \;,
\\
{\color{blue} \mc W^{(2)}_\ell(\eta)} \;:=\;
\sum_{x \in \bb T^d_N}
G(x) \, \omega_{x+A_\star} \, \{\,  \omega_{x+ \mb x_A} -
\omega^\ell_{x+ \mb x_A}\,\} \;, 
\end{gathered}
\end{equation}
so that
\begin{equation}
\label{2-29}
\mc W(\eta) \;=\; \mc W^{(1)}_\ell(\eta)
\,+\,\mc W^{(2)}_\ell(\eta) \;.
\end{equation}

We first examine $\mc W^{(2)}_\ell$.  The first bound requires further
notation.

\smallskip\noindent{\bf Flows.}
Let $\ms G$ be a finite set. For probability measures $\mu$ and $\nu$ on
$\ms G$, a function $\Phi:\ms G\times \ms G\to\bb R$ is called a {\it flow
  connecting $\mu$ to $\nu$} if
\begin{enumerate}
\item $\Phi$ is anti-symmetric: $\Phi(x,y) \,=\, -\, \Phi(y,x)$ for
all $x,y\in \ms G$;
\item The divergence of $\Phi$ at $x$ is $\mu(x) - \nu(x)$:
$\sum_{y\in \ms G} \Phi(x,y) = \mu(x) - \nu(x)$ for all $x\in \ms G$.
\end{enumerate}

The next result is \cite[Theorem 3.9]{jm1}.  

\begin{lemma}
\label{l01}
There exist a finite constant $C_d$, depending only on the dimension
$d$, and, for all $\ell\ge 1$, a flow $\Phi_\ell$ connecting the Dirac
measure at the origin to the measure $m^{(2)}_\ell$. This flow is
supported in $\Lambda_{2(\ell-1)}$ and on nearest-neighbour bonds. That
is,
\begin{equation*}
\Phi_\ell(x,y) \;=\;0
\end{equation*}
if $\Vert y - x \Vert \not = 1$ or if $\{x,y\} \not\subset
\Lambda_{2(\ell-1)}$. 
Moreover,
\begin{equation*}
\sum_{j=1}^d  \sum_{x\in\bb Z^d}\Phi_\ell(x,x+e_j)^2
\;\le\; C_d\; g_d(\ell)\;,
\end{equation*}
where $(g_d(\ell) : \ell\ge 1)$ is the sequence introduced in
\eqref{2-18}. 
\end{lemma}

For $1\le k\le d$, define $H^{(\ell)}_{k}\colon \bb T^d_N \to \bb R$ by
\begin{equation}
\label{2-06a}
{\color{blue} H^{(\ell)}_{k,x}} \;:=\;
H^{\ell, G, A}_{k} (x) \,=\, \sum_{\{y,y+e_k\} \subset \Lambda_{2(\ell-1)}}
\Phi_\ell (y,y+e_k) \, G(\, x-\mb x_A-y \,) \,
\omega_{x-\mb x_A-y+A_\star} \;,
\end{equation}
where $\Phi_\ell$ is the flow introduced in Lemma \ref{l01}.

\begin{proposition}
\label{2-l05}
There exists a finite constant $\mf c(\bs u^N)$, depending only on
$\bs u^N$, $A$ and the dimension $d$, uniformly bounded in $N$
($\sup_N \mf c(\bs u^N) < \infty$), such that
\begin{equation*}
\begin{aligned}
& K\, \int  \mc W^{(2)}_\ell \, f\; d\nu^N \; \le\;
\delta\, N^2 \, I^E_{N}(f; \nu^N)
\; +\; \mf c(\bs u^N)  \, K\, \sum_{x\in\bb T_N^d}
\int (\omega^{\ell} _{x})^2 \, f \; d\nu^N
\\
& \;+\; \mf c(\bs u^N)  \, 
\Big(\frac{K}{N}\Big)^2 \Big\{ \, \frac{1}{\delta}\, \Big[ \, 1 \,+\, 
\frac{K\, s_d(\ell) \,}{N^2} \,\Big] \,+ \,  (1+\lambda)  \Big\} \, \,
\int \sum_{k=1}^d \sum_{x\in\bb T_N^d} (H^{(\ell)}_{k,x})^2 \, f  \;
d\nu^N
\\
& \quad \;+\; \mf c (\bs u^N)\, K^2\, N^{d-2}
\,\Big\{\, \delta \,+\, \frac{s_d(\ell)}{\lambda} \,\Big\} 
\end{aligned}
\end{equation*}
for all $0<\delta<1$, $\lambda>0$, density $f$ with respect to $\nu^N$
and $N> 4\ell$. In this formula, $s_d(\ell)$ is the sequence
introduced in \eqref{2-18}.
\end{proposition}

This result is proved in Section \ref{sec4}. We comment that the
inhomogeneity of $\nu^N$ will play a strong role in the derivation.
Proposition \ref{2-l05} is one of the main steps in the proof of the
Boltzmann-Gibbs principle, the subject of Section \ref{sec-BG}.

In the application, $\ell_N$ will be chosen according to \eqref{2-28},
$N^2 = s_d(\ell)\, \ell^{d/2}$, and $\lambda$ as $\ell_N^{d/2}$.  With
these choices, it will be useful to give an alternate
version of Proposition \ref{2-l05}. Assume that
$\sup_N (K^2_N/\ell^d_N)<\infty$.  Then, there exists a finite
constant $\mf c(\bs u^N)$, depending only on $\bs u^N$, $A$, $d$ and
uniformly bounded in $N$, such that
\begin{equation}
\label{2-41}
\begin{aligned}
& K\, \int  \mc W^{(2)}_\ell \, f\; d\nu^N \; \le\;
\delta\, N^2 \, I^E_{N}(f; \nu^N)
\; +\; \mf c (\bs u^N)  \, K\, \sum_{x\in\bb T_N^d}
\int (\omega^{\ell} _{x})^2 \, f \; d\nu^N
\\
& \;+\; \mf c(\bs u^N)   \,  \frac{1}{\delta}\, 
\frac{K^2}{s_d(\ell)} \,
\int \sum_{k=1}^d \sum_{x\in\bb T_N^d} (H^{(\ell)}_{k,x})^2 \, f  \;
d\nu^N
\;+\; \mf c (\bs u^N)  \, K^2\, (N/\ell)^{d} \,
\end{aligned}
\end{equation}
for all $0<\delta<1$, $N\ge 1$ and density $f$ with respect to
$\nu^N$. In this formula, $\ell = \ell_N$.

At this point, we use the entropy inequality and a concentration
inequality to estimate the integral of $(H^{(\ell)}_{k,x})^2 $ and
$(\omega^{\ell} _{x})^2$. The next result follows from the proof of
Proposition 5.2 (page 4179) and Lemma 5.7 in \cite{jl}.

\begin{lemma}
\label{2-l08}
There exists a finite constant $\mf c(\bs u^N)$, depending
only on $\bs u^N$, $A$, $d$ and uniformly bounded in $N$, such that
\begin{equation*}
\begin{gathered}
\int \sum_{k=1}^d \sum_{x\in\bb T_N^d} (H^{(\ell)}_{k,x})^2 \, f  \;
d\nu^N \;\le \; 
\mf c (\bs u^N)\, \Vert G\Vert^2_\infty \, s_d(\ell)
\, \big\{ H_N(f) \,+\, (N/\ell)^d\, \big\}\;,
\\
\int \sum_{x\in\bb T_N^d} (\omega^{\ell} _{x})^2 \, f  \;
d\nu^N \;\le
\mf c (\bs u^N) 
\, \big\{ H_N(f) \,+\, (N/\ell)^d\, \big\}\;.
\end{gathered}
\end{equation*}
for all $N\ge 1$, $0<\ell<N/4$, and density $f$ with respect to $\nu^N$.
\end{lemma}

Reporting these estimates in Proposition \ref{2-l05} yields a bound
for the integral of $\mc W^{(2)}_\ell$ with respect to the measure
$f\; d\nu^N$. This is the content of the next result.

\begin{lemma}
\label{2-l06}
Let $(\ell_N:N\ge 1)$ be the sequence introduced in the statement of
Theorem \ref{mt1}.  Assume that $\sup_N (K^2_N/\ell^d_N)<\infty$.
Then, there exists a finite constant $\mf c(\bs u^N)$, depending only
on $\bs u^N$, $A$, $d$ and uniformly bounded in $N$, such that
\begin{equation*}
K\, \int  \mc W^{(2)}_\ell \, f\; d\nu^N \; \le\;
\delta\, N^2 \, I^E_{N}(f; \nu^N) 
\;+\; \mf c(\bs u^N)\, \frac{1}{\delta} \, \{\, 1 + \Vert
G\Vert^2_\infty \,\} \,
K^2 \, \big\{ H_N(f) \,+\, (N/\ell_N)^d\, \big\} 
\end{equation*}
for all $0<\delta<1$, $N\ge 1$ and density $f$ with respect to $\nu^N$.
\end{lemma}

\begin{proof}
Insert the estimates obtained in Lemma \ref{2-l08} at the right-hand
side of Proposition \ref{2-l05}.  Choose the constant $\lambda$ for
$\lambda \, (K/N)^2 \, s_d(\ell)\, (N/\ell)^{d}$ to be of the same
order as $\lambda^{-1}\, K^2\, N^{d-2} \, s_d(\ell)$, that is,
$\lambda= \ell^{d/2}$.

With this choice, the expression multiplying
$\{ H_N(f) \,+\, (N/\ell)^d\, \}$ is equal to
\begin{equation*}
\mf c(\bs u^N)  \, \Big\{ K \,+\, \Vert G\Vert^2_\infty\, s_d(\ell) \,
\Big(\frac{K}{N}\Big)^2 \,  \Big[ \, \frac{1}{\delta}\, \Big( \, 1+
s_d(\ell) \,
\frac{K}{N^2} \,\Big) \,+ \, 1 \,+\, \ell^{d/2} \,\Big]
\Big\}\;.
\end{equation*}
Choose $\ell_N$ according \eqref{2-28},
$N^2 = \ell^{d/2}\, s_d(\ell)$. By hypothesis,
$s_d(\ell) \, (K/N^2) \le C_0$. Thus, and since $\delta<1$, the
expression inside square brackets is bounded by
$C_0 \delta^{-1} \ell^{d/2}$, and the previous expression is less than
or equal to
\begin{equation*}
\mf c(\bs u^N)  \, \frac{1}{\delta}\, \Big\{ K \,+\,
\Vert G\Vert^2_\infty\, s_d(\ell) \,
\Big(\frac{K}{N}\Big)^2 \, \ell^{d/2} \, \Big\} \,=\,
\mf c(\bs u^N)  \, \frac{1}{\delta}\, \big\{ K \,+\, K^2\,
\Vert G\Vert_\infty \, \big\} \;\le\;
\mf c(\bs u^N)  \, \frac{1}{\delta}\, \Vert G\Vert^2_\infty \,
K^2\;,
\end{equation*}
where we used the identity $s_d(\ell) \, \ell^{d/2} =N^2$.

With the choice of $\lambda$, and since
$s_d(\ell)/\ell^{d/2} \ge \ell^{d/2} \to \infty$, the last term on the
right-hand side of the statement of Proposition \ref{2-l05} is given
by
\begin{equation*}
\mf c (\bs u^N)\, K^2\, N^{d-2}
\,\Big\{\, \delta \,+\, \frac{s_d(\ell)}{\ell^{d/2}} \,\Big\}
\;\le \;
\mf c (\bs u^N)\, K^2\, \frac{1}{N^{2}}\, 
\,s_d(\ell) \, \ell^{d/2} \,  (N/\ell)^d\;. 
\end{equation*}
By definition of $\ell_N$, as $(\ell_N^{d/2}s_d(\ell_N)=N^2$), this
expression is bounded by $\mf c (\bs u^N)\, K^2\, (N/\ell)^d$, which
completes the proof of the lemma.
\end{proof}

Lemma \ref{2-l06} completes the proof of the estimation of
$\mc W^{(2)}_\ell$. We turn to $\mc W^{(1)}_\ell$.

\begin{lemma}
\label{2-l09}
There exists a finite constant $\mf c(\bs u^N)$, depending
only on $\bs u^N$, $A$, the dimension $d$ and uniformly bounded in
$N$, such that
\begin{equation*}
\int  |\, \mc W^{(1)}_\ell\,| \, f\; d\nu^N \; \le\;
\mf c(\bs u^N)\, \Vert G\Vert_\infty \,
\big\{ H_N(f) \,+\, (N/\ell)^d\, \big\} 
\end{equation*}
for all density $f$ with respect to $\nu^N$ and $N> 4\ell$. 
\end{lemma}

\begin{proof}
By definition of the measure $m^{(2)}_\ell$, and a change of
variables,
\begin{equation}
\label{2-31}
\mc W^{(1)}_\ell (\eta) \;=\;
\sum_{x\in \bb T^d_N}  M^{(1)}_\ell (x) \, M^{(2)}_\ell(x) \;,
\end{equation}
where
\begin{equation*}
\begin{gathered}
M^{(1)}_\ell (x) \;=\;  \sum_{y\in \Lambda_{\ell-1}} m_\ell(y) \,
G (x-y) \, \omega_{x-y+A_\star} \;,
\\
M^{(2)}_\ell(x) \;=\; \sum_{z\in \Lambda_{\ell-1}}
m_\ell(z) \, \omega_{x+z+ \mb x_A}\;.
\end{gathered}
\end{equation*}
By Young's inequality, $2 ab \le \beta\, a^2 + \beta^{-1} b^2$, 
$\beta>0$,
\begin{equation*}
|\, \mc W^{(1)}_\ell\,|  \;\le\;
\frac{1}{2\,\beta }\, \sum_{x\in \bb T^d_N}  M^{(1)}_\ell (x)^2
\;+\;
\frac{\beta}{2}\, \sum_{x\in \bb T^d_N}  M^{(2)}_\ell (x)^2 
\end{equation*}
for all $\beta>0$.  By \cite[Lemma 5.7]{jl},
\begin{equation*}
\int  \sum_{x\in \bb T^d_N}  M^{(1)}_\ell (x)^2  \, f\; d\nu^N \; \le\;
\mf c(\bs u^N)\, \Vert G \Vert^2_\infty\,
\big\{ H_N(f) \,+\, (N/\ell)^d\, \big\} 
\end{equation*}
for a finite constant $\mf c (\bs u^N)$ depending only on $\bs u^N$,
$A$, the dimension $d$ and uniformly bounded in $N$. The same
inequality holds for $ M^{(1)}_\ell$ with $\Vert G \Vert^2_\infty$
replaced by $1$. Note that this result is proved in \cite{jl} in the
case where the density $\bs u^N(\cdot)$ is constant, but the same
proof applies provided $\bs u^N(\cdot)$ is bounded away from $0$ and
$1$.  To complete the proof of the lemma, chooses
$\beta = \Vert G \Vert_\infty$.
\end{proof}

\begin{proof}[Proof of Theorem \ref{t02}]
The result is a consequence of the decomposition \eqref{2-29} and
Lemmata \ref{2-l06}, \ref{2-l09}.
\end{proof}

\section{An energy estimate}
\label{sec4}

In this section, we prove Proposition \ref{2-l05}. Here,
$\color{blue} \mf c = \mf c (\bs u^N)$ represents a constant uniformly
bounded in $N$, whose value may changed from line to line and which
depends only on the sequence $(\bs u^N: N\ge 1)$, $A$ and the dimension
$d$.  We first use the definition of the flow to derive an alternative
formula for $\mc W^{(2)}_\ell$.

\begin{lemma}
\label{l2-01}
With the notation of Proposition \ref{2-l05},
\begin{equation*}
\mc W^{(2)}_\ell (\eta) \; = \; \sum_{k=1}^d  \sum_{x\in\bb T_N^d}
\{\, \omega_{x}  \,-\, \omega_{x+e_k} \,\} \, H^{(\ell)}_{k,x} \;,
\end{equation*}
where $H^{(\ell)}_{k,x}$ is given by \eqref{2-06a}.
\end{lemma}

\begin{proof}
Rewrite $\mc W^{(2)}_\ell$ in \eqref{2-04} as
\begin{equation*}
\mc W^{(2)}_\ell(\eta)\; =\; \sum_{x \in \bb T^d_N} G(x) \, \omega_{x+A_\star}
\sum_{y\in \Lambda_{2(\ell-1)}} \omega_{x+\mb x_A +y} \,
\{\, \delta_{0}(y)  \,-\,  m^{(2)}_\ell(y) \,\}\;,
\end{equation*}
where $\delta_0$ stands for the Dirac measure concentrated at $0$.

Denote by $\Phi_\ell$ the flow introduced in Lemma \ref{l01}. In the
previous equation, we may replace $\Lambda_{2(\ell-1)}$, on which
$\Phi_\ell$ is supported, by $\bb Z^d$. This simplifies the summation
by parts performed below. After this replacement, since the flow
connects $\delta_0$ to $m^{(2)}_\ell$, it is anti-symmetric and
supported on nearest-neighbour bonds, the sum over $y$ becomes
\begin{align*}
& \sum_{y\in \bb Z^d} \omega_{x+\mb x_A+y} \,
\sum_{z: \Vert z\Vert =1} \Phi_\ell (y,y+z) \\
&\quad \;=\; \sum_{k=1}^d \sum_{y\in \bb Z^d} \omega_{x+\mb x_A+y} \,
\{\, \Phi_\ell (y,y+e_k) \,-\,
\Phi_\ell (y-e_k,y)\,\} \;.
\end{align*}
Performing a summation by parts, this last sum becomes
\begin{equation*}
\sum_{k=1}^d \sum_{y\in \bb Z^d} \Phi_\ell (y,y+e_k)
\,
\{\, \omega_{x+\mb x_A+y}  \,-\, \omega_{x+\mb x_A+y+e_k} \,\}\;.
\end{equation*}
As $\Phi_\ell$ is supported on $\Lambda_{2(\ell-1)}$, we may restrict
the sum over $y$ to the set of all points in $\bb Z^d$ such that
$\{y,y+e_k\} \subset \Lambda_{2(\ell-1)}$.

Perform a change of variables $x'=x+\mb x_A+y$ and recall the
definition of $\omega_x$ to conclude that
\begin{equation*}
\mc W^{(2)}_\ell \; =\; \sum_{k=1}^d  \sum_{x\in\bb T_N^d}
\{\, \omega_{x}  \,-\, \omega_{x+e_k} \,\}
\sum_{\{y,y+e_k\} \subset \Lambda_{2(\ell-1)}} \Phi_\ell (y,y+e_k) \,
G(x-\mb x_A-y)\,  \omega_{x-\mb x_A-y+A_\star} \;.
\end{equation*}
To complete the proof of the assertion, it remains to recall the
definition of $H^{(\ell)}_{k,x}$ introduced in \eqref{2-06a}.
\end{proof}

\smallskip\noindent{\bf Integration by parts.}  We continue our path
to the proof of Proposition \ref{2-l05}, based on the formula for
$\mc W^{(2)}_\ell$ provided by Lemma \ref{l2-01}.  Since $\mb x_A$ is
a maximal point of $A$ and $\Lambda_\ell = \{0,1,\ldots, \ell-1\}^d$,
the support of $H^{(\ell)}_{k,x}$ is disjoint from $\{x,x+e_k\}$ in
the sense that the indices $z$ of $\omega_z$ which appear in the
definition of $H^{(\ell)}_{k,x}$ are different from $x$ and
$x+e_k$. In particular,
\begin{equation}
\label{15b}
H^{(\ell)}_{k,x} (T^{x, x+e_k} \eta) \;=\;
H^{(\ell)}_{k,x} (\eta)\;.
\end{equation}

The next lemma requires some notation.  For $x\in \bb T^d_N$,
$1\le j\le d$, let $I^E_{x,x+e_j}$ be the functional $I^E_N$ restricted to
the bond $\{x,x+e_j\}$:
\begin{equation*}
I^E_{x,x+e_j}(f) \;=\;  \frac{1}{2}\, \int
\Big\{\, \sqrt{ f(T^{x,x+e_j}\eta)}
\,-\, \sqrt{f(\eta) \vphantom{T^{x,x+e_j}} } \, \Big\}^2 \; d \nu^N
\;, \quad f\colon \Omega_N \to \bb R \;.
\end{equation*}
Moreover, let $A_x(\cdot)$, $x\in \bb T^d_N$, be the local function
given by
\begin{equation}
\label{2-05}
{\color{blue} A_x(\eta) } \;:=\; \frac{1}{2\chi(x)} \,
\Big\{ \, [1-2\bs u^N(x)]
\, \omega_{x+e_1}  \, +\, [1-2\bs u^N(x+e_1)] \, \omega_{x} \,-\,
2\, \omega_{x} \, \omega_{x+e_1} \Big\} \;.
\end{equation}
Mind that the last two terms on the right-hand side of the next lemma
vanish if the density profile $\bs u^N(\cdot)$ is constant.

\begin{lemma}
\label{l2}
Fix $1\le j\le d$. Then, there exists a finite constant
$\mf c(\bs u^N)$ such that
\begin{equation*}
\begin{aligned}
& \int h \, [\omega_x - \omega_{x+e_j}]\, f\; d\nu^N
\;\leq\; \beta  \, I^E_{x,x+e_j}(f)
\;+\;  \frac{\mf c(\bs u^N)}{\beta}  \, \int h^2 \, f  \; d\nu^N 
\\
&\quad \;-\;
[\bs u^N(x+e_j) - \bs u^N(x)] \,
\int  h(\eta)  \, A_x(\eta)  \,  f(\eta)  \; d\nu^N
\;+\;  \beta\, \mf c (\bs u^N)\, [\bs u^N(x+e_j) - \bs u^N(x)]^4
\end{aligned}
\end{equation*}
for all $\beta >0$, $x \in \bb T^d_N$, $h\colon \Omega_N \to \bb R$ such that
$h(T^{x,x+e_j}\eta) = h(\eta)$ for all $\eta \in \Omega_N$, and
density $f\colon \Omega_N \to [0,\infty)$ with respect to $\nu^N$.
\end{lemma}

\begin{proof}
Write the left-hand side as
\begin{equation}
\label{2-01}
\int h \, [\eta_x - \eta_{x+e_j}]\, f\; d\nu^N
\;-\; [\bs u^N(x) - \bs u^N(x+e_j)] \, \int h \, f\; d\nu^N \;.
\end{equation}

Since $h(T^{x,x+e_j}\eta) = h(\eta)$, the first term can be written as
\begin{equation}
\label{2-02}
\begin{aligned}
& \frac{1}{2} \,\int  h(\eta) \, [\eta_x - \eta_{x+e_j}]\, 
[f(\eta) - f(T^{x,x+e_j} \eta) ] \; d\nu^N \\
& \quad +\; \frac{1}{2} \,\int  h(\eta)
 \, [\eta_x- \eta_{x+e_j}] \,
f(\eta) \, \Big\{ 1 - \frac{\nu^N ( T^{x,x+e_j} \eta)}
{\nu^N( \eta)} \Big\}  \; d\nu^N\;.
\end{aligned}
\end{equation}

In the first integral, write $c(b-a)$ as
$c (\sqrt{b} - \sqrt{a}) (\sqrt{b} + \sqrt{a})$, where $c$, $b$, $a$
represent $(1/2) \, h(\eta) \, [\eta_x - \eta_{x+e_j}]$, $f(\eta)$,
$f(T^{x,x+e_j} \eta)$, respectively.  Bound this product by
$(\beta/2)\, (\sqrt{b} - \sqrt{a})^2 + (1/2\beta) \, c^2\, (\sqrt{b} +
\sqrt{a})^2$, $\beta>0$. The second term of this sum is less than or
equal to $(1/\beta) \, c^2\, (b+a)$. Bound $(\eta_x - \eta_{x+e_j})^2$
by $1$ to get that $c^2 \le (1/4) h(\eta)^2$. Perform the change of
variables $\xi = T^{x,x+e_j} \eta$, recall that
$h(T^{x,x+e_j}\eta) = h(\eta)$, to conclude that the first term in
\eqref{2-02} is bounded above by
\begin{equation*}
\beta  \, I^E_{x,x+e_j}(f)
\;+\; \frac{1}{4\, \beta} \, \int h^2 \,
f \, \Big\{\, 1 \,+\, 
\frac{\nu^N ( T^{x,x+e_j} \eta)} {\nu^N( \eta)} \Big\} \; d\nu^N\;. 
\end{equation*}
for all $\beta>0$.  As $\bs u^N(\cdot)$ is constant along the $j$-th
direction for $j\neq 1$, $\nu^N ( T^{x,x+e_j} \eta) = \nu^N( \eta)$ if
$j\neq 1$. On the other hand, as
\begin{equation}
\label{2-13}
\frac{\nu^N ( T^{x,y} \eta)} {\nu^N( \eta)} \;=\;
\exp\Big\{\, -  \, \Big[\, \log
\frac{1-\bs u^N(y)} {\bs u^N (y)} \,-\, \log
\frac{1-\bs u^N(x)} {\bs u^N (x)}\, \Big]\, 
\,[\eta_y - \eta_x] \,\Big\} \;,
\end{equation}
by \eqref{2-38b}, this ratio is bounded by $\mf c(\bs u^N)$. (We could
get a better estimate, but since this ratio is summed to $1$ in the
penultimate equation, there is no reason to improve it). In
conclusion, the first term in \eqref{2-02} is bounded above by
\begin{equation}
\label{n-02}
\beta  \, I^E_{x,x+e_j}(f)
\;+\; \frac{\mf c(\bs u^N)}{\beta} \, \int h^2 \,
f \,  \; d\nu^N
\end{equation}
for all $\beta>0$.

We turn to the second term in \eqref{2-02}. It vanishes for $j\neq 1$.
For $j=1$, by \eqref{2-13} and a Taylor expansion, the expression
inside braces is equal to
\begin{equation*}
\frac{1}{\chi(x)}\, [\bs u^N(x+e_1) - \bs u^N(x)]\, [\eta_{x+e_1} -
\eta_x] \; +\; R^{(2)}_N \;.
\end{equation*}
Here and below, $R^{(2)}_N$ represents an error whose value may change
from line to line and whose absolute value is bounded by
$\mf c (\bs u^N)\, [\bs u^N(x+e_1) - \bs u^N(x)]^2$. The second term in
\eqref{2-02} is thus bounded by
\begin{equation}
\label{n-01}
\frac{-\, 1}{2\chi(x)} \, [\bs u^N(x+e_1) - \bs u^N(x)] \,
\int  h(\eta)  \, [\eta_x- \eta_{x+e_1}]^2  \,
f(\eta)  \; d\nu^N \;+\;  
R^{(2)}_N \,  \int  \big|\, h(\eta)\,\big|   \, f(\eta)  \; d\nu^N\;.
\end{equation}
Since
$[\eta_x- \eta_{x+e_1}]^2 = \eta_x + \eta_{x+e_1} - 2
\eta_x \, \eta_{x+e_1}$,
\begin{equation*}
[\eta_x- \eta_{x+e_1}]^2 \;=\; 2 \chi(x) \, +\, [1-2\bs u^N(x)]
\, \omega_{x+e_1}  \, +\, [1-2\bs u^N(x+e_1)] \, \omega_{x} \,-\,
2\, \omega_{x} \omega_{x+e_1} \;+\; R^{(1)}_N\;,
\end{equation*}
where $R^{(1)}_N$ stands for an error whose absolute value is bounded
by $\mf c (\bs u^N)\, |\bs u^N(x+e_1) - \bs u^N(x)|$. 
Therefore, the first term in \eqref{n-01} is equal to
\begin{equation*}
-\, [\bs u^N(x+e_1) - \bs u^N(x)] \, \Big\{
\int  h(\eta)  \, \big\{ \, 1\,+\, A_x(\eta) \,\big\}
\,  f(\eta)  \; d\nu^N   \;+\;  R^{(1)}_N \, \frac{1}{2\chi(x)}
\int \big|\, h(\eta)\,\big|  \, f(\eta)  \; d\nu^N \Big\}
\end{equation*}
where $A_x$ has been introduced in \eqref{2-05}. As $\chi(x)$ is
bounded below by a strictly positive constant, uniformly in $N$, by
definition of $R^{(2)}_N$, this expression is equal to
\begin{equation}
\label{n-03}
-\, [\bs u^N(x+e_1) - \bs u^N(x)] \, 
\int  h(\eta)  \, \big\{ \, 1\,+\, A_x(\eta) \,\big\}
\,  f(\eta)  \; d\nu^N   \;+\;  R^{(2)}_N \, 
\int \big|\, h(\eta)\,\big|  \, f(\eta)  \; d\nu^N \;.
\end{equation}
The second term in this sum coincides with the last one in
\eqref{n-01}. Thus, the expression in \eqref{n-01} (and therefore the
second term in \eqref{2-02}) is bounded by the previous expression.

Up to this point, we proved that the first integral in \eqref{2-02} is
bounded by \eqref{n-02}, and the second one by \eqref{n-03}. 
The expression
\begin{equation*}
-\, [\bs u^N(x+e_1) - \bs u^N(x)] \, 
\int  h(\eta)  \,  f(\eta)  \; d\nu^N   
\end{equation*}
which appears in \eqref{n-03} cancels with the second term in
\eqref{2-01}. This proves that \eqref{2-01} is bounded by the sum of
\eqref{n-02} with 
\begin{equation*}
-\, [\bs u^N(x+e_1) - \bs u^N(x)] \, 
\int  h \, A_x 
\,  f  \; d\nu^N   \;+\; \frac{1}{\beta}\, 
\int h^2  \, f  \; d\nu^N \;+\; \beta\, \mf c(\bs u^N)\, 
[\bs u^N(x+e_1) - \bs u^N(x)]^4
\end{equation*}
for all $\beta>0$. We used in the last step Young's inequality and the
definition of $R^{(2)}_N $.
\end{proof}

By \eqref{15b}, the functions $H^{(\ell)}_{k,x}(\cdot)$ satisfy the
hypotheses of Lemma \ref{l2}. Hence, by Lemma \ref{l2-01} and
Lemma \ref{l2} applied to $h= H^{(\ell)}_{j,x}(\cdot)$, as
$\bs u^N(\cdot)$ is constant along the $k$-th direction for $k\neq 1$,
\begin{equation}
\label{2-11}
\begin{aligned}
& \int  \mc W^{(2)}_\ell \, f\; d\nu^N \; \le\;
\beta \, I^E_{N}(f;\nu^N)
\;+\;  \frac{\mf c(\bs u^N)}{\beta}  \,
\int \sum_{k=1}^d \sum_{x\in\bb T_N^d} (H^{(\ell)}_{k,x})^2 \, f  \; d\nu^N 
\\
&\quad \;-\;
\sum_{x\in\bb T_N^d} [\bs u^N(x+e_1) - \bs u^N(x)] \,
\int  H^{(\ell)}_{1,x} \, A_x  \,  f  \; d\nu^N
\;+\; \beta\, \mf c (\bs u^N)\,
\sum_{x\in\bb T_N^d} [\bs u^N(x+e_1) - \bs u^N(x)]^4
\end{aligned}
\end{equation}
for all $\beta>0$.

To complete the proof of Proposition \ref{2-l05}, it remains to estimate
the third term on the right-hand side of \eqref{2-11}. This is the
content of the next result.

Recall from \eqref{2-06a} the definition of $H^{(\ell)}_{k,x}$. Let
\begin{equation}
\label{2-27}
{\color{blue} a_4(x)} \,:=\, [\bs u^N(x+e_1) - \bs u^N(x)] / 2\chi(x)\;.
\end{equation}
and rewrite the contribution from the degree 2 part
of $A_x$ in the third term on the right-hand side of \eqref{2-11} as
\begin{align}
\label{2-25}
& \frac{1}{\chi(x)}
\sum_{x\in\bb T_N^d} [\bs u^N(x+e_1) - \bs u^N(x)]\, 
\,  H^{(\ell)}_{1,x} \, \omega_x \, \omega_{x+e_1} \\
& \quad =\;
2\, \sum_{x\in\bb T_N^d} \sum_{\{y,y+e_k\} \subset \Lambda_{2(\ell-1)}}
\Phi_\ell (y,y+e_k) \, a_4(x)\, G(\, x-\mb x_A-y \,) \,
\omega_x \, \omega_{x+e_1} \, \omega_{x-y + (A_\star - \mb x_A)}\;.
\nonumber
\end{align}
This expression is similar to the sum appearing in the definition
\eqref{06a} of $\mc W$. By \eqref{2-84} and \eqref{2-52}, we gain a
factor $\sqrt{K}/N$, implicit in $a_4(\cdot)$. We may, thus, iterate
the argument used to estimate $\mc W$.

Define
$M^{p, \ell}_{x}$, $1\le p\le 3$, by
\begin{equation}
\label{2-06b}
{\color{blue} M^{p, \ell}_{x} (\eta)} \;:=\;
\sum_{k=1}^d  \sum_{\{y,y+e_k\} \subset \Lambda_{2(\ell-1)}}
\Phi_\ell (y,y+e_k) \, m^{p, \ell}_{x,y} (\eta)  \;,
\end{equation}
where $\Phi_\ell$ is the flow introduced in Lemma \ref{l01}.
and
\begin{equation}
\label{2-24}
\begin{gathered}
\vphantom{\Big\{}
m^{1, \ell}_{x,y} (\eta)  \;=\; a_4(\, x-e_1-y \,) \,
H^{(\ell)}_{1,x-e_1-y} (\eta)\,  \omega_{x-e_1-y} \;, \\
\vphantom{\Big\{}
m^{2, \ell}_{x,y} (\eta)  \;=\; a_4(\, x-e_1-y \,) \,
H^{(\ell)}_{1,x-e_1-y} (\eta)\,  [\, 1-2\bs u^N(x-e_1-y) \,] \;, \\
\vphantom{\Big\{}
m^{3, \ell}_{x,y} (\eta)  \;=\; a_4(\, x-y \,) \,
H^{(\ell)}_{1,x-y} (\eta)\,  \,  [\, 1-2\bs u^N(x+e_1-y) \,] \;.
\end{gathered}
\end{equation}

In the following lemma, $m^{1, \ell}_{x,y}$ arises from the degree $2$
part of $A_x$, while $m^{2, \ell}_{x,y}$ and $m^{3, \ell}_{x,y}$ from
its degree $1$ part.

\begin{lemma}
\label{2-l02}
With the notation of Lemma \ref{l2},
\begin{equation*}
\begin{aligned}
& -\, \sum_{x\in\bb T_N^d} [\bs u^N(x+e_1) - \bs u^N(x)] \, \int
H^{(\ell)}_{1,x} \, A_x \, f \; d\nu^N \\
& \quad 
\;\le\;
3\, \beta  \, I^E_{N}(f;\nu^N)
\;+\; \beta\,  \mf c (\bs u^N)\, \sum_{x\in\bb T_N^d}
[\bs u^N(x+e_1) - \bs u^N(x)]^4 \\
& \quad +\; \mf c(\bs u^N)\, \sum_{k=1}^d
\sum_{p=1}^3 \sum_{x\in\bb T_N^d}
\Big\{ \frac{1}{\beta}
\int [M^{p, \ell}_{k,x}]^2 \, f  \; d\nu^N 
\;+\; |\, \bs u^N(x+e_1) - \bs u^N(x) \,| \,
\int  |M^{p, \ell}_{k,x}|  \,  f  \; d\nu^N \,\Big\} 
\\
&\quad +\; \mf c(\bs u^N)  \sum_{x\in\bb T_N^d}
\int
\Big\{\, \gamma \, [\bs u^N(x+e_1) - \bs u^N(x)]^2 \,
[H^{(\ell)}_{1,x}]^2 \, +\, 
\gamma^{-1}\, (\omega^{\ell} _{x})^2
\, \Big\}  \, f \; d\nu^N
\end{aligned}
\end{equation*}
for all $\beta>0$, $\gamma>0$. 
\end{lemma}

\begin{proof}
The function $A_x(\cdot)$ is the sum of three terms. To fix ideas,
consider the expression
\begin{equation}
\label{2-10}
2\, \int\, \sum_{x\in\bb T_N^d}  a_4(x) \, 
H^{(\ell)}_{1,x} (\eta) \, \omega_x\, \omega_{x+e_1} \, f(\eta) \; d\nu^N \;,
\end{equation}
where $a_4(x)$ has been defined in \eqref{2-27}. This term corresponds
to the integral of \eqref{2-25}. The other terms, from the degree $1$
part of $A_x$, are handled analogously.  Rewrite this integral as
\begin{equation}
\label{2-26}
\begin{aligned}
& 2\, \int \, \sum_{x\in\bb T_N^d}  a_4(x) \, 
H^{(\ell)}_{1,x} (\eta) \, \omega_x\,
\big\{\, \omega_{x+e_1} \,-\, \omega^{\ell}_{x+e_1} \,\big\}
\, f(\eta) \; d\nu^N \\
&\quad \;+\;
2\, \int\, \sum_{x\in\bb T_N^d}  a_4(x) \, 
H^{(\ell)}_{1,x} (\eta) \, \omega_x\,
\omega^{\ell} _{x+e_1}  \, f(\eta) \; d\nu^N \;.
\end{aligned}
\end{equation}
By Young's inequality and the definition of $a_4(x)$, the second term
is bounded by
\begin{equation*}
\mf c(\bs u^N)  \sum_{x\in\bb T_N^d} \int
\Big\{\, \gamma\, [\,\bs u^N(x+e_1) - \bs u^N(x)\,]^2 \,
H^{(\ell)}_{1,x} (\eta)^2 \, +\, \gamma^{-1} \,
(\omega^{\ell} _{x+e_1})^2\, \Big\}  \, f(\eta) \; d\nu^N
\end{equation*}
for all $\gamma>0$.

We turn to the first term in \eqref{2-26}. In identity \eqref{2-04}
set $G(x) = 2\, a_4(x)\, H^{(\ell)}_{1,x} (\eta)$, $A= \{0, e_1\}$,
$A_\star = \{0\}$, $\mb x_A = e_1$. By Lemma \ref{l2-01}, the first
term in \eqref{2-26} is equal to
\begin{equation}
\label{2-09}
\sum_{k=1}^d \sum_{x\in\bb T_N^d}  \int
\{\, \omega_{x} \,-\, \omega_{x+e_k} \,\} \, M^{(\ell)}_{k,x} (\eta)
\, f(\eta) \; d\nu^N \;,
\end{equation}
where
\begin{equation*}
{\color{blue} M^{(\ell)}_{k,x} (\eta)} \;:=\;
2\, \sum_{\{y,y+e_k\} \subset \Lambda_{2(\ell-1)}}
\Phi_\ell (y,y+e_k) \, a_4(\, x-e_1-y \,) \,
H^{(\ell)}_{1,x-e_1-y} (\eta)\,  \omega_{x-e_1-y} \;.
\end{equation*}
By definition of $\Lambda_{2(\ell-1)}$, $x-e_1-y \not\in \{x, x+e_k\}$
for $y\in \Lambda_{2(\ell-1)}$.  Therefore, since $\mb x_A = e_1$,
$A_\star = \{0\}$, by the definition \eqref{2-06a} of
$H^{(\ell)}_{1,z}$, $M^{(\ell)}_{k,x}$ does not depend on the variables
$\eta_x$, $\eta_{x+e_k}$. In particular, by Lemma \ref{l2},
\eqref{2-09} is bounded above by
\begin{equation*}
\begin{aligned}
& \beta \, I^E_{N}(f;\nu^N)
\;+\;  \frac{\mf c(\bs u^N)}{\beta} \, \sum_{k=1}^d 
\sum_{x\in\bb T_N^d} \,
\int M^{(\ell)}_{k,x} (\eta)^2 \, f (\eta)  \; d\nu^N
\;+\; \beta\, \mf c (\bs u^N)\, \sum_{x\in\bb T_N^d}
[\bs u^N(x+e_1) - \bs u^N(x)]^4
\\
&\quad \;-\; \sum_{k=1}^d
\sum_{x\in\bb T_N^d} [\bs u^N(x+e_1) - \bs u^N(x)] \,
\int  M^{(\ell)}_{k,x} (\eta)  \, A_x(\eta)  \,  f(\eta)  \; d\nu^N 
\end{aligned}
\end{equation*}
for all $\beta>0$. To complete the proof of the lemma, it remains to
recall that the absolute value of $A_x(\eta)$ is bounded by
$\mf c(\bs u^N)$, and that we estimated only one term of $A_x$. This
explains why $M^{(\ell)}_{k,x}$ has to be replaced by
$M^{p, \ell}_{k,x}$, $1\le p\le 3$, in the statement of the lemma.
\end{proof}

By \eqref{2-11} and Lemma \ref{2-l02}, setting $\gamma = 1$, and
replacing $4\beta$ by $\beta$,
\begin{equation}
\label{2-07}
\begin{aligned}
& \int  \mc W^{(2)}_\ell \, f\; d\nu^N \; \le\;
\beta  \, I^E_{N}(f;\nu^N)
\;+\;  \frac{\mf c(\bs u^N)}{\beta}  \,
\int \sum_{k=1}^d \sum_{x\in\bb T_N^d} (H^{(\ell)}_{k,x})^2 \, f  \; d\nu^N 
\\
&\quad 
\;+\; \beta\, \mf c (\bs u^N)\, \sum_{x\in\bb T_N^d} [\bs u^N(x+e_1) -
\bs u^N(x)]^4  \\
& \quad +\; \mf c(\bs u^N)\, \sum_{k=1}^d \sum_{p=1}^3 \sum_{x\in\bb T_N^d}
\Big\{ \frac{1}{\beta}
\int [M^{p, \ell}_{k,x}]^2 \, f  \; d\nu^N 
\;+\; |\, \bs u^N(x+e_1) - \bs u^N(x) \,| \,
\int  |M^{p, \ell}_{k,x}|  \,  f  \; d\nu^N \,\Big\} 
\\
&\quad +\; \mf c(\bs u^N)  \sum_{x\in\bb T_N^d}
 \int
\Big\{\, [\bs u^N(x+e_1) - \bs u^N(x)]^2 \, [H^{(\ell)}_{1,x}]^2 \, +\, 
(\omega^{\ell} _{x})^2 \,\Big\}  \, f \; d\nu^N
\end{aligned}
\end{equation}
for all $\beta>0$, density $f$ with respect to
$\nu^N$ and $N> 4\ell$. 

Recall from \eqref{2-53a}, \eqref{2-52}, \eqref{2-84} that there
exists a finite constant $\mf c(\bs u^N)$ such that
\begin{equation}
\label{2-15}
\max_{x\in \bb T^d_N} |\, \bs u^N(x+e_1) - \bs u^N(x) \,| \;\le\;
\mf c(\bs u^N) \,  \frac{\sqrt{K}}{N} 
\end{equation}
for all $N\ge 1$.

\begin{lemma}
\label{2-l03}
Recall the definition of $M^{p,\ell}_{k,x}$ introduced in
\eqref{2-06b}. There exists a finite constant $\mf c(\bs u^N)$,
uniformly bounded in $N$, such that
\begin{equation*}
\begin{gathered}
\sum_{k=1}^d \sum_{x\in\bb T_N^d}  M^{p, \ell}_{k,x} (\eta)^2
\;\le\; \mf c(\bs u^N) \, s_d(\ell)\, \frac{K}{N^2}\, 
\sum_{x\in\bb T_N^d}  H^{(\ell)}_{1,x} (\eta)^2 \;, 
\\
\frac{\sqrt{K}}{N}\,
\sum_{k=1}^d  \sum_{x\in\bb T_N^d}  |\,M^{p, \ell}_{k,x} (\eta)\,|
\;\le\; \mf c(\bs u^N)\,  \frac{K}{N^2} \, 
\Big\{\,  s_d(\ell)\, \frac{ N^{d}}{\gamma}
\;+\; \gamma\, 
\sum_{x\in\bb T_N^d} H^{(\ell)}_{1,x} (\eta)^2\,\Big\}
\end{gathered}
\end{equation*}
for all $\gamma>0$, $1\le p\le 3$. 
\end{lemma}

\begin{proof}
Recall the definitions \eqref{2-24}, \eqref{2-27} of
$m^{p, \ell}_{x,y}$.  By Schwarz inequality,
\begin{equation*}
\sum_{k=1}^d \sum_{x\in\bb T_N^d}  M^{p, \ell}_{k,x} (\eta)^2
\;\le\;
\sum_{x\in\bb T_N^d} \Big( \sum_{j=1}^d 
\sum_{\{y,y+e_j\} \subset \Lambda_{2(\ell-1)}}
\Phi_\ell (y,y+e_j)^2  \Big) \Big( \sum_{k=1}^d 
\sum_{\{y,y+e_k\} \subset \Lambda_{2(\ell-1)}}
m^{p, \ell}_{x,y} (\eta)^2 \Big)\;.
\end{equation*}
Here, after Schwarz inequality, we added the sum over $j$ to the 
term involving $\Phi_\ell (y,y+e_j)$, and then we exchanged the order
of the sums over $k$ and $x$.  By Lemma \ref{l01}, the definitions
\eqref{2-24}, \eqref{2-27} of $m^{p, \ell}_{x,y}$, $a_4(\cdot)$,
respectively, and \eqref{2-15}, the previous expression is bounded by
\begin{equation*}
\mf c(\bs u^N) \, g_d(\ell)\, \frac{K}{N^2}\, \ell^d\, 
\sum_{x\in\bb T_N^d} 
H^{(\ell)}_{1,x} (\eta)^2
\end{equation*}
for some finite constant $\mf c(\bs u^N)$.  To complete the proof of
the first claim of the lemma it remains to recall that $s_d(\ell) =
g_d(\ell)\, \ell^d$. 

We turn to the second assertion.  By the definition \eqref{2-06b} of
$M^{p, \ell}_{k,x} (\eta)$, \eqref{2-27}, and \eqref{2-15},
\begin{equation*}
\frac{\sqrt{K}}{N} \sum_{k=1}^d
\sum_{x\in\bb T_N^d}  |\,M^{p, \ell}_{k,x} (\eta)\,|
\;\le\; \mf c(\bs u^N)\,  \frac{K}{N^2}\, \sum_{x\in\bb T_N^d}
\big|\, H^{(\ell)}_{1,x} (\eta) \,\big|
\sum_{k=1}^d \sum_{\{y,y+e_k\} \subset \Lambda_{2(\ell-1)}}
\big|\, \Phi_\ell (y,y+e_k) \,\big| \;.
\end{equation*}
By Schwarz inequality and Lemma \ref{l01}, the previous
expression is bounded by
\begin{equation*}
\mf c(\bs u^N)\, \frac{K}{N^2} \,
\sqrt{g_d(\ell) \ell^d }\,
\sum_{x\in\bb T_N^d} \big|\, H^{(\ell)}_{1,x} (\eta) \,\big|
\;.
\end{equation*}
To complete the proof of the second assertion, it remains to apply
Young's inequality.
\end{proof}

\begin{proof}[Proof of Proposition \ref{2-l05}]
The result is a consequence of \eqref{2-07} (setting
$\beta = \delta N^2/K$), Lemma \ref{2-l03}, and the estimate
\eqref{2-15}.
\end{proof}

\section{Boltzmann-Gibbs principle}
\label{sec-BG}

Throughout this section we adopt the same convention that
$\mf c(\bs u^N)$ represents a constant uniformly bounded in $N$,
depending only on $\bs u^N(\cdot)$, a finite subset $B$ of $\bb Z^d$ or
a cylinder function $f$, and the dimension $d$, and whose value may
change from line to line. The sequence $(\ell_N: N\ge 1)$ is the one
introduced in Theorem \ref{mt1}.

Recall from \eqref{2-35} the definition of $\Xi_{\bs u^N(\cdot), x}
f$, where $f$ is a cylinder function and $x\in\bb T^d_N$.

\begin{theorem}[Boltzmann-Gibbs principle]
\label{t01}
Fix a bounded function $F\colon \bb R_+ \times \bb T^d_N \to \bb R$, and
a cylinder function $f$.  Then, there exists a finite constant
$\mf c(\bs u^N)$, uniformly bounded in $N$, such that
\begin{align*}
& \bb E_{\nu^N} \Big[\, \Big|\,
\int_0^t \, \sum_{x\in \bb T^d_N} F (s,x)\,
(\Xi_{\bs u^N(\cdot), x} f) (\eta^N(s))
\, ds \, \Big| \,\Big]
\\
& \quad \;\le\;
\mf c(\bs u^N) \, \Big\{ \frac{1}{K_N} \;+\;
K_N\,  \int_0^t (\, 1 + \Vert F(s)\Vert^2_\infty\,)\, \big[ \,
H_N(f^N_s)  + (N/\ell_N)^d \,\big] \, ds\; \Big\} 
\end{align*}
for all $t> 0$.  In this equation, $f^N_s$ is given by \eqref{2-34c},
and $(\ell_N : N\ge 1)$ is the sequence introduced in the statement of
Theorem \ref{mt1}.
\end{theorem}

By \eqref{2-17}, the previous theorem follows from the next result.

\begin{proposition}
\label{2-l14} 
Fix a bounded function $F\colon \bb R_+ \times \bb T^d_N \to \bb R$,
and a finite subset $B$ of $\bb Z^d_\perp$ with at least two elements.
Then, there exists a finite constant $\mf c(\bs u^N)$, uniformly
bounded in $N$ and depending only on $\bs u^N(\cdot)$, $B$ and $d$,
such that
\begin{align*}
& \bb E_{\nu^N} \Big[\, \Big|\,
\int_0^t \, \sum_{x\in \bb T^d_N} F (s,x)\,
\omega_{B+x}(s)  \, ds \, \Big| \,\Big]
\\
& \quad
\;\le\; \mf c(\bs u^N) \, \Big\{ \frac{1}{K_N} \;+\;
K_N \,  \int_0^t (\, 1 + \Vert F(s)\Vert^2_\infty\,)\, \big\{ \,
H_N(f^N_s)  + (N/\ell_N)^d \,\big\} \, ds\; \Big\} 
\end{align*}
for all $t> 0$.  In this equation, $f^N_s$ is given by \eqref{2-34c},
and $(\ell_N : N\ge 1)$ is the sequence introduced in the statement of
Theorem \ref{mt1}.
\end{proposition}

The proof of Proposition \ref{2-l14} is divided in several steps and
is close to the one of Theorem \ref{t02}. The difference is that we
have an absolute value. We appeal to Feynman-Kac formula to overcome
this impediment.

Let
\begin{equation}
\label{2-30}
{\color{blue} \mc F (s,\eta)} 
\;:=\; \sum_{x \in \bb T^d_N} F(s,x) \, \omega_{x+B}\;,
\end{equation}
and rewrite $\mc F (s,\eta)$ as
$\mc F^{(1)}_\ell (s,\eta) + \mc F^{(2)}_\ell (s,\eta)$, as in
\eqref{2-04}, \eqref{2-29}.  Recall the definition of $\Upsilon$,
introduced in \eqref{2-08}, and also rewrite it as
$\Upsilon^{(1)}_\ell (\eta) + \Upsilon^{(2)}_\ell (\eta)$ in the same
way.

Denote by $\color{blue} F^{(\ell)}_{k,s,x}$, the function
$H^{\ell, F_s, B}_{k}(x)$ introduced in \eqref{2-06a}. Note that $G$
and $A$ have been replaced by $F(s,\cdot)$ and $B$, respectively.
Recall from Lemma \ref{2-l07} the definition of the functions
$G_1=G_{1,1}$, $G_2$, and denote by $\color{blue} J^{1,\ell}_{k,x}$,
$\color{blue} J^{2,\ell}_{k,x}$, the function
$H^{\ell, G_1, A_1}_{k} (x)$, $H^{\ell, G_2, A_2}_{k} (x)$ introduced
in \eqref{2-06a} with $A_1 = \{0,e_1\}$, $A_2 = \{-e_1,0,e_1\}$,
respectively.

Recall the estimate \eqref{2-41}. The constant $\mf c(\bs u^N)$
(uniformly bounded in $N$) and the function $\Psi_{s,x}$ introduced
below correspond to the terms appearing on the right-hand side of the
inequality.  Let $\Psi_{s} \colon \Omega_N \to \bb R$,
$s \in \bb R_+$, be given by
\begin{equation}
\label{2-33}
\begin{aligned}
{\color{blue} \Psi  (s,\eta)} \; & =\;
\frac{1}{2}\, \Upsilon^{(1)}_\ell (\eta)
\;+\; 2\, \mf c(\bs u^N) \, \sum_{x\in\bb T_N^d} (\omega^{\ell} _{x})^2
\\
& +\;  \mf c(\bs u^N) \, K \, \frac{1}{s_d(\ell)}\, \sum_{k=1}^d
\sum_{x\in\bb T_N^d} \Big\{\, 2\, (F^{(\ell)}_{k,s, x})^2
\, + \,  (J^{1,\ell}_{k,x})^2 \, +  \,   (J^{2,\ell}_{k,x})^2\,\Big\} \;.
\end{aligned}
\end{equation}

In Lemma \ref{l19b} below, we estimate the expectation (with respect
to $\nu_N$) of
\begin{equation*}
\Big|\;
\int_0^t \mc F^{(2)}_\ell (s,\eta^N(s)) \, ds \, \Big|
\,-\, \int_0^t \Psi  (s,\eta^N(s)) \, ds \;.
\end{equation*}
In Lemma \ref{l15} and equation \eqref{2-32}, we bound the expectations of
\begin{equation*}
\int_0^t \Big\{ \,\Big|\, \mc F^{(1)}_\ell (s,\eta^N(s)) \,\Big| \;+\;
\Big|\Upsilon^{(1)}_\ell (\eta^N(s))\,\Big| \,\Big\} \; ds \,,
\end{equation*}
and 
\begin{equation*}
\begin{gathered}
\int_0^t \sum_{x\in\bb T_N^d} (\omega^{\ell} _{x})^2 \, ds \;,
\qquad 
\int_0^t \sum_{k=1}^d \sum_{x\in\bb T_N^d} \Big\{\, 
2 \,  (F^{(\ell)}_{k,s, x})^2 \,+\,  (J^{1,\ell}_{k,x})^2
\, + \, (J^{2,\ell}_{k,x})^2\,\Big\} \, ds \;.
\end{gathered}
\end{equation*}
Proposition \ref{2-l14} follows from these bounds.

Below, the factor $\Vert F\Vert^2_\infty$ is hidden in
$\Psi$, and not in the right-hand side.

\begin{lemma}
\label{l19b}
There exists a finite constant $\mf c(\bs u^N)$ such that 
\begin{equation*}
\begin{aligned}
& \bb E_{\nu^N} \Big[\,
\Big|\;
\int_0^t \mc F^{(2)}_\ell (s,\eta^N(s)) \, ds \, \Big|
\,-\, \int_0^t  \Psi  (s,\eta^N(s)) \, ds
\,\Big]
\\
&\quad  \;\le\; \frac{\log 2}{K} \;+\; 
\mf c(\bs u^N) \, K\, (N/\ell)^{d} \, t
\end{aligned}
\end{equation*}
for all $t>0$ and  $N\ge 1$.
\end{lemma}

\begin{proof}
By Jensen's inequality, the expectation appearing in the statement
of the lemma is bounded above by
\begin{equation*}
\begin{aligned}
\frac{1}{\gamma} \, \log \bb E_{\nu^N} \Big[\,
\exp \, \gamma \,  \Big\{\, \Big|\,
\int_0^{t} \mc F^{(2)}_\ell (s,\eta^N(s)) \, ds \, \Big| 
\,-\, \int_0^{t} \Psi  (s,\eta^N(s)) \, ds \,\Big\}
\,\Big] 
\end{aligned}
\end{equation*}
for all $\gamma>0$. As
$e^{|a|} \le e^a + e^{-a} \le 2 \max\{e^a , e^{-a}\}$, by the
linearity of the expectation, this expression is less than or equal to
\begin{equation*}
\frac{1}{\gamma} \, \log 2 \;+\; \max_{b=\pm 1}
\frac{1}{\gamma} \, \log \bb E_{\nu^N} \Big[\,
\exp \,\gamma\, \Big\{\, 
\int_0^{t} \big\{\, b\, \mc F^{(2)}_\ell (s,\eta^N(s)) 
\,-\, \Psi  (s, \eta^N(s)) \,\big\}\, ds
\,\Big\} \,\Big]\;.
\end{equation*}

We estimate the second term for $b=1$; the same argument applies to
$b=-1$. By the Feynman-Kac formula (cf. \cite[Lemma  B.1]{jl}), the
second term of this expression is bounded by
\begin{equation*}
\frac{1}{\gamma}\, \int_0^{t}\, 
\sup_{f} \Big\{ \int \gamma\, \widehat{\mc F}^{(2)}_\ell (s) \, f\, d\nu^N
\;+\; \frac{1}{2} \, \int L^*_N \mb 1 \, f\, d\nu^N
\;-\; N^2  \, I^E_N(f; \nu^N)\,\Big\}\, ds \;, 
\end{equation*}
where the supremum is carried over all probability densities $f$ with
respect to $\nu^N$, and
$\color{blue} \widehat{\mc F}^{(2)}_\ell (s) \,=\, \mc F^{(2)}_\ell
(s) - \Psi (s)$. Mind that we kept only the integrated carr\'e du
champ associated to the exclusion dynamics. Choosing $\gamma = K$, by
Lemma \ref{2-l07} the previous integral becomes
\begin{equation} 
\label{18}
\frac{1}{K} \, \int_0^t\, \sup_{f} \Big\{ \,
K \int \big[\, \widehat{\mc F}^{(2)}_\ell (s)
\, +\,  \frac{1}{2K}\, U \, +\,  \frac{1}{2} \, \Upsilon  \, \big]\, 
f\, d\nu^N \;-\; N^2 \, I^E_N(f; \nu^N) \,\Big\}\, ds \;.
\end{equation}

By definition of $\widehat{\mc F}^{(2)}_\ell (s)$, $U$, and
$\Upsilon$,
$K \, \{ \, \widehat{\mc F}^{(2)}_\ell (s) + (1/2K)\, U + (1/2)
\Upsilon \, \}$ is equal to
\begin{equation*}
\begin{aligned}
&  (1/2)\, U  (\eta) \;+\; K \, \mc F^{(2)}_\ell (s,\eta)
\;+\; (K/2)\,  \Upsilon^{(2)}_\ell (\eta) 
\;-\; 2\, \mf c(\bs u^N)  \, K\, \sum_{x\in\bb T_N^d} (\omega^{\ell} _{x})^2
\\
&  \;-\; \mf c(\bs u^N) \, K^2\, \frac{1}{s_d(\ell)}\, 
\sum_{k=1}^d  \sum_{x\in\bb T_N^d} \Big\{\,
2 \, (F^{(\ell)}_{k,s, x})^2  \,+\, 
\, (J^{1,\ell}_{k, x})^2 \, + \, (J^{2,\ell}_{k, x})^2 \,\Big\} \;.
\end{aligned}
\end{equation*}
Note that $\widehat{\mc F}^{(2)}_\ell (s)$, $ \Upsilon (\eta)$ have
been replaced by $\mc F^{(2)}_\ell (s)$, $\Upsilon^{(2)}_\ell (\eta)$,
respectively (because we included in $\Psi$ the term
$\Upsilon^{(1)}_\ell (\eta)$).

By \eqref{2-40}, the absolute value of $(1/2)\, U (\eta)$ is less than
or equal to $\mf c(\bs u^N) \, K^2\, N^{d-2}$.  We apply the bound
\eqref{2-41} to estimate the term
$\mc F^{(2)}_\ell + (1/2) \Upsilon^{(2)}_\ell$ integrated with respect
to the measure $f(\eta)\, \nu^N(d\eta)$.  Estimate \eqref{2-41} is
used as it does not involve entropy of $f$ as in Lemma \ref{2-l08}
Mind that $(1/2)\, \Upsilon^{(2)}_\ell$ is the sum of two terms,
each one multiplied by $1/2$ and that $G_{1,j}=0$ for $2\le j\le d$.
The term $\Psi$ has been introduced to compensate the terms appearing
on the right-hand side of \eqref{2-41}. Thus, by \eqref{2-41} with
$\delta=1/2$, the expression inside braces in \eqref{18} is less than
or equal to
\begin{equation*}
\mf c(\bs u^N) \, K^2\, N^{d-2} \;+\; \mf c(\bs u^N)
\, K^2\, (N/\ell)^{d}
\;\le\; 
\mf c(\bs u^N) \, K^2\, (N/\ell)^{d}  \;.
\end{equation*}
The values of the constants $\mf c(\bs u^N)$ have changed. There is a
factor $2\, \mf c(\bs u^N)$ multiplying $(\omega^{\ell} _{x})^2$
because the term $\mc F^{(2)}_\ell (s)$ contributed with $1$ and the
two terms in $\Upsilon^{(2)}_\ell$ contributed with $1/2$ each. In the
last step, we applied the bound
$N^2 = s_d(\ell) \, \ell^{d/2} \ge \ell^d$. This completes the proof
of the lemma.
\end{proof}

\begin{remark}
\label{rm4}
The previous result holds in any dimension and for any sequence
$K_N$. It is a consequence of Feynman-Kac formula and the integration
by parts stated in Lemma \ref{l2}.
\end{remark}

We turn to the terms $\mc F^{(1)}_\ell$ and $\Upsilon^{(1)}_\ell$. Fix
a bounded function $G\colon \bb R_+ \times \bb T^d_N \to \bb R$, a
finite subset $A$ of $\bb Z^d$ with at least two elements, and recall
the definition of $\mc W^{(1)}_\ell$ introduced in \eqref{2-04} which
is now time-dependent because so is the test function $F$ appearing in
the statement of Proposition \ref{2-l14}.

\begin{lemma}
\label{l15}
Fix a bounded function $G\colon \bb R_+ \times \bb T^d_N \to \bb R$, a
finite subset $A$ of $\bb Z^d$ with at least two elements. Then, there
exists a constant $\mf c (\bs u^N)$, uniformly bounded in $N$ and
depending only on $A$, $\bs u^N$ and the dimension, such that
\begin{equation*}
\bb E_{\nu^N} \Big[\,
\int_0^t \,\Big|\, \mc W^{(1)}_\ell (s,\eta^N(s)) \,  \Big| \,
\, ds \, \Big] 
\;\le \; \mf c(\bs u^N) \,
\int_0^t  \Vert G(s)\Vert_\infty\,
\Big\{   H_N(f^N_s )\, ds\; +\;  (N/\ell)^d\,\Big\}\; ds \;.
\end{equation*}
\end{lemma}

\begin{proof}
The expectation appearing in the statement of the lemma is equal to
\begin{align*}
\int_0^t ds\; \int\, \big|\, 
\mc W^{(1)}_\ell (s,\cdot) \, \big|\, \, f^N_s \, d\nu^N\;.
\end{align*}
It remains to apply Lemma \ref{2-l09}.
\end{proof}

To complete the proof of the Boltzmann-Gibbs principle recall from
Lemma \ref{2-l08} that  for all $1\le k\le d$,
\begin{equation}
\label{2-32}
\begin{aligned}
& \bb E_{\nu^N} \Big[\,
\int_0^t \, 
\sum_{x\in \bb T^d_N} \big[\, R^{(\ell)}_{k,s,x} \,\big]^2
\, ds \, \Big]  \\
& \quad \le\;  \mf c(\bs u^N)  \, \int_0^t 
\Vert G(s)\Vert^2_\infty\,  s_d(\ell) \, \big\{\,  H_N(f^N_s)\; ds  \;+\; 
 (N/\ell)^d\,\big\} \; ds
\end{aligned}
\end{equation}
for a finite constant $\mf c(\bs u^N)$, In this equation,
$R^{(\ell)}_{k,s,x}$ represents $F^{(\ell)}_{k,s,x}$,
$J^{1,\ell}_{k,x}$, or $J^{2,\ell}_{k,x}$.  A similar bound holds for
$\omega^\ell_x$ and $G(s)=1$, without the factor $s_d(\ell)$ on the
right-hand side. Note that we have a front factor $K/s_d(\ell)$ in the
last term of $\Psi$ in \eqref{2-33}.

\begin{proof}[Proof of Proposition \ref{2-l14}]
The result follows from the definition of $\Psi_{s,x}$ given in
\eqref{2-33}, Lemmata \ref{l19b}, \ref{l15}, equation \eqref{2-32},
with $\ell_N$ as defined in relation \eqref{2-18}.
\end{proof}

\begin{proof}[Proof of Theorem \ref{t01}]
In view of the decomposition of \ $\Xi_{\bs u^N(\cdot), x} f$
presented in \eqref{2-17}, the Boltzmann-Gibbs principle follows from
Proposition \ref{2-l14}, and the definition \eqref{2-18} of the
sequence $\ell_N$.
\end{proof}

The next step is to replace $\Xi_{\bs u^N(\cdot), x} f$ in Theorem
\ref{t01} with $\Xi^c_{\bs u^N(\cdot), x} f$ introduced in
\eqref{2-35b}; see Corollary \ref{2-l10}. $\Xi^c$ is a more natural
object in view of Lemma \ref{lem:3.1-F}. We prepare a lemma.

For a finite subset $B$ of $\bb Z^d$, let
$\bs u^N_B\colon \bb T^d_N \to (0,1)$ be given by
$\color{blue} \bs u^N_B(x) = \prod_{z\in B} \bs u^N(x+z)$.  For a
function $G\colon \bb T^d_N \to \bb R$, $1\le j\le d$, let
\begin{equation*}
{\color{blue} \mf n_K (G)} \,:=\,
\{\, \Vert G \Vert_\infty \, + \, K^{-1/2}\, \Vert \nabla_N
G\Vert_\infty \}\;,
\end{equation*}
where $\color{blue} (\nabla_N G)(x)$ is the $d$-dimensional vector
whose $k$-th component is $N \,[\, G((x+e_k)/N) - G(x/N)\,]$.

\begin{lemma}
\label{l05}
Fix a finite subset $B$ of $\bb Z^d$ and a function
$G\colon \bb R_+ \times \bb T^d_N \to \bb R$. Then, there exists a
finite constant $\mf c(\bs u^N)$ such that
\begin{align*}
& \bb E_{\nu^N} \Big[\,
\int_0^t \,\Big|\, \sum_{x\in \bb T^d_N}
G (s, x)\, \bs u^N_B(x)\, 
[\, \omega_x(s) - \omega_{x+e_j}(s)\,]  \,\Big| \,  \, ds \, \Big] \\
&\qquad \le\;
\int_0^t \frac{1}{\gamma_s}  \big\{\, H_N(f^N_s) +
\log 2 \, \big\} \; ds
\;+\; \mf c (\bs u^N) \,  K\, N^{d-2}\, \int_0^t   \gamma_s \,  \mf n_K (G_s)^2\,
\, e^{c (\bs u^N) \, \gamma_s\,  \mf n_K (G_s) \, \epsilon_N }\; ds 
\end{align*}
for all $t>0$, $1\le j\le d$, $N\ge 1$, and
$\gamma\colon [0,t] \to (0,\infty)$. In this formula,
$\epsilon_N = K^{1/2}/N$.
\end{lemma}

\begin{proof}
Sum by parts, and let
$W (s,\eta) = \sum_{x\in \bb T^d_N} [\, G (s,x)\, \bs u^N_B(x) \,-\,
G(s,x-e_j)\, \bs u^N_B(x-e_j) \,] \, \omega_x (s)$.  By the
entropy inequality, the expectation appearing in the statement of the
lemma is bounded by
\begin{equation*}
\int_0^t \frac{1}{\gamma_s}   H_N(f^N_s) \, ds \;+\;
\int_0^t \frac{1}{\gamma_s}  \log \int e^{\gamma_s\, |  \, W(s) \,|} \;
d\nu^N \, ds
\end{equation*}
for every $\gamma\colon [0,t] \to (0,\infty)$.

As $e^{|a|} \le e^a + e^{-a} \le 2 \max \{ e^a , e^{-a}\}$, by
linearity of the expectation,
\begin{equation*}
\frac{1}{\gamma_s} \log \int e^{\gamma_s\,  |\,W(s)\,|} \; d\nu^N \;\le\;
\frac{\log 2}{\gamma_s} \;+\; \max_{b=\pm 1} \frac{1}{\gamma_s}
\log \int e^{b\, \gamma_s\,  W(s)} \; d\nu^N\;.
\end{equation*}
We estimate the second term with $b=1$, as the argument applies to
$b=-1$. As $\nu^N$ is a product measure,
\begin{equation*}
\frac{1}{\gamma_s} \log \int e^{\gamma_s\,  W(s)} \; d\nu^N \;=\;
\frac{1}{\gamma_s} \sum_{x\in \bb T^d_N} \log \int
e^{\beta_x \, \gamma_s  \,  \omega_x} \; d\nu^N
\end{equation*}
where
$\beta_x = [\, G (s,x) \,\bs u^N_B(x)\, -\, G(s,x-e_j) \, \bs
u^N_B(x-e_j) \,]$. By \eqref{2-53a}, \eqref{2-52}, \eqref{2-84},
$|\beta|\le \mf c (\bs u^N)\, (K^{1/2}/N) \, \mf n_K
(G_s)$, where $\mf n_K (G_s)$ has been introduced above the statement
of the lemma.  Since $e^a \le 1 + a + a^2 e^{|a|}$,
$E_{\nu^N}[\omega_x]=0$, and $\log (1+b) \le b$, the previous
expression is bounded by
\begin{equation*}
\gamma_s\, \sum_{x\in \bb T^d_N} \beta^2_x\, 
e^{\gamma_s \, |\beta_x| } \;\le\;
c (\bs u^N) \, \gamma_s\,  \mf n_K (G_s)^2\, K\, N^{d-2}
\, e^{c (\bs u^N) \, \gamma_s\,  \mf n_K (G_s) \, (K^{1/2}/N) }
\;.
\end{equation*}
To complete the proof of the lemma, it remains to recollect the
previous estimates.
\end{proof}

Recall the definition of the local function
$\Xi^{\rm c}_{\bs u^N (\cdot), x} f$ introduced in \eqref{2-35b},
where $f$ represents a cylinder function. Next result is a consequence
of the Boltzmann-Gibbs principle and the previous lemma.

\begin{corollary}
\label{2-l10}
Fix a function $G\colon \bb R_+ \times \bb T^{d}_N \to \bb R$ and a
cylinder function $f$. Then, there exists a finite constant
$\mf c(\bs u^N)$, depending only on $\bs u^N$, $f$ and $d$, uniformly
bounded in $N$, such that
\begin{align*}
& \bb E_{\nu^N} \Big[\, \Big|\,
\int_0^t \, 
\sum_{x\in \bb T^d_N} G (s, x)\,
(\Xi^{\rm c}_{\bs u^N (\cdot), x}  f) (\eta (s))
\, ds \, \Big| \,\Big]
\\ 
&\quad \;\le \;
\mf c(\bs u^N) \, \Big\{ \frac{1}{K_N} \;+\;
K_N\,  \int_0^t (\, 1 + \Vert G(s)\Vert^2_\infty\,)\, \big[ \,
H_N(f^N_s)  + (N/\ell_N)^d \,\big] \, ds\; \Big\}
\\
& \quad \; + \;
\int_0^t \frac{1}{\gamma_s}  \big\{\, H_N(f^N_s) +
\log 2 \, \big\} \; ds
\;+\; \mf c (\bs u^N) \,  K\, N^{d-2}\, \int_0^t   \gamma_s \,  \mf n_K (G_s)^2\,
\, e^{c (\bs u^N) \, \gamma_s\,  \mf n_K (G_s) \, \epsilon_N }\; ds 
\end{align*}
for all $t>0$, $\gamma\colon [0,t] \to (0,\infty)$.  In this equation,
$f^N_s$ is given by \eqref{2-34c}, $(\ell_N : N\ge 1)$ is the sequence
introduced in the statement of Theorem \ref{mt1}, and
$\epsilon_N = K^{1/2}/N$.
\end{corollary}

\begin{proof}
Fix a function $G\colon \bb R_+ \times \bb T^{d}_N \to \bb R$
satisfying the hypotheses of the corollary and a cylinder function
$f$.  In view of \eqref{2-42}, assume that $f(\eta) = \eta_B$ for some
finite set $B$. Recall from \eqref{2-35} the definition of
$\Xi_{\bs u^N(\cdot), x} f$.  The expectation appearing in the
statement is bounded by
\begin{equation}
\label{2-43}
\begin{aligned}
& \bb E_{\nu^N} \Big[\, \Big|\,
\int_0^t \, 
\sum_{x\in \bb T^d_N} G (s, x)\,
(\Xi_{\bs u^N(\cdot), x}  f) (\eta(s))  
\, ds \, \Big| \,\Big] \\
&\quad
+\; 
\bb E_{\nu^N} \Big[\, \Big|\,
\int_0^t \, 
\sum_{x\in \bb T^d_N} G (s, x)\, \Big\{\, 
(\Xi^{\rm c}_{\bs u^N (\cdot), x} f) (\eta (s))
\,-\, (\Xi_{\bs u^N(\cdot), x}  f) (\eta(s))  \,\Big\}
\, ds \, \Big| \,\Big] \;.
\end{aligned}
\end{equation}
Theorem \ref{t01} provides a bound for the first term. We turn to the
second.

Recall the definition of $\zeta_\cdot$ in \eqref{2-20}.
By definition,
\begin{align*}
& \sum_{x\in \bb T^d_N} G (s,x)\, \big\{\, 
(\Xi^{\rm c}_{\bs u^N (\cdot), x}  f) (\eta)
\,-\, (\Xi_{\bs u^N(\cdot), x}  f) (\eta) 
\,\big\}
\\
&\quad =\;  
-\,\sum_{x\in \bb T^d_N} G (s,x)\, \sum_{y\in\bb Z^d}
E_{\nu^N}
[\, \tau_x f\, \zeta_{x+y} \,]
\big\{\, \omega_{x+y} - \omega_{x}\,\}\;.
\end{align*}
As $f(\eta) = \eta_B$,
\begin{align*}
\sum_{y\in\bb Z^d} E_{\nu^N}
[\, \tau_x f\, \zeta_{x+y} \,]
\big\{\, \omega_{x+y} - \omega_{x}\,\} 
\;=\; \sum_{z\in B} \bs u^N_{B\setminus \{z\}} (x)\, 
\big\{\, \omega_{x+z} - \omega_{x}\,\}\;,
\end{align*}
where $\bs u^N_{C} (\cdot)$ has been introduced just before Lemma
\ref{l05}. The penultimate displayed formula is thus equal to
\begin{align*}
-\, \sum_{x\in \bb T^d_N} G (s,x)\, 
\sum_{y\in B} \bs u^N_{B\setminus \{y\}} (x)\, 
\big\{\, \omega_{x+y} - \omega_{x}\,\} \;.
\end{align*}
The second term in \eqref{2-43} is thus bounded by
\begin{align*}
\sum_{y\in B}
\bb E_{\nu^N} \Big[\, \Big|\,
\int_0^t \, 
\sum_{x\in \bb T^d_N} G (s, x)\,
\bs u^N_{B\setminus \{y\}} (x)\, 
\big\{\, \omega_x (s) - \omega_{x+y} (s) \,\}
\, ds \, \Big| \,\Big] \;.
\end{align*}
To complete the proof of the corollary, it remains to rewrite the
difference $\omega_x (s) - \omega_{x+y} (s)$ as a sum of terms of the
form $\omega_{z} (s) - \omega_{z+e_k} (s)$, and to apply Lemma
\ref{l05}.
\end{proof}

Recall that, to study the fluctuation of the interface, we introduced
in \eqref{2-44} the stretching operator $A$ to the normal direction of
the interface by $\sqrt{K}$ combining with the division by $N$, which
associates the microscopic variable to the macroscopic one.

Fix a function
$F\colon \bb R_+ \times (\sqrt{K} \bb T \times \bb T^{d-1}) \to \bb
R$, and let $G\colon \bb R_+ \times \bb T^{d}_N \to \bb R$ be given by
$G(s,x) = \widehat {F_{s}}(Ax)$ (cf. \eqref{2-69}). Then,
\begin{equation*}
\begin{gathered}
\Vert G(s) \Vert_\infty \,\le\, \Vert F(s) \Vert_\infty\;, \quad
\mf n_K (G_s) \,\le\,  \mf m_K(F(s))\;,
\\
\text{where} \;\; {\color{blue} \mf m_K(F(s)) }
\, := \, \Vert F(s) \Vert_\infty \,+\,
\Vert \partial_\vartheta  F(s) \Vert_\infty \,+\, \frac{1}{\sqrt{K}}
\, 
\sum_{j=2}^d \Vert \partial_{\theta_j}  F(s) \Vert_\infty \;.
\end{gathered}
\end{equation*}
With this notation, Corollary \ref{2-l10} can be restated as

\begin{corollary}
\label{2-l12}
Fix a function
$F\colon \bb R_+ \times (\sqrt{K} \bb T \times \bb T^{d-1}) \to \bb R$
and a cylinder function $f$. Then, there exists a finite constant
$\mf c(\bs u^N)$, depending only on $\bs u^N$, $f$ and $d$, uniformly
bounded in $N$, such that
\begin{align*}
& \bb E_{\nu^N} \Big[\, \Big|\,
\int_0^t \, 
\sum_{x\in \bb T^d_N} \widehat {F_{s}}(Ax) \,
(\Xi^{\rm c} _{\bs u^N (\cdot), x} f) (\eta (s))
\, ds \, \Big| \,\Big]
\\
&\quad \;\le \;
\mf c(\bs u^N) \, \Big\{ \frac{1}{K_N} \;+\;
K_N\,  \int_0^t (\, 1 + \Vert F(s)\Vert^2_\infty\,)\, \big[ \,
H_N(f^N_s)  + (N/\ell_N)^d \,\big] \, ds\; \Big\}
\\
& \quad \; + \;
\int_0^t \frac{1}{\gamma_s}  \big\{\, H_N(f^N_s) +
\log 2 \, \big\} \; ds
\;+\; \mf c (\bs u^N) \,  K\, N^{d-2}\, \int_0^t   \gamma_s \,  \mf m_K (F(s))^2\,
\, e^{c (\bs u^N) \, \gamma_s\,  \mf m_K (F(s)) \, \epsilon_N }\; ds
\end{align*}
for all $t>0$, $\gamma\colon [0,t] \to (0,\infty)$.  In this equation,
$f^N_s$ is given by \eqref{2-34c}, $(\ell_N : N\ge 1)$ is the sequence
introduced in the statement of Theorem \ref{mt1}, and
$\epsilon_N = K^{1/2}/N$.
\end{corollary}

The next estimate is needed in the proof of the convergence of the
density field $X^N_t$.

\begin{lemma}
\label{2-l11}
Fix a function $G\colon \bb R_+ \times \bb T^{d}_N \to \bb R$. Then,
\begin{align*}
& \bb E_{\nu^N} \Big[\, 
\int_0^t \Big| \, 
\sum_{x\in \bb T^d_N} G (s, x)\, [\eta_x(s) - \bs u^N(x)\,]
\, \Big| \, ds \, \Big]
\\
&\quad \leq \,
\int_0^t \frac{1}{\gamma_s}\,   \{\, H_N(f^N_s)  \,+\, \log 2 \,\} \, ds
\,+\, \int_0^t \gamma_s \,  \, \sum_{x\in \bb T^d_N}
G (s, x)^2  \, e^{\gamma_s \, \Vert G(s) \Vert_\infty}\; ds
\end{align*}
for all  $t>0$, $\gamma\colon [0,t]\to (0,\infty)$. 
\end{lemma}

\begin{proof}
By the entropy inequality, the expectation appearing in the statement
of the lemma is bounded by
\begin{equation*}
\int_0^t \frac{1}{\gamma_s}\,   H_N(f^N_s)\, ds\;
+\;  \int_0^t \frac{1}{\gamma_s}\,  \log \,  E_{\nu^N} \Big[\, 
\exp \Big\{\,  \Big|\, 
\gamma_s \, \sum_{x\in \bb T^d_N} G (s, x)\, [\eta_x - \bs u^N(x)\,]
\, \Big|\,  \Big\} \, \Big] \, ds
\end{equation*}
for all $\gamma\colon [0,t]\to (0,\infty)$. Repeat the argument
presented in the proof of Lemma \ref{l19b} to get rid of the absolute
value to conclude that the second term of the previous expression is
bounded by
\begin{equation*}
\int_0^t \frac{1}{\gamma_s}\, \log 2  \, ds
\;+\; \max_{b=\pm1}
\int_0^t \frac{1}{\gamma_s}\,  \log \,  E_{\nu^N} \Big[\, 
\exp \Big\{\,  b\, \gamma_s\, 
\sum_{x\in \bb T^d_N} G (s, x)\, [\eta_x - \bs u^N(x)\,]
\,  \Big\} \, \Big] \, ds\;.
\end{equation*}
As $\nu^N$ is a product measure, the second term for $b=1$ is equal to
\begin{equation*}
\int_0^t \frac{1}{\gamma_s} \, \sum_{x\in \bb T^d_N}
\log \, E_{\nu^N} \Big[\,  
\exp \Big\{\,  \gamma_s
G (s, x)\, [\eta_x - \bs u^N(x)\,]
\, \Big\} \, \Big] \, ds\;.
\end{equation*}
Since $e^a \le 1 + a + a^2 e^{|a|}$ and $\log (1+b) \le b$,
$a\in\bb R$, $b>0$, and $\eta_x - \bs u^N(x)$ has mean zero with
respect to $\nu^N$ and $|\eta_x - \bs u^N(x)|\leq 1$, a second order
Taylor expansion yields that the previous expression is bounded by
\begin{equation*}
\int_0^t \, \gamma_s \,  \sum_{x\in \bb T^d_N}
G (s, x)^2  \, e^{\gamma_s \, \Vert G(s) \Vert_\infty}\; ds\;.
\end{equation*}
To complete the proof, it remains to recollect all previous
estimates.
\end{proof}

\subsection*{Conclusion}

In this last part of this section, we recall the estimate for the
entropy obtained in Corollary \ref{m-cor} and restate the last two
results in view of this bound.  Fix $T\ge 1$, and recall from
\eqref{2-50b} the definition of $\kappa_N(T)$.

By \eqref{2-50b}, in the statement of Corollary \ref{2-l12} bound
$H_N(f^N_s)$ and $\log 2$ by $\kappa_N(T) (N/\ell)^d$, After this
replacement, to make the two terms in the second line of the same
order, choose
$\gamma_s = (\kappa_N(T)/K)^{1/2} \, (N/\ell^{d/2}_N)\, [\, 1 + \mf
m_K(F(s))\,]^{-1}$. After these replacements, the right-hand side of
the statement of Corollary \ref{2-l12} becomes
\begin{align*}
& \mf c(\bs u^N) \, \Big\{ \frac{1}{K_N} \;+\;
K_N\,  \kappa_N(T)\,  (N/\ell)^d  \,
\int_0^t (\, 1 + \Vert F(s)\Vert^2_\infty\,) \, ds\; \Big\}
\\
& \quad \; + \;
\mf c (\bs u^N) \,  \sqrt{ K\, \kappa_N(T)}\,
\frac{N^{d-1}}{\ell^{d/2}_N} \,
\Big \{\, 1 \,+\, \, e^{\mf c (\bs u^N) \, \sqrt{\kappa_N(T)} /
\ell^{d/2}_N} \,\Big\}  \int_0^t   [\, 1 + \mf m_K (F(s)) \,] \; ds
\end{align*}
By the assumptions \eqref{2-72} on the sequence $K_N$, \eqref{2-67}
holds so that $\sqrt{\kappa_N(T)} / \ell^{d/2}_N \to 0$. On the other
hand, $1/K_N$ is bounded by the second term in the first line.  Hence,
the previous expression is less than or equal to
\begin{align*}
& \mf c(\bs u^N) \, 
K_N\,  \kappa_N(T)\,  (N/\ell)^d  \,
\int_0^t (\, 1 + \Vert F(s)\Vert^2_\infty\,) \, ds\; 
\\
& \quad \; + \;
\mf c (\bs u^N) \,  \sqrt{ K\, \kappa_N(T)}\,
\frac{N^{d-1}}{\ell^{d/2}_N} \, \int_0^t   [\, 1 + \mf m_K (F(s)) \,] \; ds\;.
\end{align*}
By \eqref{2-28},
$\sqrt{ K\, \kappa_N(T)} \,\ell^{d/2}_N \le K\, \kappa_N(T)\, N$. In
conclusion, under the hypotheses \eqref{2-72},
\begin{equation}
\label{2-46}
\begin{aligned}
& \bb E_{\nu^N} \Big[\, \Big|\,
\int_0^t \, 
\sum_{x\in \bb T^d_N} \widehat {F_{s}}(Ax) \,
(\Xi^{\rm c}_{\bs u^N (\cdot), x} f) (\eta (s))
\, ds \, \Big| \,\Big]
\\
&\quad \;\le \;
\mf c(\bs u^N) \, 
K_N\,  \kappa_N(T)\,  (N/\ell)^d  \,
\int_0^t (\, 1 + \mf m_K (F(s)) \,) \, ds\; 
\end{aligned}
\end{equation}
for all $t\le T$.

Analogously, assume that the sequence $K_N$ fulfills the assumptions
\eqref{2-72} so that, by \eqref{2-67}, $\kappa_N(T)/\ell^d_N \to 0$.
In Lemma \ref{2-l11}, as $H_N(f^N_s)$ is bounded by
$\kappa_N(T) (N/\ell)^d$, estimate $\log 2$ by
$\kappa_N(T) (N/\ell)^d$, and set
$\gamma_s = (\kappa_N(T)/\ell^d)^{1/2}/[1+\Vert G(s)\Vert_\infty]$ to
get that
\begin{equation}
\label{2-65}
\begin{aligned}
& \bb E_{\nu^N} \Big[\, 
\int_0^t \Big| \, 
\sum_{x\in \bb T^d_N} G(s,Ax)
\, [\eta_x(s) - \bs u^N(x)\,]
\, \Big| \, ds \, \Big]
\\
&\quad
\le \;\mf c(\bs u^N) \, \frac{N^d}{\ell^{d/2}_N}\,
\sqrt{\kappa_N(T)}\, 
\int_0^t \{\, 1 + \Vert G(s) \Vert_\infty\,\} \,   ds
\end{aligned}
\end{equation}
for all $0\le t \le T$, $G\colon \bb R_+ \times \bb T^{d}_N \to \bb R$
because $\kappa_N(T)/\ell^d_N \to 0$.

\section{Proof of Theorem \ref{mt2}}
\label{sec5}

Throughout this section, $d\le 2$, and $T\ge 1$.  The proof of Theorem
\ref{mt2} relies on several lemmata.  The following elementary
estimate will be used repeatedly.  It explains the introduction of the
average \eqref{2-69} in the definition of the fluctuation field to get
the uniform estimates in \eqref{2-62}.  For a function
$G\in \ms D(\bb R \times \bb T^{d-1})$, denote by $\nabla_{N,j} G$,
$1\le j\le d$, the discrete partial derivatives of $G$ given by
\begin{equation*}
{\color{blue} (\nabla_{N,j} G)(Ax)} \,: =\,
N\,  [\, G(A(x+e_j)) - G(Ax)\, ] \;.
\end{equation*}
There exists a finite contant $C_0$ such that
\begin{equation}
\label{2-62}
\begin{gathered}
\frac{\sqrt{K}} {N^d} \sum_{x\in \bb T^d_N} \widehat G(Ax)^2 
\;\le\; \Vert G \Vert^2_{2} \;,
\quad
\frac{1}{\sqrt{K}\, N^d} 
\sum_{x\in \bb T^d_N} 
[(\nabla_{N,1} \widehat G )(Ax)]^2 \;\le\; 
C_0\, \Vert \partial_\vartheta G \Vert^2_{2}\;,
\\
\frac{1}{\sqrt{K}\, N^d} 
\sum_{x\in \bb T^d_N} 
[(\nabla_{N,j} \widehat G )(Ax)]^2 \;\le\; C_0\, 
\frac{1}{K}\, \, \Vert \partial_{\theta_j} G \Vert^2_{2}  \;,
\end{gathered}
\end{equation}
for all $2\le j\le d$, $N\ge 1$ and functions $G$ in
$\ms D (\bb R \times \bb T^{d-1})$. Here and below, $\color{blue} C_0$
represents a constant which depends only on the model (the dimension
$d$ and the cylinder function $c_0$), and whose value may change from
line to line. Moreover, for a real function $G$ defined on a set
$\Lambda$ (which can be $\bb R \times \bb T^{d-1}$ or
$\sqrt{K} \bb T \times \bb T^{d-1}$), $\Vert G \Vert_2$ represents its
$L^2$ norm:
$\color{blue} \Vert G \Vert_2^2 = \int_{\Lambda} G(z)^2 \, dz$.

We start with a simple consequence of the entropy inequality, which
shows the local ergodicity.  For a cylinder function $h$, let
$R_{h} \colon \bb T^d_N \to \bb R$ be the function given by
\begin{equation}
\label{2-03}
{\color{blue} R_{h}(x)} \;:=\; E_{\nu^N} [\,\tau_x h\,] \;.
\end{equation}
If $h$ is the function $h_0$ defined by
\begin{equation}
\label{2-63}
{\color{blue} h_0(\eta) } \;:=\; [1-2\eta_0]\, c_0(\eta)\;,
\end{equation}
we represent $R_{h_0}(x)$ by
\begin{equation}
\label{2-14}
{\color{blue} R (x)} \;:=\; E_{\nu^N} [\,\tau_x h_0\,] \;.
\end{equation}
The proof of the next result is similar to the one of Lemma \ref{l05}
and left to the reader. 

\begin{lemma}
\label{l2-02}
Fix a cylinder function $f$. Then, there exists a finite constant
$C_1 = C_1(f)$ such that
\begin{align*}
& \bb E_{\nu^N} \Big[\,
\int_0^t \,\Big|\, \sum_{x\in \bb T^d_N} J_x \,
[\, f(\tau_x \eta^N(s)) - R_f(x) \,]  \,\Big| \,  \, ds \, \Big] \\
&\qquad \le\;
\frac{1}{\lambda} \, \int_0^t  \big\{\, H_N(f^N_s)
\,+\, \log 2 \, \big\} \, ds
\;+\; 
C_1 \, \lambda\, e^{C_1 \, \lambda\,  \Vert  J \Vert_\infty}
\, \sum_{x\in \bb T^d_N} J_x^2 \; t
\end{align*}
for every function $J:\bb T^d_N \to \bb R$, $t>0$, $N\ge 1$,
$\lambda>0$.
\end{lemma}

Fix a test function $F \in \ms D(\bb R\times \bb T^{d-1}) $. For $K$
sufficiently large, $F$ can be understood as a smooth function defined
on $\sqrt{K}\, \bb T \times \bb T^{d-1}$.  Let $(M^N_t: t\ge 0)$ be
the Dynkin martingale
\begin{equation}
\label{2-60b}
{\color{blue} M^{N}_t (F)} \;:=\; X^N_t(F) \,-\, X^N_0(F) \,-\,
\int_0^t L_N X^N_s(F) \; ds \;.
\end{equation}
By \cite[Lemma A1.5.1]{kl},
\begin{gather}
\nonumber
{\color{blue} M^{N,2}_t (F)}  \,:=\, 
M^{N}_t(F) ^2 \;-\; \int_0^t \Gamma^{N,2}_s(F) \, ds\;,
\\
\text{where}\quad 
{\color{blue} \Gamma^{N,2}_s(F)} \,:=\,
L_N X^N_s(F)^2 \,-\, 2\, X^N_s(F)  \, L_N X^N_s(F)\;, 
\label{2-68}
\end{gather}
is a martingale. A straightforward computation yields that
\begin{equation}
\label{47-b}
\Gamma^{N,2} (F) 
\;=\; \frac{1}{\sqrt{K}\, N^d} \sum_{j=1}^d
\sum_{x\in \bb T^d_N} [\eta_{x+e_j} - \eta_x]^2\,
[(\nabla_{N,j}  \widehat F )(Ax)]^2
\; +\; \frac{\sqrt{K}}{N^d} \sum_{x\in \bb T^d_N} c_0(\tau_x \eta)\, 
\widehat F(Ax)^2 
\end{equation}
for $N$ sufficiently large.  Clearly,
\begin{equation*}
\Gamma^{N,2}_2(F) 
\;\le \; \frac{1}{\sqrt{K}\, N^d} \sum_{j=1}^d
\sum_{x\in \bb T^d_N} 
[(\nabla_{N,j} \widehat F )(Ax)]^2
\; +\; \frac{C_0\, \sqrt{K}}{N^d} \sum_{x\in \bb T^d_N} 
\widehat F(Ax)^2 \;.
\end{equation*}
By \eqref{2-62},
\begin{equation}
\label{2-70}
\sup_{\eta\in\Omega_N} \Gamma^{N,2}(F) \;\le\; C_0\,\Big\{
\, \Vert F\Vert^2_{2} \,+\,
\Vert \partial_\vartheta F\Vert^2_{2} \,+\, (1/K)
\Vert \nabla_\theta F \Vert^2_{2} \,\Big\}  \;,
\end{equation}
where
$\color{blue} \nabla_\theta F = (\partial_{\theta_2} F, \dots,
\partial_{\theta_d}  F)$. As $F$ belongs to $\ms D (\bb R \times \bb
T^{d-1})$, the right-hand side is bounded so that
\begin{equation}
\label{35-b}
\bb E^N_{\eta} \big[\, M^{N}_t(F)^2\,\big] \;\le\;
C_0\, \big\{\, \Vert F\Vert^2_{2} + \Vert \nabla F\Vert^2_{2}\, \big\}
\, t
\end{equation}
for all $t>0$, $\eta\in \Omega_N$.

The proof of the next result uses Lemma \ref{l2-02} for
$J_x = a\, F (Ax) + b \cdot \nabla F (Ax)$, $a\in \bb R$,
$b\in\bb R^d$. It requires, in particular, $F$ and $\nabla F$ to be
bounded. Mind that only the derivative in the first coordinate appears
in the covariance structure of the field, and that the noise has a
conservative and a non-conservative part.

\begin{lemma}
\label{l38}
The sequence of measures $\bb Q^M_N = \bb P_{\nu^N} \circ (M^N)^{-1}$
on $D([0,T], \ms D'(\bb R \times \bb T^{d-1}))$ converges, as
$N\to\infty$, to the centered Gaussian random field whose covariances
are given by
\begin{equation}
\label{x4}
\begin{aligned}
\bb Q^M \big[\, M_t(F) \, M_s(G)\,\big]\; & =\;
2\, (s\wedge t)\,  \int_{\bb R} d\vartheta \; 
\chi (\phi  (\vartheta)) \, \int_{\bb T^{d-1}} d\theta \;
(\partial_\vartheta F)(\vartheta, \theta)\,
(\partial_\vartheta G)(\vartheta, \theta)\,
\\
& +\; (s\wedge t)\, \int_{\bb R} d\vartheta \;
\widehat c_0(\phi  (\vartheta)) \,
\int_{\bb T^{d-1}} d\theta \; F(\vartheta, \theta)\,
G (\vartheta, \theta)\,   \;
\end{aligned}
\end{equation}
for all $F$, $G$ in $C^\infty(\bb R \times \bb T^{d-1})$ with compact
support. Here, $\bb Q^M$ stands also for the respective expectation.
\end{lemma}

\begin{proof}
By \cite[Th\'eor\`eme IV.1]{Fou} and Lemma \ref{2-l16}, the sequence
$\bb Q^{M}_{N}$ is tight. It remains to check the uniqueness of limit
points. Denote by $\bb Q^M$ one of them and assume, without loss of
generality, that the sequence $\bb Q^M_{N}$ converges to $\bb Q^M$.

Fix a test function $F$ in $\ms D (\bb R \times \bb T^{d-1})$.  By
\eqref{35-b}, the martingale $M^{N}_t (F)$ is uniformly bounded in
$L^2(\bb P^N_{\nu^N})$.  Therefore, under the measure $\bb Q^M$,
$M_t(F)$ is a martingale.

Recall the definition of the martingale $M^{N,2}_t (F)$ introduced in
\eqref{2-68} and the formula \eqref{47-b} for the compensator
$\Gamma^{N,2}_s (F) $. At this point we apply Lemma \ref{l2-02} with
$J_x = K^{-1/2} \, N^{-d} [(\nabla_{N,j} \widehat F )(Ax)]^2$,
$f(\eta) = [\eta_{e_j} - \eta_0]^2$, and
$J_x = K^{1/2} \, N^{-d} [\widehat F (Ax)]^2$, $f(\eta) =
c_0(\eta)$. In both cases, $|J_x|$ is uniformly bounded (in $x$ and
$N$) by
$C_0 (\sqrt{K}/N^d) (\Vert \nabla F\Vert_\infty^2 + \Vert
F\Vert_\infty^2)$ for some finite constant $C_0$ independent of $N$
and $F$. By \eqref{2-67}, the aforementioned lemma and Corollary
\ref{m-cor}, choosing a sequence $\lambda_N$ such that
$(N/\ell_N)^d \kappa_N(T) \ll \lambda_N\ll N^d/K_N$, yields that
\begin{equation*}
\lim_{n\to\infty} \bb E_{\nu^N} \Big[\,
\int_0^T \,\Big|\, \Gamma^{N,2}_s(F)
\,-\, I_N(F) \,\Big| \,  \, ds \, \Big]
\;=\;0\;.
\end{equation*}
where
\begin{equation*}
\begin{aligned}
I_N(F) \; &=\;
\frac{1}{\sqrt{K}\, N^d} \sum_{j=1}^d
\sum_{x\in \bb T^d_N} 
[(\nabla_{N,j} \widehat F )(Ax)]^2 \, [\, \bs u^N(x+e_j) + \bs u^N(x)  -  2\bs
u^N(x) \bs u^N(x+e_j) \,]  \\
\; & +\; \frac{\sqrt{K}}{N^d} \sum_{x\in \bb T^d_N} 
\widehat F (Ax)^2\, R_{c_0}(x)\;,
\end{aligned}
\end{equation*}
and $R_{c_0}(x)$ has been introduced in \eqref{2-03}.  By the
definition \eqref{2-69} of $\widehat F$ and Lemma \ref{CP_lemma},
\begin{equation*}
\begin{aligned}
\lim_{N\to\infty} I_N(F) \;=\; I(F) \; & : =\; 2\,  \int_{\bb R} d\vartheta
\, \chi (\phi (\vartheta)) \, \int_{\bb T^{d-1}} d\theta \;
[(\partial_\vartheta F)(\vartheta, \theta)]^2
\\ \; & +\; \int_{\bb R} d\vartheta \;
\widehat c_0(\phi (\vartheta)) \, 
\int_{\bb T^{d-1}} d\theta \; F(\vartheta, \theta)^2 \;.
\end{aligned}
\end{equation*}

We estimate the martingale $M^{N,2}_t (F)$ $L^4$-norm based on
\cite[Lemma B.3]{jl}. For $\Psi\colon \Omega_N \to \bb R$, let
$\Xi^N_k(\eta, \Psi)$, $k\ge 2$, be given by
\begin{equation*}
\Xi^N_k(\eta,\Psi) \;:=\; \sum_{\xi\in \Omega_N} R_N(\eta,\xi)\,
\big[\, \Psi(\xi) - \Psi(\eta)\,\big]^k\;,
\end{equation*}
where $R_N(\eta,\xi)$ stands for the rate at which the process
$\eta^N(t)$ jumps from $\eta$ to $\xi$. Note that
$\Xi^N_2(\eta,X^N(F)) = \Gamma^{N,2}(F)$.  A straightforward
computation yields that there exists a finite constant $\mf c (F)$,
depending only on the support of $F$, such that
\begin{equation*}
|\, \Xi^N_k(\eta,X^N(F))\, | \, \le\, \mf c  (F) \,
\{\, \Vert F\Vert^k_\infty + \Vert \nabla F\Vert^k_\infty \,\}
\quad \text{for $2\le k\le 4$, $N\ge 1$, $\eta\in\Omega_N$}\;.
\end{equation*}
Therefore, by \cite[Lemma B.3]{jl}, there exists a finite constant
$\mf c(F)$, depending only on the support of $F$, such that
\begin{equation*}
\bb E^N_\eta\big[\, M^{N}_t(F)^4 \,\big] \;\le\;
\mf c (F)\, \int_0^t \bb E^N_\eta\big[\, c_4 \,+\, c_3\, |\, M^{N}_s(F)\,|
\,+\, c_2\, M^{N}_s(F)^2 \,\big]\; ds\;,
\end{equation*}
where $c_k = \Vert F\Vert^k_\infty + \Vert \nabla F\Vert^k_\infty$.
Hence, by \eqref{35-b} and Young's inequality, 
\begin{equation}
\label{36}
\bb E^N_\eta\big[\, M^{N}_t(F)^4 \,\big] \;\le\;
\mf c (F) \, (1+T)^2 \, \{\, 1 + \Vert F\Vert^4_\infty + \Vert \nabla
F\Vert^4_\infty \,\} 
\end{equation}
all $0\le t\le T$, $\eta\in\Omega_N$, $N\ge 1$.

By definition of the martingale $M^{N,2}_t (F)$, \eqref{2-70}, and
\eqref{36}, the martingale $M^{N,2}_t (F)$ is uniformly bounded and
uniformly integrable in $L^2(\bb P^N_{\nu^N})$.  Therefore, under the
measure $\bb Q^M$, $M_t(F)^2 \,-\, I(F) \, t$ is a martingale. In
particular, $M_t(F)$ is a time-changed Brownian motion, and $M_t(F)$ a
Gaussian random variable.  To complete the proof of the lemma, it
remains to compute the covariance of $M_t(F)$ and $M_s(G)$ through
polarization.
\end{proof}

For a smooth function
$F\colon [0,T] \times (\sqrt{K} \bb T \times \bb T^{d-1}) \to \bb R$,
denote by $(M^{N}_t (F(\cdot)) : 0\le t\le T)$ the Dynkin martingale
given by
\begin{equation}
\label{2-60}
M^{N}_s (F(\cdot))\;=\; X^N_s(F(s)) \,-\, X^N_0(F(0)) \,-\,
\int_0^s \big\{\, X^N_r(\partial_r F(r))
+ L_N X^N_r(F(r))\,\big\}\; dr \;.
\end{equation}
In Proposition \ref{2-p1} below, we consider a special case of
time-dependent functions.  Fix a smooth test function
$F \colon \bb R \times \bb T^{d-1} \to \bb R$ with compact support (as
before, considered as a function defined on
$\sqrt {K} \bb T \times \bb T^{d-1}$ for $N$ large enough). Fix
$0<t\le T$, and let $\color{blue} F(s) = T^K_{t-s} F$, $0\le s\le t$,
where $\color{blue} (T^K_r:r\ge 0)$ is the semigroup associated to the
operator $\mf A_K$ defined in \eqref{2-71}.  The goal is to rewrite
$X^N_t(F)$ into a stochastic integral $M^N_t (T^K_{t-\cdot} F)$ called
a mild form in Corollary \ref{2-cor1}.

\begin{proposition}
\label{2-p1}
Fix $0\le t \le T$ and a function
$F\in \ms D (\bb R \times \bb T^{d-1})$.  Let $F(s) = T^K_{t-s} F$,
$0\le s\le t$.  Then,
\begin{equation*}
\lim_{N\to \infty} \bb E_{\nu^N} \Big[\, \Big|\, 
\int_0^t \big\{\, X^N_s(\partial_s F(s))
+ L_N X^N_s(F(s))\,\big\}\; ds \,\Big | \,\Big]\;=\; 0\;.
\end{equation*}
\end{proposition}

The proof of this proposition is divided in several steps. We first
compute the expression inside braces.  A long but elementary
computation yields that it is equal to
\begin{equation}
\label{2-58}
\begin{aligned}
& \frac{1} {\sqrt{N^d  \sqrt{K}}}
\sum_{x\in \bb T^d_N} (\partial_s \widehat F)(s,Ax) \,
[\,\eta_x - \bs u^N (x)\,] \\
&\quad +\;
\frac{N^2} {\sqrt{N^d  \sqrt{K}}} \sum_{j=1}^d
\sum_{x\in \bb T^d_N} (D^2_j\widehat F)(s,Ax) \,
[\, \eta_x - \bs u^N(x)\,]  \\
&\quad +\;  \frac{K} {\sqrt{N^d  \sqrt{K}}}
\sum_{x\in \bb T^d_N} \widehat F (s,Ax) \, \big(\tau_x h_0 -
E_{\nu^N} [\tau_x h_0]\,\big) \\
&\quad +\; \frac{1} {\sqrt{N^d  \sqrt{K}}}
\sum_{x\in \bb T^d_N} \widehat F (s,Ax) \, \Big(\ \sum_{j=1}^d
(\Delta^{(j)}_N \bs u^N)(x)\, +\, K \, E_{\nu^N} [\tau_x
h_0]\, \Big) \;.
\end{aligned}
\end{equation}
Note that the last term allows to center by $\bs u^N(x)$ and
$E_{\nu^N}[\tau_xh_0]$ in the second and third terms. In the above
formula and below,
\begin{equation}
\label{2-87}
\begin{gathered}
(D^2_jH)(Ax) \,=\, H(A(x+e_j)) + H(A(x-e_j)) - 2 H(Ax) \;,
\\
(\Delta^{(j)}_N h) (x) \;=\;
N^2\, \big\{\, h(x+e_j) + h (x-e_j) - 2 h (x) \,\big\}
\;, \quad 1\le j\le d\;, 
\end{gathered}
\end{equation}
and, recall, $\{e_j : 1\le j\le d\}$ represents the canonical basis of
$\bb R^d$.

In the next paragraphs, we estimate the terms in \eqref{2-58}. We
start with the last one which is a sequence of real functions (and not
of random variable as the other ones). We claim that in dimension
$d\le 3$,
\begin{equation}
\label{2-59}
\lim_{N\to\infty} \sup_{0\le s\le T}
\frac{1}{\sqrt{N^d\sqrt{K_N}}}\,
\sum_{x\in \bb T^d_N} \Big|\, \widehat F (s,Ax) \, \Big(\ \sum_{j=1}^d
(\Delta^{(j)}_N \bs u^N)(x)\, +\, K \, E_{\nu^N} [\tau_x
h_0]\, \Big) \, \Big| \;=\; 0\;.
\end{equation}

To prove this claim, first observe that, as $\bs u^N (\cdot)$ is
constant in the $k$-th coordinate, $k\neq 1$,
$\Delta^{(j)}_N \bs u^N =0$ for $j\neq 1$.  Thus, by Lemma
\ref{2-l18}, by definition of $\widehat F$, the last term in
\eqref{2-58} bounded by
\begin{equation*}
\frac{\mf c(\rho^K)\,  \, K^2}{N^2}\,
\frac{1}{\sqrt{N^d\sqrt{K}}}\, \sum_{x\in \bb T^d_N}  \big|\,
\widehat F (s,Ax) \, \big|
\;\le\;
\frac{\mf c(\rho^K)\,  \, K^2}{N^2}\,
\frac{ N^{d/2}}{K^{3/4}}
\int_{\sqrt{K} \bb T \times \bb T^{d-1}}   \big|\, F (s,
\vartheta,  \theta) \, \big|\;
d\vartheta\, d\theta
\end{equation*}
for some finite constant $\mf c(\rho^K)$ depending only on $c_0$ and
$\rho^K$. By \cite[Lemma B.7]{F1} and Schwarz inequality, the integral
is bounded by $C(F) K^{1/4}$ for some finite constant $\mf c_0$. To
complete the proof of the claim \eqref{2-59}, it remains to recall the
estimate \eqref{2-67}.

Note that the previous argument does not apply to dimension $d\ge 4$.
In high dimensions the approximation of the solution of the discrete
equation \eqref{2-93} by the solution of the continuous one
\eqref{2-56} is not sharp enough in the context of the requirements of
$K_N$ needed for other estimates. The next result permits to replace
the discrete Laplacians by the continuous one in the second term of
\eqref{2-58}.

\begin{lemma}
\label{lem:3.2-F}
There exists a finite constant $C_0$ such that
\begin{equation*}
\begin{aligned}
& \Big |\, \sum_{j=1}^d N^2 (D^2_j \widehat G)(Ax)\,-\, 
K\, (\partial^2_{\vartheta} \widehat G)(Ax) \,-\,
\sum_{k=2}^d (\partial^2_{\theta_k} \widehat G)(Ax)\, \Big|
\\
&\quad 
\, \le\, C_0\, \frac{1}{N} \, 
\Big\{\, K^{3/2}\, \| \partial^3_\vartheta G \|_\infty \,+\,
\sum_{j=2}^d \| \partial^3_{\theta_j} G \|_\infty \, \Big\} 
\end{aligned}
\end{equation*}
for all $N\ge 1$, $x \in \T_N^d$, and smooth
$G\colon \sqrt{K} \bb T \times \bb T^{d-1}\to \R$.
\end{lemma}

\begin{proof}
Fix a smooth function $G\colon \sqrt{K} \bb T\times \T^{d-1}\to
\R$. Write $x \in \T_N^d$ as $x= (x_1,\ux)$, where $x_1 \in \T_N$,
$ \ux \in \T_N^{d-1}$. The sum $\sum_{1\le j\le d} N^2 (D^2_j\widehat G)(Ax)$
can be decomposed into the sum $I_1+I_2$, where
\begin{gather*}
I_1 \,=\, \sum_{b=\pm 1}  N^2\, \big\{\,  \widehat G((x_1+b)\sqrt{K}/N, \ux/N)
- \widehat G(x_1\sqrt{K}/N, \ux/N)\, \big\} \, , 
\\
I_2 = \sum_{j=1}^{d-1} \sum_{b=\pm 1}
N^2\, \big\{\,  \widehat G (x_1\sqrt{K}/N, (\ux + b  \widehat e_j)/N)
- \widehat G (x_1\sqrt{K}/N, \ux/N)\, \big\} \;,
\end{gather*}
and $\{\widehat  e_1, \dots \widehat  e_{d-1}\}$ represents the canonical basis of
$\R^{d-1}$.

A third order Taylor expansion yields that
\begin{equation*}
I_1 \,=\, K\,  (\partial_\vartheta^2 \widehat G)(Ax) \,+\, R\;, \quad
\text{where}\quad
|R| \,\le\, \frac{K^{3/2}}{N}\,  \|\partial_\vartheta^3 G\|_\infty\; .
\end{equation*}
Similarly,
\begin{align*}
I_2 \,=\,  \sum_{j=2}^d (\partial^2_{\theta_j} \widehat G)(Ax) \,+\, R\;, \quad
\text{where}\quad
|R| \,\le\, \frac{1}{N} \, \sum_{j=2}^d
\| \partial^3_{\theta_j}  G\|_\infty\; .
\end{align*}
This complete the proof of the lemma.
\end{proof}

Let $J \colon [0,T] \times \bb T^d_N \to \bb R$ be given by
\begin{equation}
\label{2-92}
J(s,x) \,=\, \sum_{j=1}^d N^2 (D^2_j \widehat F)(s, Ax)\,-\, 
K\, \widehat{(\partial^2_{\vartheta} F)} (s,Ax) \,-\,
\sum_{j=2}^d  \widehat {(\partial^2_{\theta_j} F)} (s, Ax)\;.
\end{equation}
By the previous result and \eqref{2-65} applied to
$G= N^{-d/2}K^{-1/4}\, J$,
\begin{align*}
& \bb E_{\nu^N} \Big[\, 
\int_0^t \, \Big|\, \frac{1}{\sqrt{N^d \sqrt{K}}}
\sum_{x\in \bb T^d_N} J(s,Ax)
\, [\eta_x(s) - \bs u^N(x)\,]\, \Big|
\, ds \, \Big]
\\
&\quad \le
\;\mf c(\bs u^N) \, \frac{K^{5/4}}{N} \, \Big(\frac{N}{\ell_N}\Big)^{d/2} \,
\sqrt{\kappa_N(T)}\, 
\int_0^t \Big\{\, 1 + \Vert (\partial^3_{\vartheta} F)(s)
\Vert_\infty
+ \sum_{j=2}^d \Vert (\partial^3_{\theta_j} F)(s) \Vert_\infty
\,\Big\} \,   ds\;,
\end{align*}
where, recall, $F(s) = T^K_{t-s} F$. By Lemma \ref{maximum lemma}, the
previous integral is bounded by $C(G) K^3 \, e^{4 \mf c_5 K T}$. Thus,
by \eqref{2-90}, the previous expression vanishes as $N\to\infty$:
\begin{equation}
\label{2-57}
\lim_{N\to\infty}\, \bb E_{\nu^N} \Big[\, 
\int_0^t \, \Big|\, \frac{1}{\sqrt{N^d \sqrt{K}}}
\sum_{x\in \bb T^d_N} J(s,Ax)
\, [\eta_x(s) - \bs u^N(x)\,]\, \Big|
\, ds \, \Big]  \;=\; 0\;.
\end{equation}

We turn to the third term in \eqref{2-58}.  Recall the definition of
$\Xi^{\rm c}_{\bs u^N (\cdot) , x} f \colon \Omega_N\to \bb R$ given
in \eqref{2-35b}.  By the Boltzmann-Gibbs principle \eqref{2-46},
\begin{equation}
\label{2-64}
\begin{aligned}
& \bb E_{\nu^N} \Big[\, \Big|\,
\int_0^t \, \frac{K^{3/4}}{N^{d/2}} \,
\sum_{x\in \bb T^d_N} \widehat F (s,Ax)\,
(\Xi^{\rm c}_{\bs u^N (\cdot), x}  f) (\eta (s))
\, ds \, \Big| \,\Big]
\\ & \quad
\;\le \; \mf c(\bs u^N) \, K^{7/4}\,
\, \Big(\frac{N}{\ell_N}\Big)^d \frac{1}{N^{d/2}} \, \kappa_N(T)
\int_0^t [\, 1 + \mf m_K(F_s)\,] \, ds \;.
\end{aligned}
\end{equation}
By Lemma \ref{maximum lemma}, $1 + \mf m_K(F_s)$ is bounded by
$C(F)\, K\, e^{2(\mf c_5 K +1)T}$.  By \eqref{2-90}, the right-hand
side vanishes as $N\to\infty$. This estimate is the only one which
requires the dimension to be less than $3$ in the estimation of the
third term.

By the previous estimate, in the third term of \eqref{2-58}, we may
replace $\tau_x h_0 - E_{\nu^N}[\tau_x h_0]$ by
$\Xi^{\rm c}_{\bs u^N (\cdot), h_0}(x)\, [\eta_x - \bs u^N(x)]$.  Let
$J \colon [0,T] \times \bb T^d_N \to \bb R$ be given by
\begin{equation*}
J(s,x) \,=\, \widehat F (s,Ax)\,
\big\{\, \Xi^c_{\bs u^N (\cdot), h_0}(x) \,+\, V''(\bs u^N(x))
\,\big\} \;.
\end{equation*}
By Lemma \ref{lem:3.1-F} and \eqref{2-65} applied to $G =
(N/\sqrt{K})\, J$,
\begin{equation*}
\begin{aligned}
& \bb E_{\nu^N} \Big[\, 
\int_0^T \, \Big|\, \frac{K}{\sqrt{N^d \sqrt{K}}}
\sum_{x\in \bb T^d_N} J(s,Ax)
\, [\eta_x(s) - \bs u^N(x)\,]\, \Big|
\, ds \, \Big]  \\
&\quad \;\le \;
\mf c(\bs u^N) \, \Big(\frac{N}{\ell_N}\Big)^{d/2}\,
\frac{K^{5/4}}{N}\, 
\sqrt{\kappa_N(T)}\, \int_0^T \big \{\, 1\,+\,
\Vert T^K_{s} F \Vert_\infty \,\big\} \, ds
\;.
\end{aligned}
\end{equation*}
Note that $\|T^K_s F\|_\infty \leq e^{(\frak{c_5}K+1)T}$ by Lemma
\ref{maximum lemma}.

Since the right-hand side of \eqref{2-64} vanishes as $N\to\infty$, by
\eqref{2-89},
\begin{equation}
\label{2-55}
\begin{aligned}
& \lim_{N\to\infty}\,
\bb E_{\nu^N} \Big[\, \Big|\,
\int_0^t \, \frac{K^{3/4}}{N^{d/2}} \,
\sum_{x\in \bb T^d_N} \widehat F (s,Ax)\, \Psi(x, \eta(s)) 
\, ds \, \Big| \,\Big] \; = \; 0\;.
\\
&\quad \text{where}\;\;
\Psi(x, \eta )  \;:=\;
\tau_x h_0 (\eta ) \,-\, E_{\nu^N} [\tau_x h_0]
\, + \, V''(\bs u^N(x)) \,
[\eta_x - \bs u^N(x)\,]\;.
\end{aligned}
\end{equation}

\begin{proof}[Proof of Proposition \ref{2-p1}]
Recall formula \eqref{2-58} for
$X^N_s(\partial_s F(s)) + L_N X^N_s(F(s))$.  By \eqref{2-59}, the last
term vanishes. Therefore, by \eqref{2-55} and \eqref{2-57},
\begin{equation*}
\begin{aligned}
& \limsup_{N\to \infty} \bb E_{\nu^N} \Big[\, \Big|\, 
\int_0^t \big\{\, X^N_s(\partial_s F(s))
+ L_N X^N_s(F(s))\,\big\}\; ds \,\Big | \,\Big]
\\
&\quad
\le\; \limsup_{N\to \infty} \bb E_{\nu^N} \Big[\, \Big|\, 
\int_0^t 
\frac{1} {\sqrt{N^d  \sqrt{K}}}
\sum_{x\in \bb T^d_N} \widehat{H_K} (s,x) 
\, [\,\eta_x(s)  - \bs u^N (x)\,] \; ds \,\Big | \,\Big] \;,
\end{aligned}
\end{equation*}
where
\begin{equation*}
H_K(s,x) \; =\;  (\partial_s F)(s,Ax) \,+\, \; (\mf A_K F) (s,Ax)
\;=\; 0 \;.
\end{equation*}
To complete the proof it remains to observe that
$\partial_s F \,=\, \partial_s T^K_{t-s} F \,=\, -\, \mf A_K F$,
where $\mf A_K$ is defined in \eqref{2-71}.
\end{proof}

\begin{corollary}
\label{2-cor1}
Under the hypotheses of Proposition \ref{2-p1},
\begin{equation*}
\lim_{N\to\infty} 
\bb E_{\nu^N} \big[ \, \big|\, X^N_t(F) \,-\, M^{N}_t(F(\cdot)) \,\big|\,
\big] \;=\; 0\;.
\end{equation*}
\end{corollary}

\begin{proof}
In view of Proposition \ref{2-p1}, it is enough to show that
\begin{equation*}
\lim_{N\to\infty} 
\bb E_{\nu^N} \big[ \, \big|\, X^N_0(T^K_t F) \,\big|\,
\big] \;=\; 0\;.
\end{equation*}
As $\nu^N$ is a product measure, by Schwarz inequality and by the
definition \eqref{2-69} of the average $\widehat G$, the square of the
expectation is bounded by
\begin{equation*}
\frac{1}{N^d K^{1/2} }
\sum_{x\in \bb T^d_N} [\, \widehat F (t,Ax)\,]^2\,  \chi(\bs u^N(x))
\;\le\; \frac{1}{K} \, \Vert T^K_t F \Vert^2_2 \;.
\end{equation*}
By \cite[Lemma B.7]{F1}, $\Vert T^K_tF \Vert^2_2$ is uniformly
bounded, which completes the proof of the lemma.
\end{proof}

Recall from \eqref{2-83} the definition of the function
$\bs e\colon \bb R \to \bb R$, and that $(S_t:t\ge 0)$ represents the
semigroup of the heat equation in $\bb T^{d-1}$. Denote by
$\bs e_K\colon \sqrt{K} \bb T \to \bb R$ the function $\bs e$ immersed
in $\sqrt{K} \bb T$. More precisely, if
$\sqrt{K} \bb T \cong [-\sqrt{K}/2, \sqrt{K}/2)$, $\bs e_K$ and
$\bs e$ coincide on $[-\sqrt{K}/4, \sqrt{K}/4]$ and $\bs e_K$ smoothly
interpolates on the rest of the torus so that
$|\partial^j_\vartheta \bs e_K|\leq C|\partial^j_\vartheta \bs e|$ on
$\sqrt{K}\bb T$ for $j=0,1,2$.

Denote by
$\color{blue} (P^K_t:t\ge 0)$ the semigroup associated to the
generator $K\, \ms A_K$ introduced in \eqref{2-49}. In particular, for
smooth functions $F\colon \sqrt{K}\bb T \to \bb R$,
$G\colon \bb T^{d-1} \to \bb R$, $T^K_t (FG) = P^K_t F\, S_t G$.

Fix $0\le t\le T$, $p\ge 1$, and functions $F_j\in C_c^\infty(\bb R)$,
$G_j\in C^\infty( \bb T^{d-1})$, $1\le j\le p$. Consider $F_j$ as
defined on $\sqrt{K}\bb T$ for $K$ sufficiently large. Let
$F_j^{(1)} (s, \vartheta) = (P^{K}_{t-s} F_j)(\vartheta)$,
$F^{(2)}_j(\vartheta) = \bs e_K (\vartheta) \, \<\bs e, F_j\>$,
$G_j(s, \theta) = (S_{t-s} G_j)(\theta)$,
\begin{equation*}
\begin{gathered}
H^{(1)}_s(\vartheta, \theta) \,=\, \sum_{j=1}^p
F_j^{(1)} (s, \vartheta)\, G_j(s, \theta)\;, \quad
H^{(2)}_s(\vartheta, \theta) \,=\, \sum_{j=1}^p
F_j^{(2)} (\vartheta)\, G_j(s, \theta)\;,
\\
H_s(\vartheta, \theta) \,=\,
H^{(1)}_s(\vartheta, \theta)\,-\,
H^{(2)}_s(\vartheta, \theta)\;. 
\end{gathered}
\end{equation*}
We comment by density arguments that it is sufficient to identify
martingale limits with respect to these linear combinations of
products of functions of $\vartheta$ and $\theta$.

In these formulas, as before, $\< F, G\>$ stands for the scalar
product in $L^2(\bb R)$.  Therefore, if
$\color{blue} H (\vartheta, \theta) := \sum_{1\le j\le p}
F_j(\vartheta)\, G_j(\theta)$, with the notation introduced in
\eqref{2-75}, 
\begin{equation}
\label{2-73}
\begin{gathered}
H^{(1)}_s (\vartheta, \theta)  \;=\; (T^K_{t-s} H) (\vartheta,
\theta)\;, 
\\
H^{(2)}_s (\vartheta, \theta) \,=\, \bs e_K (\vartheta)
\sum_{j=1}^p \< F_j\>\,
(S_{t-s} G_j) (\theta)
\,=\, \bs e_K (\vartheta) \, S_{t-s} \<H\> (\theta)  \;.
\end{gathered}
\end{equation}
Mind that $H^{(1)}_t = H$, while $H^{(2)}_t = \bs e_K \, \<H\>$.  By
\eqref{CP properties}, $\bs e_K$ converges to $\bs e$ in $L^2(\bb
R)$. Hence, $H^{(2)}_s \to S_{t-s}\, \Pi H$, as $N\to\infty$.

\begin{lemma}
\label{2-l04}
Assume that $d\le 2$, and fix $0\le t \le T$.  Then,
\begin{equation*}
\lim_{N\to\infty} 
\bb E_{\nu^N} \big[ \, \big\{\, M^{N}_t(H^{(1)} (\cdot))
\,-\, M^{N}_t(H^{(2)} (\cdot))  \,\big\}^2\,
\big] \;=\; 0\;.
\end{equation*}
\end{lemma}

\begin{proof}
By linearity, for $i=1$, $2$, 
$M^{N}_t(H^{(i)} (\cdot)) = \sum_{1\le j\le p} M^{N}_t(J^{(i)}_j
(\cdot))$, where
$J^{(i)}_j (s, \vartheta, \theta) = F^{(i)}_j (s, \vartheta) G_j (s,
\theta) $. It is therefore enough to prove that
\begin{equation*}
\lim_{N\to\infty} 
\bb E_{\nu^N} \big[ \, \big\{\, M^{N}_t(J^{(1)}_j (\cdot))
\,-\, M^{N}_t (J^{(2)}_j (\cdot)) \,\big\}^2\,
\big] \;=\; 0
\end{equation*}
for $1\le j\le p$.

Fix such $j$. By linearity again,
$ M^{N}_t(J^{(1)}_j (\cdot)) \, -\, M^{N}_t(J^{(2)}_j (\cdot)) \,=\,
M^{N}_t(J_j (\cdot))$, where $J_j = J^{(1)}_j - J^{(2)}_j$. By the
formula for the quadratic variation of a Dynkin martingale,
\cite[Lemma A1.5.1]{kl},
\begin{equation*}
\bb E_{\nu^N} \big[ \, \big\{\, M^{N}_t(J^{(1)}_j (\cdot))
\,-\, M^{N}_t (J^{(2)}_j (\cdot)) \,\big\}^2
\big] \;=\; \bb E_{\nu^N} \Big[ \,
\int_0^t \Gamma^N_2 \big( J_j(s)) \, ds \, \Big] \;,
\end{equation*}
where $\Gamma^N_2(\cdot)$ is defined in \eqref{47-b}. Since
$c_0(\cdot)$ is bounded, the previous expression is less than or equal
to
\begin{equation}
\label{2-66}
\int_0^t \Big\{\, \frac{1}{\sqrt{K}\, N^d} \sum_{k=1}^d
\sum_{x\in \bb T^d_N} 
[(\nabla_{N,k} \widehat J_j )(s,Ax)]^2
\; +\; \frac{\sqrt{K}}{N^d} \sum_{x\in \bb T^d_N} 
\widehat J_j (s, Ax)^2 \,\Big\}\, ds \;.
\end{equation}

Recall the definition of $J_j(s)$, $F^{(i)}_j(s)$, $G_j(s)$.  By the
maximum principle for the heat equation, $G_j(r,\theta)$ is bounded by
$\Vert G_j \Vert_\infty$, $0\le r\le t$. Therefore,
by Schwarz inequality and the definition \eqref{2-69} of $\widehat G$,
the second term is bounded by
\begin{equation*}
\Vert G_j \Vert^2_{\infty} \, 
\int_0^t \Vert F^{(1)} _j (s) \,
-\, F^{(2)} _j (s) \Vert^2_{2} 
\; ds \;.
\end{equation*}
Similarly, as $\partial_{\theta_k} G_j(r,\theta)$ is
bounded by
$\Vert \partial_{\theta_k} G_j \Vert_\infty$ the
first term in \eqref{2-66} is bounded by
\begin{equation*}
\begin{aligned}
& \frac{1}{K} \, \sum_{k=2}^{d}\Vert
\partial_{\theta_k} G_j \Vert^2_{\infty} 
\int_0^t \, \Vert F^{(1)} _j (s) \,
-\, F^{(2)} _j \Vert^2_{2}
\; ds 
\\
&\quad +\;
\Vert G_j \Vert^2_{\infty}  \,
\int_0^t \big \Vert \partial_\vartheta 
F^{(1)} _j (s) \, -\, \partial_\vartheta  F^{(2)} _j 
\,  \big \Vert^2_{2}\; ds \;.
\end{aligned}
\end{equation*}
By \cite[Proposition B.4]{F1}, the penultimate display vanishes as
$N\to\infty$.

To complete the proof of the lemma, it remains to recall from Lemma
\ref{2-l19} that the last integral also vanishes as $K\to\infty$ for
all $t>0$.
\end{proof}

\begin{remark}
\label{rm2}
The proof of this lemma shows that in non-equilibrium it is easier to
estimate martingales than the density field. For martingales, passing
to the quadratic variation leads to a sum divided by $N^d$. If we had
to deal with the density field directly, one would need to estimate a
sum divided by $N^{d/2}$ through the entropy inequality. In view of
Lemma \ref{2-l11}, as the entropy is large, the free parameter
$\gamma$ has to be large, which would force the test function $G$ to
be very small.
\end{remark}

Recall the notation introduced in the previous lemma.  In particular,
the definition of
$H \colon \sqrt{K} \bb T \times \bb T^{d-1} \to \bb R$.  By
\eqref{2-73}, Corollary \ref{2-cor1}, and Lemma \ref{2-l04},
\begin{equation}
\label{2-74}
\lim_{N\to\infty} 
\bb E_{\nu^N} \big[ \, \big|\, X^N_t(H)
\,-\, M^{N}_t ( H^{(2)} (\cdot)) \,\big|\,
\big] \;=\; 0\;.
\end{equation}

\begin{proof}[Proof of Theorem \ref{mt2}]
Fix $0\le t\le T$, $p\ge 1$, and functions $F_j\in C_c^\infty(\bb R)$,
$G_j\in C^\infty( \bb T^{d-1})$, $1\le j\le p$. Recall the notation
introduced in Lemma \ref{2-l04}.  By \eqref{2-74}, we need to compute
the asymptotic behavior of the finite-dimensional distributions of
$M^{N}_t (H^{(2)} (\cdot))$.

By equation (8.1) in \cite{jl}, and since $H^{(2)}_t = \bs e_K \<H\>$
and $\partial_s H_s^{(2)}=\Delta_\theta H_s^{(2)}$,
\begin{equation}
\label{2-76}
M^{N}_t (H^{(2)} (\cdot))  \;=\; M^N_t(\bs e_K \<H\>) \,+\, 
\int_0^t  M^N_s(\Delta_\theta H^{(2)}_s) \; ds \;,
\end{equation}
where on the right-hand side $M^N_s (G_s)$ has to be understood as the
martingales $(M^{N}_r (G_s) : r\ge 0)$ introduced in \eqref{2-60b}
taken at time $r=s$. 

By definition of $H^{(2)} (\cdot)$,
$\partial_s H^{(2)}_s \,=\, -\, \Delta_\theta H^{(2)}_s$. Hence, by
Lemma \ref{l38}, the uniform convergence of $\bs e_K$ to
$\bs e$  (by \eqref{CP properties}),
$X^N_t(H)$ converges in distribution to the mean-zero
Gaussian random variable
\begin{equation*}
M_t(\Pi H) \,+\, 
\int_0^t  M_s(\Delta_\theta  \, S_{t-s} \Pi H) \; ds \;.
\end{equation*}

To complete the proof of the theorem, it remains to compute the
covariances of these variables. We present the details for the
variance, the covariance being computed by polarization.  Recall the
definition of $H^{(2)}$ given in \eqref{2-73}.  By Lemma \ref{l38},
$\bb Q^N[\, M_t(\Pi F)^2\,] $ is given by \eqref{x4}.  Since
$-\, \Delta_\theta H^{(2)}_s = \partial_s H^{(2)}_s$, by \eqref{x4}
and an integration by parts, since
$H^{(2)}_t = \bs e_K \<H\> \to \Pi H$ uniformly,
\begin{equation*}
\begin{aligned}
& 2\, \bb Q^M \big[\, M_t(\Pi H) \, \int_0^t
M_s(\Delta_\theta H^{(2)}_s)\,ds\,\big]
\\
\quad =\;
& - \, 4\, t\,  \int_{\bb R} d\vartheta \; 
\chi (\phi  (\vartheta)) \, \int_{\bb T^{d-1}} d\theta \;
(\partial_\vartheta \Pi H)(\vartheta, \theta)\,
(\partial_\vartheta \, \Pi H)(\vartheta, \theta)\,
\\
\quad & +\, 4\, \int_0^t ds \,
\int_{\bb R} d\vartheta \; 
\chi (\phi (\vartheta)) \, \int_{\bb T^{d-1}} d\theta \;
(\partial_\vartheta \Pi H)(\vartheta, \theta)\,
(\partial_\vartheta H^{(2)}_s)(\vartheta, \theta)\,+\, A_1\,,
\end{aligned}
\end{equation*}
where $A_1$ corresponds to the second term in \eqref{x4} (the one
without the derivatives) and which is similar to the expression
above. Finally, and by the same reasons,
\begin{equation*}
\begin{aligned}
& \bb Q^M \Big[\, \Big( \int_0^t
M_s(\Delta_\theta H^{(2)}_s)\,ds \Big)^2 \,\Big]
\\
\quad & =\;
2\, t\,  \int_{\bb R} d\vartheta \; 
\chi (\phi  (\vartheta)) \, \int_{\bb T^{d-1}} d\theta \;
(\partial_\vartheta \Pi H)(\vartheta, \theta)\,
(\partial_\vartheta \, \Pi H)(\vartheta, \theta)\,
\\
\quad & +\, 2\, \int_0^t ds \,
\int_{\bb R} d\vartheta \; 
\chi (\phi (\vartheta)) \, \int_{\bb T^{d-1}} d\theta \;
(\partial_\vartheta H^{(2)}_s)(\vartheta, \theta)\,
(\partial_\vartheta H^{(2)}_s)(\vartheta, \theta)
\\
\quad & -\, 4\, \int_0^t ds \,
\int_{\bb R} d\vartheta \; 
\chi (\phi (\vartheta)) \, \int_{\bb T^{d-1}} d\theta \;
(\partial_\vartheta \Pi H)(\vartheta, \theta)\,
(\partial_\vartheta H^{(2)}_s)(\vartheta, \theta)\,+\, A_2\,,
\end{aligned}
\end{equation*}
with the same convention for $A_2$ as above. Adding all terms, and
since $H^{(2)}_s = S_{t-s} \Pi H$, yields
that
\begin{equation*}
\begin{aligned}
& \bb Q^M \Big[\, \Big( M_t (\Pi H) \,+\, \int_0^t
M_s(\Delta_\theta S_{t-s} \Pi H )\,ds \Big)^2 \,\Big]
\\
\quad & = \, 2\, \int_0^t ds \,
\int_{\bb R} d\vartheta \; 
\chi (\phi (\vartheta)) \, \int_{\bb T^{d-1}} d\theta \;
(\partial_\vartheta S_{t-s} \Pi H )(\vartheta, \theta)^2\,
\\
\quad & + \,  \int_0^t ds \,
\int_{\bb R} d\vartheta \; 
\widehat c_0  (\phi (\vartheta)) \, \int_{\bb T^{d-1}} d\theta \;
(S_{t-s} \Pi H )(\vartheta, \theta)^2\,,
\end{aligned}
\end{equation*}
as claimed.
\end{proof}

\section{Proof of Theorem \ref{mt3}}
\label{sec6}

Fix $T>1$ and recall that we denote by $\bb Q_N$ the probability
measure on $C([0,T], \ms D')$ induced by the process
$\bb X^N_t =\int_0^t X^N_s\, ds$ and the initial product measure
$\nu^N$. 

\begin{proof}[Proof of Theorem \ref{mt3}]
As $\ms D(\bb R \times \bb T^{d-1})$ is a LF-space (a locally convex
inductive limit of a sequence of Fr\'echet spaces), according to
\cite[Th\'eor\`eme IV.1]{Fou}, it is enough to prove that for each
$F\in \ms D(\bb R\times \bb T^{d-1} )$ the sequence of measures on
$C([0,T], \bb R)$ induced by the process $\bb X^N_t(F)$ is tight.

By definition,
\begin{equation*}
\bb X^N_t(F) \,=\,
\int_0^t X^N_r(F)\, dr\;, \quad 0\le t\le T\;.
\end{equation*}
Fix $r\in [0,T]$. By equation \eqref{2-60} with $F(q) = T^K_{r-q} F$,
$0\le q\le r$, 
\begin{equation*}
X^N_r(F)\,=\, 
M^{N}_r (F(\cdot)) \,+ \, X^N_0(T^K_r F) \,+\,
\int_0^r \big\{\, X^N_q(\partial_q T^K_{r-q} F)
+ L_N X^N_q(T^K_{r-q} F)\,\big\}\; dq \;.
\end{equation*}
We consider each term on the right-hand side separately.

In Lemma \ref{2-l16b} we show that the process
$\int_0^t M^N_r(F (\cdot))\, dr$ is tight  in $C([0,T], \bb R)$. On
the other hand, by Schwarz inequality and the proof of Corollary 
\ref{2-cor1},  
\begin{equation*}
\begin{aligned}
\bb E_{\nu_N} \Big[\, \sup_{|t-s|\le \delta}
\Big(\int_s^t X_0(T^K_r F) \, dr \Big)^2 \,\Big]
\, & \le\,
\delta\, \int_0^T \bb E_{\nu_N} \Big[\, 
X_0(T^K_r F) ^2 \Big] \, dr \,
\\
& \le\, \frac{C_0\, \delta}{K} \,
\int_0^T \Vert T^K_r F
\Vert_{2}^2 \; dr  
\end{aligned}
\end{equation*}
for some finite constant $C_0$. By \cite[Lemma B.7]{F1}, the integral
is bounded by
$T\, \Vert F\Vert_{2}^2$.
Therefore, the process $\int_0^t X^N_0(T^K_s F) \, ds$ is also tight
in $C([0,T], \bb R)$.

It remains to prove that the process
\begin{equation*}
\bb Y^N_t \,:=\, \int_0^t dr \int_0^r \big\{\, X^N_q(\partial_q F(q))
+ L_N X^N_q(F(q))\,\big\}\; dq
\end{equation*}
is tight in $C([0,T], \bb R)$. Writing $\bb Y^N_t = \int_0^t U^N_r\,
dr$, Proposition \ref{2-p1} showed that $\bb E_{\nu^N}[\, |U^N_r|\,]
\to 0$ for each $0\le r\le T$. If one can show that $\sup_{0\le r\le
T} \sup_N \bb E_{\nu^N}[\, |U^N_r|^p\,] < \infty$ for some $p>1$, we
get the tightness of $\bb Y^N_t$ since
\begin{equation*}
\bb E_{\nu^N}[\, |\bb Y^N_t - \bb Y^N_s|^p\,] \,=\,
\bb E_{\nu^N} \Big[\, \Big|\, \int_s^t U^N_r\; dr \, \Big|^p \,\Big]
\,\le\, (t-s)^{p/q} \, \bb E_{\nu^N}\Big[\, \int_s^t | U^N_r|^p \; dr  \,\Big]
\,\le\, C (t-s)^{1+q}\;,
\end{equation*}
where $(1/p) + (1/q)=1$ so that $q>0$. In Subsection 8.2, we derive
estimates with this flavour.

Recall the formula \eqref{2-58}. By
\eqref{2-59}, we may ignore the last term in \eqref{2-58}. By
\eqref{2-92}, \eqref{2-57}, in the second term of \eqref{2-58}, we may
replace the discrete Laplacians by the continuous ones. Once this has
been done, since $\partial_q T^K_{r-q} F = - \mf A_K T^K_{r-q} F$, the
sum of the first and second lines in \eqref{2-58} is given by
\begin{equation*}
\frac{K} {\sqrt{N^d  \sqrt{K}}} 
\sum_{x\in \bb T^d_N} \widehat F (q,Ax) \, V''(\bs u^N(x))\,
[\, \eta_x(q) - \bs u^N(x)\,] \;,
\end{equation*}
where $F(q, \cdot) = (T^K_{r-q} F )(\cdot)$.  Therefore, including now
the second line of \eqref{2-58}, to prove that the process $\bb Y^N_t$
is tight, it is enough to show that the process
\begin{equation*}
\begin{aligned}
\widehat {\bb Y}^N_t \, &
:=\, \int_0^t dr \int_0^r \Big\{\,
\frac{K} {\sqrt{N^d  \sqrt{K}}} 
\sum_{x\in \bb T^d_N} \widehat F (q,Ax) \, V''(\bs u^N(x))\,
[\, \eta_x(q) - \bs u^N(x)\,] \,\Big\}\; dq \\
& +\, \int_0^t dr \int_0^r \Big\{\,
\frac{K} {\sqrt{N^d  \sqrt{K}}}
\sum_{x\in \bb T^d_N} \widehat F (q,Ax) \, \big(\tau_x h_0 (\eta(q)) -
E_{\nu^N} [\tau_x h_0]\,\big) \,\Big\}\; dq
\end{aligned}
\end{equation*}
is tight.

By the expansion \eqref{2-19} of the cylinder function $h_0$, the
tightness of the process $\widehat{\bb Y}^{N}_t$ follows from the one of
the processes $\bb Y^{j,N}_t$, $j=1$, $2$, where $\bb Y^{j,N}_t =
\int_0^t Y^{j,N}_s\, ds $,   
\begin{equation*}
\begin{gathered}
Y^{1,N}_t \,:=\, \int_0^t \frac{K} {\sqrt{N^d  \sqrt{K}}}
\sum_{x\in \bb T^d_N} \widehat F (s,Ax) \,
(\Xi_{\bs u^N(\cdot), x} h_0) (\eta(s)) \; ds  \;,
\\
Y^{2,N}_t \,:=\, \int_0^t \frac{K} {\sqrt{N^d  \sqrt{K}}}
\sum_{x\in \bb T^d_N} \widehat F (s,Ax) \, \mf R_N(\eta(s)) \; ds  \;.
\end{gathered}
\end{equation*}
In Subsection 8.$j$, we prove the tightness of the process
$\bb Y^{3-j,N}_t$, $1\le j\le 2$. This completes the proof of the
theorem.
\end{proof}

\begin{lemma}
\label{2-l16b}
Assume that $d=1$ or $2$. Fix $T>0$, and a function $F\in C^\infty_c
(\bb R^{d-1} \times \bb T)$.  Then, the sequence of probability
measures $\bb Q^M_{N}$ on $C([0,T], \bb R)$ induced by the process
$\int_{[0,t]} M^N_r(F(\cdot)) \, dr$ and the product initial measure
$\nu^N$ is tight. 
\end{lemma}

\begin{proof}
By Schwarz inequality,
\begin{equation*}
\bb E_{\nu_N} \Big[\, \sup_{|t-s|\le \delta}
\Big(\int_s^t M^N_{r} (F(\cdot))\, dr \Big)^2 \,\Big]
\,\le\,
\delta\, \bb E_{\nu_N} \Big[\, 
\int_0^T [\, M^N_{r} (F(\cdot))\,]^2 \, dr \,\Big]\;.
\end{equation*}
By the formula for the quadratic variation of a Dynkin martingale,
\cite[Lemma A1.5.1]{kl}, the previous expectation is equal to
\begin{equation*}
\bb E_{\nu^N} \Big[ \,
\int_0^T dr\, \int_0^r  \Gamma^N_2 \big(T^K_{r-q} F) \, dq \, \Big] \;,
\end{equation*}
where $\Gamma^N_2(\cdot)$ is defined in \eqref{47-b}. By \eqref{2-70},
this expression is less than or equal to 
\begin{equation*}
C_0 \int_0^T dr \int_0^r \, \Big\{ \, \Vert T^K_{q} F  \Vert^2_2 \,+\,
\Vert \partial_{\vartheta} T^K_{q} F \Vert^2_2 \,+\,
(1/K)\, \Vert \nabla_\theta T^K_{q}  F \Vert^2_2 \,\Big\}\, dq
\end{equation*}
for some finite constant $C_0$. By Lemma B.7 and the last paragraph of
section B.2 in \cite{F1}, the previous expression is bounded by a
constant independent of $N$. This completes the proof of the lemma
since $\nabla_\theta T^K_{q}  F = T^K_{q} \nabla_\theta  F$. 
\end{proof}

\subsection{Tightness of the process $\bb Y^{2,N}$}

In this subsection, we show that the process $Y^{2,N}$ is tight, which
implies that $\bb Y^{2,N}$ is tight.  This is the simplest term as,
after a summation by parts, the test functions are of order $1/N$.  By
\eqref{2-19}, $Y^{2,N}$ is composed of two terms $Y^{2,a,N}$ and
$Y^{2,b,N}$. We examine each one separately. Consider
\begin{equation*}
Y^{2,a,N}_t \,:=\, \int_0^t \frac{K} {\sqrt{N^d  \sqrt{K}}}
\sum_{x\in \bb T^d_N} \widehat F (Ax) \,  (\Pi^1_x h_0)(\eta(s)) \; ds\;.
\end{equation*}
Fix $y\in \bb Z^d$ in the support of the cylinder function $h_0$, that
is some $y$ such that
$E_{\nu_{\rho}} [\, \omega_{y} \, h_0 \,] \neq 0$ for some
$\rho\in (0,1)$.  It is enough to prove the tightness of the process
$Y^{2,y,N}_t$ obtained from $Y^{2,a,N}_t $ by replacing in the
previous formula $\Pi^1_x h_0$ by
$G_y(x) \, [\, \omega_{x+y} - \omega_x\,]$, where
\begin{equation*}
G_y(x)   \,=\, E_{\nu^N} \Big[\, \frac{\omega_{x+y}}{\chi (x+y)} \,
\tau_x h_0 \,\Big] \;.
\end{equation*}

By a summation by parts,
\begin{equation*}
Y^{2,y,N}_t \,:=\, \int_0^t \frac{K} {\sqrt{N^d  \sqrt{K}}}
\sum_{x\in \bb T^d_N}  H_y(x) \, \omega_x(s)\, ds \;,
\end{equation*}
where
$H_y(x) = \widehat F (A(x-y)) \, G_y(x-y) - \widehat F (A(x)) \,
G_y(x) $.  With this notation, by Schwarz inequality,
\begin{equation*}
\bb E_{\nu^N} \Big[ \, \sup_{t\le T} \big|\, Y^{2,y,N}_t
\, \big|\, \Big]^2  \;\le\; T\, 
\bb E_{\nu^N} \Big[ \, 
\int_0^T \Big(\, \frac{K} {\sqrt{N^d  \sqrt{K}}}
\sum_{x\in \bb T^d_N} H_y(x)\, \omega_x(s)\;  \Big)^2   \; ds
\, \Big] \, 
\end{equation*}
By the entropy inequality this expression is bounded by 
\begin{equation*}
\frac{T}{A}\,  \int_0^T  H_N(f^N_s) \, ds \;+\;
\frac{T^2}{A}\, \log 
\bb E_{\nu^N} \Big[ \, \exp \Big\{\, A\, \Big(\, \frac{K} {\sqrt{N^d  \sqrt{K}}}
\sum_{x\in \bb T^d_N} H_y(x)\, \omega_x \;  \Big)^2  \Big\}  \, \Big]
\end{equation*}
for all $A>0$.

A simple computation yields that
$\sup_{x\in \bb T^d_N} |H_y(x)| \le C_0 \sqrt{K}/N$ for some finite
constant $C_0$ depending on $y$, $\Vert \nabla F\Vert_\infty$ and the
bounds \eqref{2-38b}, \eqref{2-84} on the stationary solution of the
reaction-diffusion equation. Set $A= a N^2/K^{5/2}$ for some small
constant $a$. By the concentration inequality (cf. equation (4.8) and
Lemma 4.4 in \cite{jl}), for $a$ small enough the previous
expression is bounded by
\begin{equation*}
\frac{K^{5/2}}{aN^2} \, \Big(\, T \,  \int_0^T  H_N(f^N_s) \, ds \;+\;
C_0\, T^2 \,\Big) \,\le\,
C_0\, T^2\, \frac{K^{5/2}}{aN^2} \, \big(\, (N/\ell_N)^d\, \kappa_N(T)
\;+\; 1 \,\big)  
\end{equation*}
for some finite constant $C_0$. In the last step, we used the bound
\eqref{2-50b} on the entropy. By \eqref{2-72}, this expression
vanishes as $N\to\infty$. This proves that the process $Y^{2,a,N}_t$
converges to $0$ in the uniform topology.

The process $Y^{2,b,N}_t$, obtained from $Y^{2,a,N}_t$ by replacing
$\Pi^1_x h_0$ by $R^{(1)}_N$, is handled similarly.

\subsection{Tightness of the process $\bb Y^{1,N}$}

We follow the same route pursued in \cite[Section 7.2]{jl} based on
Kolmogorov-\v{C}entsov theorem, \cite[Section 2.4, Problem 4.11]{ks91}
or \cite[Proposition 7]{RacSuq01}

Fix $T>0$ and denote by $C^\alpha([0,T],\bb R)$, $0<\alpha<1$,
the H\"older space of continuous functions endowed with the norm
\begin{equation*}
\Vert Y \Vert_{C^\alpha} \;=\;
\sup_{0\le t \le T} |\, Y_t \,| \;+\; 
\sup_{0\le s <t \le T}
\frac{|\, Y_t - Y_s\,|}{|t-s|^\alpha}\;\cdot
\end{equation*}
Note that this topology is stronger than the uniform topology of
$C([0,T], \bb R)$.

\begin{theorem}
\label{l46}
Fix $T>0$ and $0<\alpha<1$. A sequence of probability measures
$\bb M_n$ on $C^\alpha([0,T], \bb R)$ is tight if
\begin{enumerate}
\item[(i)] There exist a constant $a>0$ such that
$\sup_{n\ge 1} \bb M_{n} \big[ \, |\, Y_0 \,|^a \,\big] \;<\, \infty$;
\item[(ii)] There exist a finite constant $C_0$ and positive constants
$b>0$, $c>0$ such that $c/b > \alpha$ and 
$\sup_{n\ge 1} \bb M_n \big[ \, | \, Y_t \,-\,
Y_s \, |^b\,\big] \;\le\, C_0 \, |t-s|^{1+c}$ for
all $0\le s, t\le T$.
\end{enumerate}
\end{theorem}

We used here the same notation $\bb M_n$ to represent a probability on
$C^\alpha([0,T], \bb R)$ and the expectation with respect to
this probability.

By \eqref{2-17}, there exists finite collection $\mc E$ of finite
subsets of $\bb Z^d$ with at least two elements such that
\begin{gather*}
(\Xi_{\bs u^N(\cdot)} h_0  )(\eta)
\;=\; \sum_{D \in \mc E} \widetilde c_D(\cdot) \,\omega_D \;,
\end{gather*}
where
$\color{blue} \widetilde c_D(\cdot) := \widetilde c_D (\bs
u^N(\cdot))$ depends on the density profile $\bs u^N(\cdot)$ smoothly
through a finite number of coordinates, as the rate $c_0$ is a
cylinder function of $\eta$.  In particular, to prove the tightness of
the process $\bb Y^{1,N}_t$, it is enough to prove this property for
the processes $\bb Y^{1,D,N}_t$ defined by
$\color{blue} \bb Y^{1,D,N}_t = \int_0^t Y^{1,D,N}_s\, ds$,
\begin{equation*}
{\color{blue}  Y^{1,D,N}_t}  \;:=\; \int_0^t  \ms Y^{1,D,N}_s \; ds \;, \quad
{\color{blue}  \ms Y^{1,D,N}_s } \;:=\;  \frac{K} {\sqrt{N^d  \sqrt{K}}}
\sum_{x\in \bb T^d_N} \widehat F (s, Ax) \, \widetilde c_D(Ax) \,
\omega_{x+D}(s) \;, 
\end{equation*}
where $D$ a finite subset of $\bb Z^d$ with at least two elements.
The proof of the next result is similar to the one of \cite[Lemma
7.4]{jl}. We do not need to assume in this lemma that the set $D$ has
at least two elements.  The proof relies on a H\"older
inequality (to replace the power $\alpha$ by a square), followed by the
entropy inequality and a concentration inequality to estimate the
expectation of the exponential.

\begin{lemma}
\label{2-l01}
Fix a finite subset $D$ of $\bb Z^d$. Then, there exists a finite constant
$C_0$, depending only on $D$ and the bounds \eqref{2-38b},
\eqref{2-84} on the density profile, such that
\begin{equation*}
\bb E_{\nu^N} \Big[ \, \Big|\, \sum_{x\in \bb T^d_N} G(x) \,
\omega_{x+D}(t) \,\Big|^\alpha\,\Big] \;\le\; C_0 \,
N^{\alpha d/2} \, \Vert G\Vert^\alpha _\infty \, \big( \, 1 \,+\,
H_N(f^N_t)\,\big)^{\alpha /2} 
\end{equation*}
for all $1\le \alpha\le 2$, $t>0$, $G:\bb T^d_N \to\bb R$ and
$N \ge 1$. 
\end{lemma}

The next result is a simple consequence of the previous lemma and
H\"older's inequality.

\begin{corollary}
\label{2-l13}
Fix a finite subset $D$ of $\bb Z^d$. Then, there exists a finite
constant $C_0$, depending only on $D$ and the bounds \eqref{2-38b},
\eqref{2-84} on the density profile, such that
\begin{align*}
& \bb E_{\nu^N} \Big[ \, \Big|\, \int_s^t \sum_{x\in \bb T^d_N} G_x(r)\,
\omega_{x+D}(r) \; dr \,\Big|^\alpha\,\Big] \\
&\qquad \;\le\; C_0 \,
N^{\alpha d/2} \, |t-s|^\alpha \, \sup_{s\le r \le t} \Vert G(r) \Vert^\alpha _\infty \,
\sup_{s\le r\le t} \big( \, 1 \,+\,
H_N(f^N_r)\,\big)^{\alpha /2} 
\end{align*}
for all $1\le \alpha\le 2$, $0<s<t$,
$G\colon [0,T]\times \bb T^d_N \to\bb R$ and $N\ge 1$.
\end{corollary}

The proof of the tightness of the process $\bb Y^{1,D,N}_t$ relies on
the following result.

\begin{lemma}
\label{2-l22}
Fix a finite subset $D$ of $\bb Z^d$ with at least two elements. Then,
there exist a finite constant $\mf c(\bs u^N)$, uniformly bounded in
$N$ and depending only on $\bs u^N$ and $D$, such that
\begin{equation*}
\bb E_{\nu_N} \Big[ \,
\Big | \, \int_0^t \ms Y^{1,D,N}_{s}   \; ds \, \Big |^a \,
\Big]
\,\le\, \mf c(\bs u^N) \, T^a\,
\sup_{0\le r\le T} (1+\Vert F(r) \Vert^2_\infty)^a\,
\Big(\frac{1}{N^{d/2}}\Big)^{\theta/c} \, \ms K_N^{1/c}\;,
\end{equation*}
for all $N\ge 1$, $0\le t\le T$, $a>0$, $1\leq b\leq 2$,
and $0<\theta<1$ such that
\begin{equation}
\label{2-51}
\theta \,+\, (1-\theta) b \,>\, a\;.
\end{equation}
In this formula, $c = [\, \theta \,+\, (1-\theta) b\,]/ a >1$, and
\begin{equation*}
\ms K_N \,:=\, K^{[7 \theta + 3b(1-\theta)]/4}\, [\kappa_N(T)\,
(N/\ell_N)^d]^{\theta + (1/2) b (1-\theta)}\;.
\end{equation*}
\end{lemma}

\begin{proof}
By Lemma \ref{l-jl} below, the expectation appearing in the statement
of the lemma is bounded by
\begin{equation*}
\bb E_{\nu_N} \Big[ \,
\Big | \, \int_0^t \ms Y^{1,D,N}_{s}   \; ds \, \Big | \,
\Big]^{\theta/c}
\,
\bb E_{\nu_N} \Big[ \,
\Big | \, \int_0^t \ms Y^{1,D,N}_{s}   \; ds \, \Big |^b \, \Big]
^{(1-\theta)/c} \;.
\end{equation*}
By Proposition \ref{2-l14} and \eqref{2-50b}, the first expectation
(without the exponent $\theta/c$) is bounded by
\begin{equation*}
\mf c(\bs u^N) \, T\,  \frac{K^{3/4}}{N^{d/2}} \, \sup_{0\le r\le T}
(1+\Vert F(r) \Vert^2_\infty) \, K \, \kappa_N(T) \, (N/\ell_N)^d \;.
\end{equation*}

Corollary \ref{2-l13} provides an estimate for the second
expectation. Multiplying them together in the combination given
completes the proof of the lemma.
\end{proof}

\begin{corollary}
\label{2-l21}
Suppose that the dimension is $1$ or $2$.  Fix a function $F$ in
$\ms D (\bb R \times \bb T^{d-1})$, and finite subset $D$ of $\bb Z^d$
with at least two elements. Then, the sequence of probability measures
on $C([0,T], \bb R)$ induced by the process $\bb Y^{1,D,N}_t$ and the
measure $\nu^N$ taken as the initial state is tight.
\end{corollary}

\begin{proof}
The proof relies on Theorem \ref{l46}.  The first condition of this
theorem is satisfied because $\bb Y^{1,D,N}_0=0$.

We turn to condition (ii). Fix $0<\delta<2/3$ according to condition
\eqref{2-48}.  In the previous lemma, set $\theta = 1 -\delta$,
$b =1+2\delta$. With this choice,
$\theta + (1-\theta)b = 1+ 2\delta^2$. Fix $1<a< 1+ 2\delta^2$ so that
condition \eqref{2-51} is fulfilled. Moreover,
$\theta + (1/2)b(1-\theta) = 1 - (1/2) \delta +\delta^2$,
$[7 \theta + 3b(1-\theta)] = 7- 4\delta + 6\delta^2 \le 7$ because
$\delta <2/3$.

Fix $0\le s<t\le T$. By H\"older's inequality,
\begin{equation*}
\begin{aligned}
\bb E_{\nu_N} \big[ \,
\big | \bb Y^{1,D,N}_t - \bb Y^ {1,D,N}_s \, \big |^a
\, \big]\;  & = \
\bb E_{\nu_N} \Big[ \,
\Big | \int_s^t dr \, \int_0^r \ms Y^{1,D,N}_{r'}   \; dr' \, \Big |^a
\, \Big]\;
\\
& \le\;(t-s)^{a-1} \int_s^t \, dr\, 
\bb E_{\nu_N} \Big[ \,
\Big | \, \int_0^r \ms Y^{1,D,N}_{r'}   \; dr' \, \Big |^a \, \Big] \;.
\end{aligned}
\end{equation*}
By Lemma \ref{2-l22}, the previous expectation is bounded above by
\begin{equation*}
\mf c(\bs u^N, T, D, \delta)  \,
\sup_{0\le r\le T} (1+\Vert T^K_r F \Vert^2_\infty)^a\,
\Big( \frac{1}{N^{(d/2)(1-\delta)}} \, K^{7/4}\, \big[\, \kappa_N(T) \,
(N/\ell_N)^d \,\big]^{1-(\delta/2)}\Big)^{1/c} 
\end{equation*}
for some constant $\mf c(\bs u^N, T, D, \delta)$ uniformly bounded in
$N$. Here $c = [\, \theta \,+\, (1-\theta) b\,]/ a >1$. By Lemma
\ref{maximum lemma},
$(1+\Vert T^K_r F \Vert^2_\infty)^a \le C(F) \exp\{2a (\mf c_5
K+1)T\}$. By \eqref{2-48}, the previous expression is bounded
uniformly in $N$, which completes the proof of condition (ii) in
Theorem \ref{l46} and the proof of the corollary.
\end{proof}

We conclude this section with a bound used in the proof of Lemma
\ref{2-l22}, and the proof of the tightness of the martingale
$M^N_t(F)$ introduced in \eqref{2-60b}

\begin{lemma}
\label{2-l16}
Assume that $d=1$ or $2$. Fix $T>0$, and a function
$F\in C^\infty_c (\bb R \times \bb T^{d-1})$.  Then, the sequence of
probability measures $\bb Q^{\bb R, M}_{N}$ on $D([0,T], \bb R)$ induced by the
process $M^N_t(F)$ and the product initial measure $\nu^N$ is
tight. 
\end{lemma}

\begin{proof}
In the proof of Lemma \ref{l38} we have shown that the quadratic
variation of the martingale $M^N_t(F)$ converges to the quadratic
variation of a Brownian motion changed in time. Since the jumps of
$M^N_t(F)$ are bounded by $\Vert F\Vert_\infty /N^{d/2}$, by
\cite[Theorem VIII.3.11]{js}, $M^N_t(F)$ converges to a Brownian
motion changed in time. In particular, by Prohorov's theorem, the
sequence $\bb Q^{\bb R, M}_{N}$ is tight.
\end{proof}

The next result follows from H\"older inequality.. 

\begin{lemma}
\label{l-jl}
Let $X$ be a non-negative random variable. Then,
\begin{equation*}
E[X^a] \,\le\,  E[X]^{\theta/c}\, E[X^b]^{(1-\theta)/c}
\end{equation*}
for all $a>0$, $b>0$ and $0<\theta<1$ such that
\begin{equation*}
\theta \,+\, (1-\theta) b \,>\, a\;.
\end{equation*}
In this formula $c = [\, \theta \,+\, (1-\theta) b\,]/ a >1$.
\end{lemma}

\appendix

\section{Cylinder functions expansion}
\label{ap1}

Fix a density profile $\rho\colon \bb T^d_N \to (0,1)$,
$x\in \bb T^d_N$, and let $\Xi_{\rho(\cdot) , x}$ be the operator
acting on cylinder function as
\begin{equation}
\label{2-35}
(\Xi_{\rho(\cdot) , x}  f )(\eta) \;:=\; \tau_x f(\eta) \,-\,
E_{\nu_{\rho(\cdot)}}
[\, \tau_x f\,] \,-\, \sum_{y\in \bb T^d_N}
E_{\nu_{\rho(\cdot)}} \Big[\, \frac{\omega_{x+y}}{\chi (\rho(x+y))} \,
\tau_x f \,\Big] \; \omega_{x+y} \;,
\end{equation}
where $\omega_z = \eta_z - \rho(z)$ and
$\chi(\frak{r}(z))=\frak{r}(z)(1-\frak{r}(z))$.

As $f$ is a cylinder function, only a finite number of terms in the
sum does not vanish. When $x=0$, we represent $\Xi_{\rho(\cdot) , x} $
by $\Xi_{\rho(\cdot)} $. 

As $f$ is a cylinder function, there exists a finite set $A$ and
constants $c(B)$, $B\subset A$, such that
\begin{equation}
\label{2-42}
f(\eta) \;=\; \sum_{B\subset A} c(B)\, \eta_B\;, \quad
\text{where} \;\; {\color{blue}  \eta_B} \,:=\, \prod_{x\in B} \eta_x\;.
\end{equation}
By \eqref{2-35}, and noting that
$E_{\nu_{\rho(\cdot)}} [\omega_x \eta_x] = \chi(\rho(x))$,
\begin{equation*}
(\Xi_{\rho(\cdot)} f )(\eta) \;:=\; f(\eta)
\;-\; \sum_{B\subset A} c(B)\, \rho_B
\;-\; \sum_{x\in  A} \omega_x \sum_{B: x\in B \subset A} c(B)\,
\rho_{B\setminus \{x\}}\;, \quad
\end{equation*}
where ${\color{blue} \rho_C} \,:=\, \prod_{x\in C} \rho(x)$.  Adopting
the same notation for $\omega_C$,
${\color{blue} \omega_C} \,:=\, \prod_{x\in C} \omega (x)$.
\begin{equation*}
\begin{aligned}
f(\eta) \;& =\; \sum_{B\subset A} c(B)\, \eta_B
\;=\; \sum_{B\subset A} c(B)\, \prod_{y\in B} [\omega_y + \rho_y]
\\
\;& =\; \sum_{B\subset A} c(B)\, \sum_{C\subset B}
\omega_C \rho_{B\setminus C}
\;=\; \sum_{C\subset A} \omega_C \sum_{C\subset B\subset A} c(B)\, 
\rho_{B\setminus C} \;.
\end{aligned}
\end{equation*}
The term $C=\varnothing$ corresponds to the first term in the formula
for $\Xi_{\rho(\cdot)} f$, and the terms with $|C|=1$ to the
second. Therefore, for every $x\in\bb T^d_N$,
\begin{equation}
\label{2-17}
(\Xi_{\rho(\cdot), x} f )(\eta) \;:=\;
\sum_{C\subset A : |C|\ge 2} \omega_{C+x} \sum_{C\subset B\subset A} c(B)\, 
\rho_{B\setminus C} \;.
\end{equation}

It is sometimes more convenient to have $\omega_x$ instead of
$\omega_{x+y}$ on the right-hand side of equation \eqref{2-35}. Let
$\Xi^{\rm c}_{\rho(\cdot), x} f \colon \Omega_N\to \bb R$ (c for
centered) be given by
\begin{equation}
\label{2-35b}
{\color{blue} (\Xi^{\rm c}_{\rho(\cdot),  x}  f) (\eta)} \;:=\;
\tau_x f(\eta) \,-\,  E_{\nu_{\rho(\cdot)}}
[\, \tau_x f\,] \,-\, \Xi^{\rm c}_{\rho (\cdot), f}(x) \, \omega_x \;.
\end{equation}
where
\begin{equation*}
\Xi^{\rm c}_{\rho (\cdot), f}(x)  \;:=\; E_{\nu_{\rho(\cdot)}}
\Big [\, \tau_x f \, \sum_{y\in \bb Z^d}
\frac{\eta_{x+y} - \rho(x+y)}{\chi(\rho(x+y))} \, \Big]
\;.
\end{equation*}
Recall $h_0$ in \eqref{2-63}. The next lemma asserts that
$\Xi^{\rm c}_{\bs u^N (\cdot), h_0}(x)$ is close to
$-\, V''(\bs u^N(x))$.

\begin{lemma}
\label{lem:3.1-F}
There exists a finite constant $\mf c(\bs u^N)$, depending only on the
cylinder function $h_0$ and on $\bs u^N$, and uniformly bounded in $N$,
such that
\begin{equation*}
\sup_{x\in \bb T^d_N} \Big|\, 
\Xi^{\rm c}_{\bs u^N (\cdot), h_0}(x) \,+\,
V''(\bs u^N(x)) \, \Big| \,\le\,  \mf c(\bs u^N) \,
\frac{\sqrt{K}}{N} \;.
\end{equation*}
\end{lemma}

\begin{proof}
Denote by $\bb B_m$ the $d$-dimensional cube of length $2m+1$ centered
at the origin: $\color{blue} \bb B_m = \{-m, \dots, m\}^d$.  Since
$h_0(\eta)$ depends only on $\{\eta_y; y \in \bb B_p\}$ for some
$p<\infty$, 
\begin{equation*}
\Xi^{\rm c}_{\bs u^N (\cdot), h_0}(x)  =\; \sum_{y\in \bb B_p}E_{\nu_{\bs u^N(\cdot)}}
\Big [\, \tau_x h_0 \cdot \frac{\eta_{x+y} - \bs u^N(x+y)}{\chi(\bs
u^N(x+y))} \, \Big] \;.
\end{equation*}

By \eqref{2-38b}, $\rho_- \le \bs u^N(x)\le \rho_+$. By definition
of $\bs u^N$ (cf. \eqref{2-53a}),
$|\nabla_N \bs u^N(x)| \le \Vert \partial_\theta \rho^K \Vert_\infty\,
\sqrt{K}$. Thus, there exists a finite constant $C_0$, depending only
on $h_0$ such that
\begin{align*}
\Big| \frac{\eta_{x+y} - \bs u^N(x+y)}{\chi(\bs u^N(x+y))}
- \frac{\eta_{x+y} - \bs u^N(x)}{\chi(\bs u^N(x))} \Big| 
\,\le\,  C_0  \, \frac{\sqrt{K}}N \, \Vert \partial_\theta \rho^K \Vert_\infty
\end{align*}
for all $x\in \bb T^d_N$, $y\in \bb B_p$.  Therefore, since $h_0$ is
bounded, 
\begin{equation}
\label{2-47}
\Big|\, \Xi^{\rm c}_{\bs u^N (\cdot), h_0}(x) \,-\,
\sum_{y\in \bb B_p}E_{\nu_{\bs u^N(\cdot)}}
\Big [\, \tau_x h_0 \cdot \frac{\eta_{x+y} - \bs u^N(x)}{\chi(\bs
u^N(x))} \, \Big] 
\,\Big|\,\le\, C_0  \, \frac{\sqrt{K}}N \,
\Vert \partial_\theta \rho^K \Vert_\infty
\end{equation}
for some finite constant $C_0$, depending only on the cylinder
function $h_0$.

On the other hand, recalling \eqref{2-54},
\begin{equation*}
-\, V'(\rho) \,=\, E_{\nu_\rho}[h_0]
\,=\, \sum_{\eta\in \{0,1\}^{\bb B_p}} h_0(\eta) \, 
\nu^{\bb B_p}_\rho (\eta) \;.
\end{equation*}
In this formula, for a subset $\bb A$ of $\bb Z^d$ and a density
profile $\varrho\colon\bb A \to (0,1)$, $\nu^{\bb A}_{\varrho
(\cdot)}$ stands for the Bernoulli product measure
with density $\varrho(\cdot)$ on $\{0,1\}^{\bb A}$:
\begin{equation*}
{\color{blue} \nu^{\bb A}_{\varrho (\cdot)} (\eta)} \;:=\;
\prod_{z\in \bb A} \varrho (z)^{\eta_z} \, (1-\varrho(z))^{1-\eta_z},
\quad \eta\in \{0,1\}^{\bb A} \;.
\end{equation*}
When the profile $\varrho (\cdot)$ is constant equal to $\rho$, we
abbreviate $\nu^{\bb A}_{\varrho (\cdot)}$ by $\nu^{\bb
A}_{\rho}$. With this notation, taking the derivative with respect to
$\rho$, yields that
\begin{align*}
-V''(\rho) & = E_{\nu^{\bb B_p}_\rho}
\Big[h_0\, \Big( \tfrac1\rho \sum_{y\in \bb B_p} \eta_y
- \tfrac1{1-\rho} \sum_{y\in \bb B_p} (1-\eta_y) \Big) \Big] \\
& = E_{\nu^{\bb B_p}_\rho} \Big[h_0 
  \sum_{y\in \bb B_p} \frac{\eta_y - \rho}{\chi(\rho)} \Big] \;=\;
\sum_{y\in \bb B_p} E_{\nu_\rho} \Big[\t_x h_0 
  \cdot \frac{\eta_{x+y} - \rho}{\chi(\rho)} \Big] \;.  
\end{align*}

By the previous estimate and \eqref{2-47}, to complete the proof of
the lemma it remains to estimate the difference
\begin{equation*}
E_{\nu_{\bs u^N(\cdot)}}
\Big [\, \tau_x h_0 \cdot \frac{\eta_{x+y} - \bs u^N(x)}{\chi(\bs
u^N(x))} \, \Big] 
- E_{\nu_{\bs u^N(x)}}
\Big [\, \tau_x h_0 \cdot \frac{\eta_{x+y} - \bs u^N(x)}{\chi(\bs
u^N(x))} \, \Big] .
\end{equation*}
for each $y$ in $\bb B_p$.  As $1/\chi(\bs u^N(x))$ is bounded, we
have to bound the differences
\begin{equation*}
E_{\nu_{\bs u^N(\cdot)}}
\big [\, \tau_x h_0 \cdot \eta_{x+y} \big]
- E_{\nu_{\bs u^N(x)}}
\big [\, \tau_x h_0 \cdot \eta_{x+y} \big],
\end{equation*}
and
\begin{equation*}
E_{\nu_{\bs u^N(\cdot)}}
\big [\, \tau_x h_0  \big]
- E_{\nu_{\bs u^N(x)}}
\big [\, \tau_x h_0  \big].
\end{equation*}
Since the estimates of these two terms are similar, we only give the
details for the latter. The difference is expressed as
\begin{equation*}
\sum_{\eta\in \bb B_p}  h_0 (\eta) \,\{\,
\nu^{\bb B_p}_{\bs u^N(x+\cdot)}(\eta) - \nu^{\bb B_p}_{\bs u^N(x)}(\eta)
\,\big\}\;.
\end{equation*}
To estimate this sum, we change the density
$\{\bs u^N(x+y) : y\in \bb B_p\}$ in $\{\bs u^N(x) : y\in \bb B_p\}$
modifying one by one the site densities.  Since $\bb B_p$ has a finite
number of elements, it is sufficient to compare the expectation of
$h_0$ with respect to two product measures whose densities differ at
only one site.

Fix $z\in \bb B_p$ and two density profiles $\bs u_1(\cdot)$ and
$\bs u_2(\cdot)$ such that $\bs u_1(y)= \bs u_2(y)$ for all
$y\in \bb B_p$, $y\neq z$.  Assume that
$0<\mf a \le u_1(z), u_2(z) \le 1-\mf a$, so that
\begin{align*}
\big| \, \nu_{\bs u_1^N(\cdot)} (\eta) - \nu_{\bs u_2^N(\cdot)} (\eta)
\, \big|
& \,=\, \nu_{\bs u_2^N(\cdot)} (\eta)  \, 
\Big| \, \frac{\bs u_1(z)^{\eta_{z}} (1- \bs u_1(z))^{1-\eta_{z}}}
{\bs u_2(z)^{\eta_{z}} (1- \bs u_2(z))^{1-\eta_{z}}} -1 \, \Big|  \\
& \le \mf c(\bs u^N)\, \nu_{\bs u_2^N(\cdot)} (\eta)  \,
| \bs u_1(z)- \bs u_2(z)| 
\end{align*}
for some finite constant $\mf c(\bs u^N)$ which depends only on
$\bs u^N$.  This completes the proof of the lemma since, recall,
$|\nabla_N \bs u^N(z)| \le C_0 \sqrt{K}$.
\end{proof}

In view of the previous lemma, and by \eqref{2-35}, \eqref{2-35b},
\begin{equation}
\label{2-19}
\tau_x f(\eta) \,-\,  E_{\nu_{\rho(\cdot)}} [\, \tau_x f\,] \,=\, -\,
V''(\bs u^N(x))\, \omega_x \;+\;  (\Xi_{\rho(\cdot) , x}  f )(\eta) 
\;+\; \mf R_N(\eta)   \;,
\end{equation}
where
\begin{equation*}
\begin{gathered}
\mf R_N(\eta) \,=\, (\Pi^1_x f) (\eta)  \;+\; R^{(1)}_N(x)\, \omega_x\;,
\\
(\Pi^1_x f)(\eta) \,=\,  \sum_{y\in \bb T^d_N}
E_{\nu_{\rho(\cdot)}} \Big[\, \frac{\omega_{x+y}}{\chi (x+y)} \,
\tau_x f \,\Big] \; \{ \, \omega_{x+y} - \omega_x\,\} \;,
\\
R^{(1)}_N(x) \,=\, \Xi^{\rm c}_{\rho (\cdot), f}(x) \, + V''(\bs
u^N(x)) \;.
\end{gathered}
\end{equation*}

\section{The reaction--diffusion equation}
\label{sec-c}

In this section, we present some results on the solutions
$u^{(\epsilon)}$, $\rho^K$ of equations \eqref{2-56}, \eqref{2-23b},
respectively.  Some of the results are taken from \cite{F1} which
improves and elaborates results in \cite{cp} in a stable situation on
$\bb T$. Mind that we do not assume in the analysis below that the
potential $V$ is symmetric around $\rho^*$.

Recall that $\rho_- = 1-\rho_+$, $\rho_+$ and $\rho^*$ represent the
critical points of the potential $V$.  By definition, $u^{(\epsilon)}$
has two layers and $u^{(\epsilon)} (0) = \rho^*$. Denote by
$h_2\in (0,1)$ the second point at which $u^{(\epsilon)} (\cdot)$ is
equal to $\rho^*$: $\color{blue} u^{(\epsilon)} (h_2) = \rho^*$. Since, by
definition, $(\partial_\theta u^{(\epsilon)}) (0) <0$,
$u^{(\epsilon)} (\theta) > \rho^*$ for $\theta \in (h_2, 1)$,
$u^{(\epsilon)} (\theta) < \rho^*$ for $\theta \in (0,h_2)$.  Denote by
$m_1$, $m_2$ the points at which $u^{(\epsilon)}$ attains its minimum
and maximum values, respectively. By symmetry,
$\color{blue} m_1 = h_2/2$ and $\color{blue} m_2 = (1+h_2)/2$.

We claim that
\begin{equation}
\label{2-38b}
\rho_- \,\le\, u^{(\epsilon)}  (\theta)\,\le\, \rho_+\;, \quad
\theta \in \bb T \;.
\end{equation}
By construction, it is enough to show that
$u^{(\epsilon)} (m_1) \ge \rho_-$ and
$u^{(\epsilon)} (m_2) \le \rho_+$. We prove the first bound, as the
proof of the second is similar. By construction,
$u^{(\epsilon)} (m_1) < \rho^*$. On the other hand, as $u^{(\epsilon)}$
is a solution of \eqref{2-56} and $m_1$ is a local minimum,
$V'(u^{(\epsilon)} (m_1)) = \Delta u^{(\epsilon)} (m_1) \ge 0$. As $V$
is a double well potential, $V'$ is non-negative on
$[\rho_-, \rho^*] \cup [\rho_+, 1]$. Thus, $u^{(\epsilon)} (m_1)$ belongs
to this set. Since we have already shown that
$u^{(\epsilon)} (m_1) < \rho^*$,
$u^{(\epsilon)} (m_1) \in [\rho_-, \rho^*)$, as claimed.

Recall from \eqref{2-79} that we denote by $\phi$ the decreasing
standing wave solution of the equation $\Delta u - V'(u)=0$.  By
\cite[Lemma 2.1]{AHM}, \cite[Lemma 5]{kfhps}, there exists $C>0$ and
$\lambda>0$ so that
\begin{equation}
\label{CP properties}
\begin{gathered}
\rho_+ \,-\, C\, e^{-\lambda |\vartheta|}
\,<\,  \phi(\vartheta) 
\,<\, \rho_+  \,, \quad \vartheta \,<\, 0\;, \\
\rho_-  \,<\,  \phi(\vartheta) 
\,<\, \rho_-  \,+\, C\, e^{-\lambda |\vartheta|} \,,
\quad \vartheta \,>\, 0\;, \\
|\, \partial^j_\vartheta \phi\, (\vartheta)\,|
\,\leq\, C\, e^{-\lambda |\vartheta|} \,\quad j=1\,, 2\,,\;\;
\vartheta\in {\mathbb R}\;. 
\end{gathered}
\end{equation}
Define $\bar{u}_K: \T \rightarrow [0,1]$ by
\begin{equation*}
\bar{u}_K (\theta) \,=\,\
\left\{
\begin{aligned}
& \phi(\theta\sqrt{K}) \,, \quad  \theta\in [0, m_1]\,, \\
& \phi((h_2 - \theta)\sqrt{K}) \,, \quad  \theta\in [m_1, m_2] \,,\\
& \phi((\theta - 1)\sqrt{K}) \,, \quad  \theta\in [m_2, 1]\, .
\end{aligned}
\right.
\end{equation*}
Mind that $\bar{u}_K$ is not differentiable at $m_1$, $m_2$.

The next result is Lemmata B.5 and B.6 in \cite{F1}. As $\bar{u}_K$ is
not differentiable at $m_1$, $m_2$, in the next lemma, the derivatives
at these points should be understood as left and right derivatives.

\begin{lemma}
\label{CP_lemma}
There exists a finite constant $C_0$ such that
\begin{align*}
\sup_{\theta\in \T}|\,u^{(\epsilon)}(\theta) - \bar{u}_K(\theta)\,|
\, \leq\, C_0 K^{-1/4}\;, \quad
\sup_{\theta\in \T} |\, (\partial_\theta u^{(\epsilon)}) (\theta) -
(\partial_\theta \bar{u}_K) (\theta)\, | \, \leq\,  C_0 K^{1/4} 
\end{align*}
For all $K\ge 1$. In this formula, $\epsilon = K^{-1/2}$.
\end{lemma}

It follows from the previous result and the definition of $\rho^K$
that there exists a finite constant $C_0$ such that 
\begin{gather*}
\sup_{\vartheta\in \sqrt{K}\T} |\, \rho^K(\vartheta)
- \bar{u}_K(\vartheta/\sqrt{K})\, | \,\leq\,  C_0\,  K^{-1/4} \;,
\\
\sup_{\vartheta\in \sqrt{K}\T}|\, (\partial_\vartheta \rho^K) (\vartheta) -
(\partial_\vartheta \bar{u}_K) (\vartheta/\sqrt{K})\, |  \,\leq\,
C_0\, K^{-1/4}
\end{gather*}
for all $K\ge 1$. Thus, as a consequence of the definition of
$\bar u_K$ and \eqref{CP properties},
$|\, (\partial_\vartheta \rho^K) (\vartheta)\,| \le C_0$ for some
finite constant $C_0$ and all $\vartheta\in \sqrt{K}\T$, $K\ge 1$.
Since $\partial^2_\vartheta \rho^K = V'(\rho^K)$,
$\partial^3_\vartheta \rho^K = V''(\rho^K)\, \partial_\vartheta
\rho^K$,
$\partial^{4}_\vartheta \rho^K = \partial^2_\vartheta V'(\rho^K) =
V'''(\rho^K)(\partial_\vartheta \rho^K)^2 +
V''(\rho^K)\partial^2_\vartheta \rho^K$, the same bound holds for
$\partial^2_\vartheta \rho^K$, $\partial^3_\vartheta \rho^K$,
$\partial^4_\vartheta \rho^K$:
\begin{equation}
\label{2-84}
\max_{1\le j\le 4}\, \sup_{K\ge 1} \sup_{\vartheta\in \sqrt{K}\T}
|\, (\partial^j_\vartheta \rho^K) (\vartheta)\,| \,<\, \infty\;.
\end{equation}

\subsection*{The parabolic equation}

Denote by $\mf A_K$ the continuum operator defined on smooth functions
$G\colon \sqrt{K} \, \bb T \times \bb T^{d-1} \to \bb R$ by
\begin{equation}
\label{2-71}
{\color{blue} (\mf A_K G) (\vartheta, \theta)} \,:=\,
K\, ( \ms A_K G) (\vartheta, \theta)
\,+\, (\De_{\theta} G) (\vartheta, \theta)\;,  
\end{equation}
where $\Delta_{\theta}$ represents the Laplacian on the variables
$\theta\in \bb T^{d-1}$, and $\ms A_K$ the operator defined by
\begin{equation}
\label{2-49}
{\color{blue} (\ms A_K G) (\vartheta, \theta)} \,:=\,
(\partial_\vartheta^2 G) (\vartheta, \theta) \,-\,
V''(\rho^K(\vartheta))\, G (\vartheta, \theta)\; .
\end{equation}
Denote by $\color{blue} (S_t : t\ge 0)$,
$\color{blue} (T^K_t : t\ge 0)$ the semigroup associated to the
generator $\Delta_\theta$, $\mf A_K$, respectively.

\begin{lemma}
\label{maximum lemma}
Let $\mf c_5 = \sup_{u\in (0,1)}|V''(u)|$. Then, for each
$F\in C^\infty(\sqrt{K}\T \times \bb T^{d-1})$ there exists a finite constant
$C(F)$ such that
\begin{gather*}
\|\partial^i_{\vartheta} T_s^K F\|_\infty \leq C(F)\,
K^i \, e^{(i+1) \, (\mf c_5 K +1)T}\;,
\\
\|\partial^i_{\theta_j} T_s^K F\|_\infty \,\leq\,  C(F)\,
e^{(\mf c_5 K +1)T}\;,
\end{gather*}
for all $0\le s\le T$, $2\le j\le d$, $0\le i\le 3$, $K\ge 1$.
 \end{lemma}
 
\begin{proof}
Fix $T>0$, $F\in C^\infty(\sqrt{K}\T \times \bb T^{d-1})$, and let
$v =T^K_s F\in C^{1,2}((0,T]\times \sqrt{K}\T \times \bb T^{d-1})$.
The function $v$ satisfies
\begin{equation}
\label{2-85}
\partial_t v = K\,\partial^2_{\vartheta} v \,+\, \Delta_\theta v
\,-\, K\, V''(\rho^K)\, v
\end{equation}
on $(0,T]\times \sqrt{K}\T \times \bb T^{d-1}$ with initial condition
$v(0, \vartheta, \theta) = F(\vartheta, \theta)$. By the parabolic
maximum principle, \cite[Theorem 2.10]{Lieberman} where
$\Omega = (0,T]\times \sqrt{K}\T \times \bb T^{d-1}$,
$I(\Omega) = (0,T)$,
$\mc P \Omega = \{0\} \times \sqrt{K}\T \times \bb T^{d-1}$
$k=K\, \mf c_5$, $f\equiv 0$, we have
$\sup_{0<s<T} \|v(s) \|_\infty \leq e^{(\mf c_5 K +1)T}\|F\|_\infty$.

Since $\rho^K$ is smooth, by the proof of Theorem 4.9 \cite{Lieberman}
(cf. exercise 4.5 \cite{Lieberman}),
$v\in C^{1, r}((0,T]\times \sqrt{K}\T \times \bb T^{d-1})$ for all
$r\geq 2$.  Let $v^{(1,1)} =\partial_{\vartheta}v$,
$v^{(1,j)} =\partial_{\theta_j}v$, $2\le j\le d$. Clearly, $v^{(1,1)}$
solves the equations
\begin{gather*}
\partial_t  v^{(1,1)}  \,=\,  K\,\partial^2_{\vartheta} v^{(1,1)}
\,+\, \Delta_\theta v^{(1,1)}  \,-\, K\, V''(\rho^K)\, v^{(1,1)}
\,-\, K\, V'''(\rho^K)\, \partial_{\vartheta}  \rho^K \, v\;,
\end{gather*}
while $v^{(1,j)}$ solve the equation \eqref{2-85}.  Hence, again by
the parabolic maximum principle Theorem 2.10 \cite{Lieberman},
a bound for $v^{(1,j)}$, $2\le j\le d$, similar to the one obtained in
the first part of the proof holds, and 
\begin{gather*}
\max_{1\leq i\leq d} \sup_{0<s<T}  \|v^{(1,j)} (s) \|_\infty \,\leq\,
e^{(\mf c_5 K  +1)T} \,\Big\{\, \|\nabla F\|_\infty \,+\, \mf c_0\,  K\,
\| \partial_{\vartheta}  \rho_K\|_\infty \, \|v\|_\infty \,\Big\}
\end{gather*}
for some finite constant $\mf c_0$ depending only on the potential
$V$. By \eqref{2-84},
$\sup_K\| \partial_{\vartheta} \rho_K\|_\infty < \infty$. Thus, by the
first part of the proof, the right-hand side in the previous displayed
equation is bounded by 
$\mf c_0\, C(F)\, K\, e^{2 (\mf c_5 K+1)T}$ for some finite constant $\mf c_0$
which depends only on $V$ (and $\rho_K$).

By the same scheme and maximum principle, noting that coefficients
involve $\partial^2_{\vartheta} \rho^K=V'(\rho^K)$ and
$\partial^3_{\vartheta} \rho^K = V''(\rho^K) \, \partial_{\vartheta}
\rho^K$, both uniformly bounded in $K$ by \eqref{2-84}, we recover the
bounds in the lemma statement for $i=2,3$.
\end{proof}

Recall that $\color{blue} (P^K_t : t\ge 0)$ stands for the semigroup
associated to the generator $\ms A_K$ introduced in \eqref{2-49}.
Recall also
$\bs e = \partial_{\vartheta} \phi / \| \partial_{\vartheta} \phi
\|_2$, $\bs e_K$ equals $\bs e$ on $[-\sqrt{K}/4, \sqrt{K}/4]$ and
smoothly interpolates on the rest of the torus $\sqrt{K}\bb T$ so that
$|\partial^j_\vartheta \bs e_K| \leq C|\partial^j_\vartheta \bs e|$ on
$\sqrt{K}\bb T\cong[-\sqrt{K}/2, \sqrt{K}/2)$ for $j=0,1,2$.

\begin{lemma} 
\label{2-l19}
For every $F\in C^\infty_c(\bb R)$,  $T>0$
$$
\lim_{K\uparrow\infty}\int_0^T \|\, \partial_\vartheta  P^K_t F
- \langle F, \bs e\rangle \,
\partial_\vartheta \bs e_K\|^2_{L^2(\sqrt{K}\T)} \; dt \,=\, 0\;.
$$
\end{lemma}

\begin{proof}
Integrating by parts,
\begin{align*}
& \|\, \partial_\vartheta  P^K_t F
- \langle F, \bs e\rangle \,
\partial_\vartheta \bs e_K\|^2_{L^2(\sqrt{K}\T)}
\\
&\quad  \;=\; -\, 
\int_{\sqrt{K}\bb T}
\big\{\, P^K_t F - \langle F, \bs e\rangle \, \bs e_K\,\big\}\,
\partial^2_\vartheta \, \big\{\,  P^K_t F
- \langle F, \bs e\rangle \,
\bs e_K \,\big\}\; d\vartheta\;.
\end{align*}
By Schwarz inequality, this expression is less than or equal to
\begin{equation*}
\|\, P^K_t F - \langle F, \bs e\rangle \, \bs e_K\|_{L^2(\sqrt{K}\T)}
\, \{\, \|\, \partial^2_\vartheta  P^K_t F \|_{L^2(\sqrt{K}\T)}
\,+\, \big\|\, \langle F, \bs e\rangle \,
\partial^2_\vartheta \bs e_K\|_{L^2(\sqrt{K}\T)} \, \big\}\;.
\end{equation*}
By \cite[Lemma B.7]{F1}, \eqref{CP properties} noting that
$\partial^2_\vartheta \bs e = V''(\phi)\partial_\vartheta\phi/\|\bs
e\|_{L^2(\bb R)}$, the expression inside braces is bounded by a
constant uniformly in $0\le t\le T$. By Lemma B.7 and Proposition B.4
in \cite{F1}, the time integral of the first term converges to $0$ as
$N\to\infty$.
\end{proof}

We complete this section with an estimate on the density profile
$\bs u^N\colon \bb T^d_N \to (0,1)$ introduced in \eqref{2-52},
\eqref{2-53a}.  Recall from \eqref{2-87} that $\Delta^{(1)}_N$ stands
for the discrete Laplacian on the first coordinate.

\begin{lemma}
\label{2-l18}
There exists a finite constant $C_0$, depending only on $V$, such that
\begin{equation*}
\max_{x\in \bb T^d_N} \big|\, (\Delta^{(1)}_N \bs u_N)(x) + K\,
E_{\nu^N}[\tau_x h_0] \,\big|\;\le\; \frac{C_0 \, K^2}{N^2}
\end{equation*}
for all $N\ge 1$. 
\end{lemma}

\begin{proof}
The proof relies on simple estimates.  A fourth order Taylor
expansions, noting \eqref{2-56}, yields that
\begin{equation}
\label{2-23bis}
\max_{x\in \bb T^d_N}  \big|\, (\Delta_N u^N)(x)
\,-\, K \, V'(u^N(x))\,\big| \,\le \,
\Vert \partial^4_\vartheta \rho^K \Vert_\infty \, K^2/N^2 \;, \quad x\in
\bb T_N
\end{equation}
for all $N\ge 1$. Note that
$\Vert \partial^4_\vartheta \rho^K \Vert_\infty$ is bounded by
\eqref{2-84}.

Recall from \eqref{01}, \eqref{2-63} the definition of the cylinder
function $h_0$. We claim that there exists a finite constant $C_0$,
depending only on $\gamma$ (the parameter which appears in the
definition of the jump rates $c_0$), such that
\begin{equation}
\label{2-37tris}
\max_{x\in \bb T^d_N} \big|\, E_{\nu^N}[\tau_x h_0]
\,+\,  V'(\bs u^N(x))\, \big| \;\le\;
\frac{C_0\, K}{N^2} \, \Big\{\, \Vert \partial_\vartheta\rho^K \Vert^2_\infty
\,+\, \Vert \partial^2_\vartheta \rho^K \Vert_\infty\,
\,\Big\}\, 
\end{equation}
for all $N\ge 1$.

To prove this claim, recall that $\bs u^N$ is translation invariant in
the directions $e_k$, $2\le k\le d$, so that
\begin{equation*}
E_{\nu^N}[\tau_x h_0] \;=\; \gamma\,[ \beta(x_1-1)
\,+\,  \beta(x_1+1) \,]
\,-\, \beta(x_1) \,-\, \gamma^2 \beta(x_1-1)\, \beta(x_1)\,
\beta(x_1+1) \;, 
\end{equation*}
where $\beta (y) = 2 u^N(y) - 1$, $y\in\bb T_N$, and $u^N$ is given by
\eqref{2-52}. By definition of $V'(\cdot)$, and since, by
\eqref{2-38b}, $u^N$ is uniformly bounded by $1$, the left-hand side
of \eqref{2-37tris} is bounded by
\begin{equation*}
\gamma^2 \, \max_{z\in \bb T_N} \big|\, \beta(z-1)
\beta(z+1) - \beta(z)^2 \, \big| \,+\,
\gamma \, \max_{z\in \bb T_N} \big|\, \beta(z-1) +
\beta(z+1) - 2 \beta(z) \, \big| \;.
\end{equation*}
Write $\beta(z-1) \beta(z+1) - \beta(z)^2 $ as
$[\beta(z-1) -\beta(z)]\, [\beta(z+1) - \beta(z)] + \beta(z) \,\{\,
\beta(z-1) + \beta(z+1) - 2 \beta(z) \}$. By \eqref{2-38b}, $u^N$ is
uniformly bounded by $1$. Hence, the previous expression is bounded by
\begin{equation*}
4\, \gamma^2 \, \max_{z\in \bb T_N} \big[\, u^N(z+1)
- u^N (z)\,\big ]^2 \,+\,  2\,
(\gamma^2 + \gamma) \, N^{-2}\,
\max_{z\in \bb T_N} \big|\, (\Delta_N u^N) (z) \,\big | \;.
\end{equation*}
By definition \eqref{2-52} of $u^N$, the first term is bounded by
$4\, \gamma^2 \, \Vert \partial_\vartheta\rho^K \Vert^2_\infty\, K
N^{-2}$, and the second one by
$2(\gamma^2 + \gamma) \, \Vert \partial^2_\vartheta \rho^K \Vert_\infty
\, K N^{-2}$. This completes the proof of \eqref{2-37tris}.

Adding the estimates  \eqref{2-23bis} and \eqref{2-37tris} completes
the proof of the lemma.
\end{proof}

\end{document}